\documentclass[a4paper]{amsart}
\usepackage{amssymb,latexsym}

\usepackage[active]{srcltx}
\usepackage{amsmath}
\usepackage{amsfonts}

\usepackage{amsthm}
\usepackage{amscd}
\usepackage{mathrsfs}
\usepackage{mathtools}
\usepackage{fancyhdr}
\usepackage{graphicx}
\usepackage{calc}
\usepackage{url}

\usepackage{times}
\usepackage{verbatim}
\usepackage[cmtip,arrow,matrix,curve,tips,frame]{xy}
\usepackage{hyperref}
\xyoption{matrix}

\usepackage{ifthen}
\usepackage{epsfig}
\usepackage{color}
\usepackage{amsxtra}
\usepackage{amstext}
\usepackage{amssymb}
\usepackage{stmaryrd}
\usepackage{mathrsfs}
\usepackage{pifont}
\usepackage{charter}
\usepackage{avant}
\usepackage{courier}
\usepackage[T1]{fontenc}
\usepackage{textcomp}
\usepackage{bbm}

\usepackage{tikz} 
\usetikzlibrary{matrix,arrows,calc,decorations.pathreplacing,fit}
\usepackage{tikz-cd}

\usepackage{soul}
\usepackage[most]{tcolorbox}

\numberwithin{equation}{subsection}

\theoremstyle{plain}
\newtheorem{thm}[subsection]{Theorem}
\newtheorem*{thm*}{Theorem}

\newtheorem*{exthm*}{Expected Theorem}

\newtheorem{lem}[subsection]{Lemma}

\newtheorem*{lem*}{Lemma}

\newtheorem{prop}[subsection]{Proposition}

\newtheorem*{prop*}{Proposition}

\newtheorem{cor}[subsection]{Corollary}

\newtheorem*{cor*}{Corollary}

\newtheorem*{claim*}{Claim}

\newtheorem{conj}[subsection]{Conjecture}

\newtheorem*{conj*}{Conjecture}

\theoremstyle{definition}
\newtheorem{defn}[subsection]{Definition}
\newtheorem*{defn*}{Definition}

\newtheorem{notation}[subsection]{Notation and Conventions}

\newtheorem*{exa*}{Example}

\theoremstyle{remark}
\newtheorem{rmk}[subsection]{Remark}

\newtheorem*{rmk*}{Remark}

\newtheorem*{ass*}{Assumption}

\numberwithin{figure}{subsection}
\numberwithin{table}{subsection}

\newenvironment{bulletlist}
   {
      \begin{list}
         {$\bullet$}
         {
            \setlength{\itemsep}{.5ex}
            \setlength{\parsep}{0ex}
            \setlength{\parskip}{0ex}
            \setlength{\topsep}{.5ex}
         }
   }
   {
      \end{list}
   }

\newcounter{listnum}

\newcounter{asslistcounter}
\newenvironment{assertionlist}{
 \begin{list}
  {\upshape (\alph{asslistcounter})}
  {\setlength{\leftmargin}{18pt}
   \setlength{\rightmargin}{0pt}
   \setlength{\itemindent}{0pt}
   \setlength{\labelsep}{5pt}
   \setlength{\labelwidth}{13pt}
   \setlength{\listparindent}{\parindent}
   \setlength{\parsep}{0pt}
   \setlength{\itemsep}{0pt}
   \setlength{\topsep}{-.5\parskip}
   \usecounter{asslistcounter}}}
  {\end{list}}

\newcounter{subenvcounter}
\newenvironment{subenv}{%
 \begin{list}
  {\em (\arabic{subenvcounter})}
  {\setlength{\leftmargin}{20pt}
   \setlength{\rightmargin}{0pt}
   \setlength{\itemindent}{0pt}
   \setlength{\labelsep}{5pt}
   \setlength{\labelwidth}{13pt}
   \setlength{\listparindent}{\parindent}
   \setlength{\parsep}{0pt}
   \setlength{\itemsep}{0pt}
   \setlength{\topsep}{-\parskip}
   \usecounter{subenvcounter}}}
  {\end{list}}

\DeclareMathOperator{\Aut}{Aut}

\DeclareMathOperator{\cha}{char}

\DeclareMathOperator{\Div}{Div}

\DeclareMathOperator{\Fr}{Fr}

\DeclareMathOperator{\Frob}{Frob}
\DeclareMathOperator{\id}{id}
\DeclareMathOperator{\image}{im}
\DeclareMathOperator{\Image}{Im}
\DeclareMathOperator{\Int}{Int}
\DeclareMathOperator{\inv}{inv}
\DeclareMathOperator{\Gal}{Gal}
\DeclareMathOperator{\GL}{GL}

\DeclareMathOperator{\Hom}{Hom}

\DeclareMathOperator{\Isom}{Isom}
\DeclareMathOperator{\Lie}{Lie}
\DeclareMathOperator{\Ig}{Ig}

\DeclareMathOperator{\Norm}{Norm}

\DeclareMathOperator{\pr}{pr}

\newcommand{\Sh}{\mathsf{Sh}}
\DeclareMathOperator{\SL}{SL}

\DeclareMathOperator{\Spec}{Spec}
\DeclareMathOperator{\Spf}{Spf}

\DeclareMathOperator{\supp}{supp}

\DeclareMathOperator{\Res}{Res}

\DeclareMathOperator{\rank}{rk}
\DeclareMathOperator{\tr}{tr}

\DeclareMathOperator{\vol}{vol}

\def \AA {\mathbb{A}}

\def \CC {\mathbb{C}}
\def \DD {\mathbb{D}}

\def \FF {\mathbb{F}}
\def \GG {\mathbb{G}}

\def \JJ {\mathbb{J}}

\def \NN {\mathbb{N}}
\def \OO {\mathbb{O}}

\def \QQ {\mathbb{Q}}

\def \XX {\mathbb{X}}

\def \ZZ {\mathbb{Z}}

\def \Ccal {\mathcal{C}}

\def \Ecal {\mathcal{E}}

\def \Gcal {\mathcal{G}}
\def \Hcal {\mathcal{H}}

\def \Lcal {\mathcal{L}}

\def \Tcal {\mathcal{T}}

\def \Vcal {\mathcal{V}}
\def \Wcal {\mathcal{W}}

\def \Bfr {\mathfrak{B}}

\def \Gfr {\mathfrak{G}}

\def \Kfr {\mathfrak{K}}

\def \Ascr {\mathscr{A}}

\def \Dscr {\mathscr{D}}
\def \Escr {\mathscr{E}}
\def \Fscr {\mathscr{F}}
\def \Gscr {\mathscr{G}}
\def \Hscr {\mathscr{H}}

\def \Lscr {\mathscr{L}}

\def \Oscr {\mathscr{O}}

\def \Sscr {\mathscr{S}}
\def \Tscr {\mathscr{T}}

\def \Vscr {\mathscr{V}}
\def \Wscr {\mathscr{W}}
\def \Xscr {\mathscr{X}}

\def \Zscr {\mathscr{Z}}

\def \Dhat {\hat{D}}

\def \Rhat {\hat{R}}

\def \Fbar {\bar{F}}

\def \Tbar {\bar{T}}

\def \hbar {\bar{h}}

\def \kbar {\bar{k}}

\def \sbar {\bar{s}}

\def \ftilde {\tilde{f}}

\def \ztilde {\tilde{z}}

\def \Gsf {\mathbf{G}}

\def \Ibf {\mathbf{I}}

\def \bbf {\mathbf{b}}

\def \gbf {\mathbf{g}}

\def \xbf {\mathbf{x}}

\def \Bsf  {\mathsf{B}}
\def \Csf  {\mathsf{C}}
\def \Dsf  {\mathsf{D}}
\def \Esf  {\mathsf{E}}
\def \Fsf  {\mathsf{F}}
\def \Gsf  {\mathsf{G}}
\def \Hsf  {\mathsf{H}}

\def \Jsf  {\mathsf{J}}
\def \Ksf  {\mathsf{K}}

\def \Msf  {\mathsf{M}}

\def \Tsf  {\mathsf{T}}

\def \Xsf  {\mathsf{X}}

\def \Zsf  {\mathsf{Z}}

\newcommand{\blambda}{\boldsymbol{\lambda}}
\newcommand{\bpi}{\boldsymbol{\pi}}

\def \BT    {{Barsotti-Tate group}}

\def \mono  {\hookrightarrow}

\def \epi   {\twoheadrightarrow}

\def \isom  {\stackrel{\sim}{\rightarrow}}

\def \bij   {\stackrel{1:1}{\rightarrow}}

\newcommand{\pot}[1]{ [\hspace{-0,17em}[ {#1} ]\hspace{-0,17em}] }
\newcommand{\rpot}[1]{ (\hspace{-0,23em}( {#1} )\hspace{-0,23em}) }

\newcommand{\restr}[2]{{#1}\raise-.5ex\hbox{\ensuremath|}_{#2}}

\newcommand{\der}{\mathrm{der}}

\def \Grass {{\rm Gr}}
\def \perf {{\rm perf}}

\newcommand{\loc}[2]{{#1}[{#2}^\infty]} 

\newcommand{\cl}[1]{\mkern 1.5mu\overline{\mkern-1.5mu#1\mkern-1.5mu}\mkern 1.5mu}
\newcommand{\scl}[1]{{#1}^s} 

\def \ad {\mathrm{ad}}

\newcommand{\red}{\mathrm{red}}

\newcommand{\coh}[1]{\mathrm{H}^{#1}}

\newcommand{\Rrm}{\mathrm{R}}

\newcommand{\iv}{^{-1}}

\newcommand{\riso}{\xrightarrow{\sim}}

\newcommand{\el}{\mathrm{el}}
\newcommand{\ael}{\mathrm{a,el}}

\newcommand{\reg}{\mathrm{reg}}
\newcommand{\rss}{\mathrm{rss}}
\DeclareMathOperator{\EP}{EP}
\DeclareMathOperator{\FP}{FP}
\DeclareMathOperator{\KS}{KI}

\DeclareMathOperator{\Bun}{Bun}
\DeclareMathOperator{\Hecke}{Hecke}
\DeclareMathOperator{\Gr}{Gr}
\usepackage[british]{babel}
\usepackage{xcolor}
\hypersetup{%
  colorlinks=true,
}
\author[P.~Hamacher, W.~Kim]{Paul Hamacher and Wansu Kim}

\def \Mund {{\underline{M}}}
\def \Nund {{\underline{N}}}

\def \k {{\bar{\mathbb{F}}}_p}

\def \Jsf {{\mathsf{J}}}

\def \gsf {{\mathsf{g}}}
\def \B   {{\mathrm{B}}}
\def \unif {{\varpi}}

\def \ggamma  {{\boldsymbol\gamma}}
\def \ddelta  {{\boldsymbol\delta}}
\def \tor {\mathrm{tor}}
\def \ur {\mathrm{ur}} 

\newcommand{\ab}{\mathrm{ab}}
\DeclareMathOperator{\coker}{coker}
\DeclareMathOperator{\ord}{ord}

\usepackage{bbold}

\DeclareMathOperator{\Acc}{Acc}

\DeclareMathOperator{\Fix}{Fix}
\DeclareMathOperator{\res}{res}
\DeclareMathOperator{\Groth}{Groth}
\DeclareMathOperator{\Jac}{Jac}
\DeclareMathOperator{\LJ}{LJ}
\DeclareMathOperator{\JL}{JL}
\DeclareMathOperator{\Red}{Red}

\newcommand{\bnu}{\boldsymbol\nu}

\usepackage{manfnt}

\address{Paul Hamacher\\
Technische Universit\"at M\"unchen\\
Zentrum Mathematik - M 11\
Boltzmannstra{\ss}e 3\\
85748 Garching\\
Deutschland\\ 
\newline
Wansu Kim\\%
Department of Mathematical Sciences\\%
KAIST\\%
291 Daehak-ro, Yuseong-gu\\%
Daejeon, 34141\\%
South Korea}
\email{paul.hamacher@tum.de, wansu.math@kaist.ac.kr}

\title{Point Counting on Igusa Varieties for Function Fields}

\def \f {{\FF_q}}
\def \k {{\cl\FF_q}}

\def \Fbreve {{\breve{F}}}

\newcommand{\revise}[1]{#1}
\newcommand{\reviselong}[1]{#1}

\newif\ifrevision
\revisionfalse   

\ifrevision
  \newtcolorbox{revisebox}{%
    breakable,
    enhanced,
    parbox=false,           
    boxrule=0pt,            
    colback=green!15,       
    colframe=green!15,      
    sharp corners,          
    left=0pt, right=0pt,
    top=0pt, bottom=0pt,
    before skip=0pt, after skip=0pt,
    before upper=\indent,   
  }
\else
  \newenvironment{revisebox}{}{}
\fi

\begin{document}

\begin{abstract}
 Igusa varieties over the special fibre of Shimura varieties have demonstrated many applications to the Langlands program via Mantovan's formula and Shin's point counting method. In this paper we study Igusa varieties over the moduli stack of global $\Gscr$-shtukas and (under certain conditions) calculate the Hecke action on its cohomology. As part of their construction we prove novel results about local $G$-shtukas in both equal and unequal characteristic and also discuss application of these results to Barsotti-Tate groups and Shimura varieties.
\end{abstract}

\maketitle

\tableofcontents

\section{Introduction}\label{sect-intro}

Over number fields, recent decades have witnessed a lot of progress in associating Galois representations to automorphic representations, and most of such constructions use \emph{Shimura varieties} as intermediary. Over global function fields, one can work with the \emph{moduli stacks of shtukas} in place of Shimura varieties to realise the automorphic-to-Galois direction of Langlands correspondence.

Over number fields, one can relate the cohomology of certain Shimura varieties to two ``simpler'' spaces -- the Rapoport-Zink spaces and the Igusa varieties -- which is one of the main ingredients for the construction  of local and global Langlands reciprocity (\emph{cf.} \cite{HarrisTaylor:TheBook},  \cite{ShinSW:GlobalLanglands}  and \cite{ShinSW:CohRZ}). Furthermore, Igusa varieties play an important role in the study of perfectoid Shimura varieties, as the generic fibres are precisely the fibres of the Hodge--Tate period morphism (\cite{CaraianiScholze:ShVar}).
Over function fields, Hartl and Viehmann \cite{HartlViehmann:Newton} first introduced and studied the analogue of Rapoport-Zink spaces for local $G$-shtukas when $G$ is a reductive group defined over a finite field (i.e., a constant reductive group), which was later extended to  more general groups by Arasteh~Rad and Hartl \cite {ArastehRad-Hartl:LocGlShtuka}. Built upon these works, the foliation theory for the moduli stacks of $G$-shtukas has been established by Neupert for constant reductive groups (\emph{cf.} \cite{neupert:thesis}), and simultaneously by A.~Wei\ss{} in the ``parahoric'' case (\emph{cf.} \cite{Weiss:thesis}). 

In this paper, we construct the function-field analogue of Igusa varieties in full generality (\emph{cf.} Thm.~\ref{th-main-ig-var}) and obtain a point-counting result analogous to Shin's result (\emph{cf.} Prop.~\ref{prop-main-pt-counting}). From this we deduce a formula for their cohomology in the case of division algebras (\emph{cf.} Thm.~\ref{th-main-div-alg}), and conjecture a formula for general Igusa varieties (\emph{cf.} Conj.~\ref{conj-trace-formula}). While the situation over global and local function fields is different from its classical counterpart in that we have many other approaches to the Langlands reciprocity, we cautiously expect that the study of function-field Igusa varieties constructed in this paper points to some interesting direction of research (beyond the case of division algebras) via analogy with their classical counterpart. More specifically, recent results on nearby cycles cohomology of Satake sheaves \cite{Xue:Smoothness,Salmon:UnipNearby-Shtukas,EteveXue:Nearby} suggest a cohomological implication of the foliation theory, at least for parahoric level structures. It would be interesting to explore the relationship between function-field Igusa varieties and various results on moduli stacks of global and local $G$-shtukas \cite{LafforgueZhu:elliptic,ArastehRad-Hartl:LR,ArastehRad:LocMod-RZ}, as well as the test function conjecture \cite{HainesRicharz:TestFtnParahoric}.
%

Upon completion of the first version of the paper, we learned that Jack Sempliner independently constructed Igusa varieties and obtained the almost product structure in case where $\Gsf$ is tamely ramified at $x$.

\subsection*{Results on local $G$-shtukas} 

In the first part of the paper, we prove some general results about local $G$-shtukas, which hold in greater generality than the rest of the paper. In the following let $F$ denote a local field (possibly of characteristic $0$) with residue field $\f$ and let $G$ be a reductive group over $F$. A $G$-isoshtuka over an $\f$-scheme $S$ is the local analogue of the generic point of a $\Gscr$-shtuka and if $S$ is perfect corresponds to an $F$-isocrystal with $G$-structure over $S$. The central part of our results on local $G$-shtukas is the following constancy result.

\begin{thm}[\emph{cf.}~Thm.~\ref{thm-constancy-of-G-isoshtukas}]
 Let $\Hcal$ be a $G$-isoshtuka over a perfect normal $\FF_q$-stack $S$ such that it is isomorphic to $(LG,b\sigma)$ for some $b \in G(\Fbreve)$ over every geometric point of $S$. Then there exists a profinite \'etale cover $S' \to S$ such that $\Hcal_{S'} \cong (LG,b\sigma)_{S'}$.
\end{thm}

  From this statement we derive short and direct proofs of Tate's isogeny theorem for local $G$-isoshtukas with constant Newton point and the purity theorem for local $G$-isoshtukas, without further conditions on $G$ or the base scheme.

\begin{prop}[Tate's isogeny theorem; \emph{cf.}~Prop.~\ref{prop-tates-thm}]
 Let $\underline\Hcal_1,\underline\Hcal_2$ be two $G$-isoshtukas with constant Newton point over an integral normal $\kappa_F$-scheme $S$ with generic point $\eta$. Then the canonical restriction map
 \[
  \Isom(\underline\Hcal_1,\underline\Hcal_2) \to \Isom(\underline\Hcal_{1,\eta},\underline\Hcal_{2,\eta})
 \]
 is an isomorphism.
\end{prop}
 Previous results on Tate's isogeny theorem have some additional restrictions on the base scheme or $G$; see \cite[Rmk.~4.2.7]{CaraianiScholze:ShVar} for Barsotti-Tate groups, and \cite[Prop.~2.7.6]{neupert:thesis} for equal characteristic local $G$-isoshtukas when the base is noetherian and $G$ is a constant reductive group.

\begin{prop}[Purity theorem; \emph{cf.}~Prop.~\ref{prop-purity}]
 Let $\underline\Hcal$ be an isoshtuka over a $\kappa_F$-scheme $S$. Denote by $S^{[b]_i} \subset S$ the locally closed subscheme determined by the geometric points where the isomorphism class of $\underline\Hcal$ is less or equal to $[b]$ and whose Newton point contains the break point of $[b]$ at $i$. (See \S\ref{ssect-applications} for details.) Then the embedding $S^{[b]_i} \mono S$ is an affine morphism of schemes.
\end{prop}
 The original purity theorem for $F$-crystals of de Jong-Oort \cite[Thm.~4.1]{deJongOort:Purity} has been strengthened by subsequent work of Vasiu \cite[Main~Theorem~B]{Vasiu:CrystallineBoundedness} and Yang \cite{Yang:Purity}. In the equal characteristic setting, the purity result was shown by Viehmann \cite{Viehmann:NewtonStratLoopGp} for split $G$ and integral Noetherian $S$. 

Moreover, we give another proof of the result announced by of Shen and Zhang \cite[(5.4.5)]{ShenZhang:ShVarStrat} that ``canonical'' central leaves in Shimura varieties of abelian type are closed inside their respective Newton stratum and that their connected components are connected components of adjoint central leaves (\emph{cf}.~Prop.~\ref{prop-Shen-Zhang}).

\subsection*{General results on Igusa varieties for function fields} We fix a global function field $\Fsf$ associated to a smooth projective geometrically connected curve $C/\FF_q$, and set $\breve\Fsf \coloneqq \Fsf \otimes_{\FF_q} \cl\FF_q$. For $x\in|C|$ (with $|C|$ denoting the set of closed points) let $\Fsf_x$ be the completion of $F$ at $x$ with valuation ring $O_x$, and $\breve \Fsf_x$ the completion of the maximal unramified extension of $\Fsf_x$. Let $\Gsf$ be a connected reductive group over $\Fsf$, and $\Gscr$ a smooth affine model over $C$ with geometrically connected fibres. 

We fix a tuple of geometric points $\overline\xbf = (\overline x_i)\in C(\cl\FF_q)^n$, and denote by $\Xscr_{\Gscr,\Ibf,\overline\xbf}$ the ind-Deligne--Mumford moduli stack of unbounded global $\Gscr$-shtukas with legs $\overline\xbf$. Let $\xbf=(x_i)\in|C|^n$ denote the tuple of closed points underlying $\overline\xbf$, and we assume that $x_i$'s are pairwise distinct. (See \S\ref{ssect-infinite-level-igusa} for the precise prerequisite, which is slightly more general.) When $\Gscr$ is a constant reductive group or more generally a parahoric group scheme, Igusa varieties on the moduli stack of $\Gscr$-shtukas have been defined in the theses of Neupert (\cite{neupert:thesis}) and Wei\ss{} (\cite{Weiss:thesis}), respectively. We give a construction for general $\Gscr$ as follows.

\begin{thm}[Thm.~\ref{th-infinite-level-igusa-var}, Def.~\ref{def-infinite-level-igusa-variety}, \S\ref{ssect-global-J-orbits}]\label{th-main-ig-var}
  We fix $\bbf = (b_i) \in \prod_{i=1}^n \Gsf(\breve\Fsf_{x_i})$.
\begin{subenv}
\item\label{th-main-ig-var:cen-leaf} There exists a Deligne--Mumford stack $C^\bbf_{\Gscr,\xbf}$ locally of finite type and smooth over $\cl\FF_q$, which admits a locally closed immersion into $\Xscr_{\Gscr,\Ibf,\overline\xbf}$ whose image is characterised as follows: a global $\Gscr$-shtuka $\underline\Vscr_\bullet\in \Xscr_{\Gscr,\Ibf,\overline\xbf}(\cl\kappa)$ over an algebraically closed extension $\cl\kappa$ of $\cl\FF_q$ is in the image of $C^\bbf_{\Gscr,\xbf}$ if and only if the localisation $\underline\Vscr_0[x_i^\infty]$ at each $x_i\in\xbf$ \eqref{eq-loc-j} is isomorphic to the one defined by $b_i$ in the sense of Rmk.~\ref{rmk-gl-loc} (especially, \eqref{eq-gl-loc-b}).
\item\label{th-main-ig-var:igusa} The functor simultaneously trivialising the localisation of the universal $\Gscr$-shtuka over $C^\bbf_{\Gscr,\xbf}$ at each place $x \in |C|$ is represented by a scheme $\Ig^{\bbf}_{\Gsf,\xbf}$, which only depends on $\Gsf$ and $\xbf$ (but not on $\Gscr$ and $\overline\xbf$) up to isomorphism. Furthermore, the perfection of $\Ig^{\bbf}_{\Gsf,\xbf}$ is a pro-\'etale Galois cover of $C^{\bbf,\perf}_{\Gscr,\xbf}$ and is equipped with a natural action of $\JJ_\bbf\coloneqq \Gsf(\AA^{\xbf})\times \prod_{i=1}^nJ_{b_i}(\Fsf_{x_i})$. \revise{(See \S\ref{ssect-aut-group} for the definition~of~$J_{b_i}(\Fsf_{x_i})$.)}
\item Its closed geometric points can be described in group-theoretical terms as follows: 
\[ \Ig_{\Gsf,\xbf}^{\bbf}(\cl\FF_p) = \bigsqcup_{[b]_\sigma \in \B(\Fsf,\Gsf)_\bbf} \iota_{\ztilde}(\Jsf_b(\Fsf)) \backslash \JJ_\bbf.\]
Here $\B(\Fsf,\Gsf)_\bbf$ denotes the set of $\sigma$-conjugacy classes in $\Gsf(\breve\Fsf)$ localising to $[b_i]_\sigma$ at $x_i$ and to $[1]_\sigma$ at places away from $\xbf$, and the embedding $\iota_{\ztilde}\colon\Jsf_b(\Fsf) \mono \JJ_\bbf$ is defined in \S\ref{ssect-global-J-orbits}. 
\end{subenv}
\end{thm}
 
Moreover, we prove the following analogue of part (\ref{th-main-ig-var:igusa}); namely, the moduli stacks of global $\Gscr$-shtukas with infinite level structure at a place $v \in |C|$ are independent (up to canonical isomorphism) of the choice of $\Gscr_{O_v}$. As a consequence, we can express certain (finite) level structures at $v$ by modifying $\Gscr$ at $v$ and thereby extend the integral model to $v$ (\emph{cf}.~Prop.~\ref{prop-infinite-level}, Rmk.~\ref{rmk-infinite-level}).

Just as in the classical case, we can group-theoretically describe the set of fixed points by a Hecke correspondence on Igusa varieties after a suitable iterated Fro\-be\-nius-twist. Unfortunately, Igusa varieties may \emph{not} be quasi-compact in general, so we should count fixed points in some quasi-compact open subscheme (classically de\-fi\-ned via Harder--Narasimhan stratification); see Prop.~\ref{prop-preliminary-point-counting} for more details. But in the favourable case when the Igusa variety is quasi-compact, one can use the Lefschetz trace formula to describe the Hecke action on the cohomology of the Igusa variety in terms of the fixed points. The following is an analogue of the \emph{first form of the point-counting formula} for Igusa varieties of PEL and Hodge type (\emph{cf.} \cite[Lem.~7.4]{ShinSW:IgusaPointCounting}, \cite[Prop.~4.2.4]{Mack-Crane:2022wo}) in the quasi-compact case.

\begin{prop}[{\emph{cf.}~Prop.~\ref{prop-preliminary-point-counting}, \ref{prop-preliminary-trace-formula}}]\label{prop-main-pt-counting}
 Let $\Xi \subset \Zsf_\Gsf(\AA)$ denote a discrete subgroup in the center $\Zsf_\Gsf$ of $\Gsf$ with $\Xi\cap\Zsf_\Gsf(\Fsf) = \{1\}$ such that $\Zsf_\Gsf(\Fsf)\backslash\Zsf_\Gsf(\AA)/\Xi$ is compact. We assume that $\Ig^\bbf_{\Gsf,\xbf,\Xi} \coloneqq \Ig^\bbf_{\Gsf,\xbf}\sslash \Xi$ is quasi-compact. We fix a prime $\ell$ different from the characteristic of $\FF_q$. Then for any $\varphi \in C^\infty_c(\JJ_\bbf/\Xi;\cl\QQ_\ell)$ the trace of the $s$-th iterated Frobenius twist $\varphi^{(s)}$ for $s \gg 0$ on the cohomology of $\Ig^\bbf_{\Gsf,\xbf,\Xi}$ is given by
 \[
  \sum_i(-1)^i\tr\big(\varphi\mid \coh i_c(\Ig^{\bbf}_{\Gsf,\xbf,\Xi},\cl\QQ_\ell)\big) = \sum \vol\left (\iota_{\tilde z}(\Zsf_{\Jsf_{b}(\Fsf)}(a)) \backslash \Zsf_{\JJ_\bbf}(\iota_{\tilde z}(a))/\Xi\right )\cdot O^{\JJ_{\bbf}}_{\iota_{\tilde z}(a)}(\varphi^{(s)}),\]
 where the sum runs over all $[b]_\sigma \in \B(\Fsf,\Gsf)_\bbf$ and conjugacy classes $[a] \subset \Jsf_b(\Fsf)$. Here, $O^{\JJ_{\bbf}}_{\iota_{\tilde z}(a)}(\varphi^{(s)})$ denotes the orbital integral $\int _{\Zsf_{\JJ_\bbf}(\iota_{\tilde z}(a))\backslash\JJ_\bbf}  \varphi^{(s)} (y^{-1}\iota_{\tilde z}(a) y) d\bar y$.
\end{prop}

Even when the Igusa variety is \emph{not} quasi-compact, we expect that there should be a \emph{simple trace formula} for Igusa varieties, involving only the ``elliptic part'' of local terms; \emph{cf.} Conj.~\ref{conj-trace-formula}. In \S\ref{sect-conjectures}, we explain the motivation behind this conjecture.

Classically, the preliminary point-counting formula needs to be rewritten so that the summation is over some analogue of Kottwitz triples (which we call \emph{Kottwitz--Igusa triples}; \emph{cf.} Def.~\ref{df-Kottwitz-Igusa}). Note however that in the function field case there may be a priori \emph{no} place forcing the Kottwitz--Igusa triple to be elliptic while in the classical case Kottwitz--Igusa triples are forced to be elliptic at the archimedean place. Nonetheless, we still expect that for suitable test functions only the elliptic Kottwitz--Igusa triples contribute to the trace formula (\emph{cf}.~ Conj.~\ref{conj-trace-formula}).

Thus, we focus on the elliptic part of local terms, and express them in terms of Kottwitz--Igusa triples. (See Thm.~\ref{th-point-counting-elliptic} for the precise result.)

\subsection*{Case of division algebras}
 Using our results, we calculate the trace when $\Gsf = \Dsf^\times$, where $\Dsf$ is a central division algebra over $\Fsf$ which splits at $x_i$'s.  
Choose $\Xi = \langle\xi\rangle\subset \AA^{\times}$ where $\xi$ is an id\`ele with positive degree, and view it as a subgroup of $\Dsf_\AA^\times\coloneqq (\Dsf\otimes_\Fsf\AA)^\times$. Let $\Ig^{\bbf}_{\Dsf^\times,\xbf,\Xi}$ denote the Igusa variety thus obtained, which is known to be quasi-compact.

\begin{thm}[\emph{cf.}~{Thm.~\ref{th-second-basic-id}}]\label{th-main-div-alg}
In the above setting, we additionally assume that $\dim_{\Fsf}\Dsf=n^2$ with $n<p$.
Then for any $\varphi\in C^{\infty}_c(\Dsf^{\times}\backslash \Dsf_\AA^{\times}/\Xi)$, we have 
\[
\sum_i(-1)^i\tr\big(\varphi\mid \coh i_c(\Ig^{\bbf}_{\Dsf^\times,\xbf,\Xi},\cl\QQ_\ell)\big) = \revise{(-1)^\delta\cdot}\sum_{\pi\subset C^\infty_c(\Dsf^\times\backslash\Dsf^\times_\AA/\Xi)} \tr \big(\varphi\ | \Red^\bbf(\pi)\big),
\]
where  the right hand side is the sum over all automorphic representations $\pi$ of $\Dsf^\times_\AA$ with central character trivial on $\Xi$, with only finitely many nonzero summands. Here,  $\Red^{\bbf}$ is a representation-theoretic operation defined in \eqref{eq-Red}, and the sign  \revise{$(-1)^\delta$} is explicitly determined.
\end{thm}

This theorem is an analogue of the \emph{second basic identity}  relating the cohomology of certain unitary Shimura varieties and corresponding Igusa varieties (\emph{cf.} \cite[Thm.~V.5.4]{HarrisTaylor:TheBook}, \cite[Thm.~1.6]{ShinSW:CohRZ}), provided that the discrete part of the inter\-section cohomology of the moduli stacks of $\Dscr^{\times}$-shtukas has a description similar to the case of Shimura varieties, where $\Dscr$ is a maximal order of $\Dsf$). Such a result is known under a properness assumption and the ``base change fundamental lemma'' for $\GL_n$ in \cite[Thm.~9.3.3]{Lau:Thesis} or can be deduced in full generality from \cite[Cor.~3.1]{LafforgueZhu:elliptic} assuming suitable parity vanishing of the intersection cohomology. See Cor.~\ref{cor-second-basic-id-Kottwitz} and the subsequent discussions for more details.

\subsection*{Acknowledgements}
We are grateful to Sug Woo Shin for his many helpful discussions and his careful reading of the draft. We also thank Eva Viehmann and Pol van Hoften for valuable discussions. We are sincerely grateful to the anonymous referee for the numerous insightful comments, which played a significant role in improving the paper. The first named author was partially supported by ERC Consolidator Grant 770936: NewtonStrat. The second named author was supported by the National Research Foundation of Korea(NRF) grant funded by the Korea government(MSIT) (No.~RS-2023-00208018).

\section{Igusa covers}\label{sect-Igusa-Covers}
\subsection{Setup} In this section, we  allow a more general setup than  the rest of the paper. We fix a local field $F$ of residue characteristic $p$ with ring of integer $O_F$, uniformiser $\unif$ and residue field $\kappa_F$. We denote by $\Fbreve$ the completion of the maximal unramified extension of $F$. If $\cha F = 0$ we will require any $\kappa_F$-algebra, $\kappa_F$-scheme and algebraic stack over $\kappa_F$ to be perfect without further mention. To ensure that we do not leave this category accidentally, the setup for $\cha F = 0$ is less general then for $\cha F = p$. If $\cha F = p$, we will write ``torsor'' instead of ``torsor for the fpqc-topology'' and ``stack'' instead of ``Artin stack''. If $\cha F = 0$, we will write ``torsor'' instead of ``torsor for the pro-\'etale-topology'' and ``stack'' instead of ``Deligne-Mumford stack''. We fix a smooth affine group scheme $G$ over $O_F$ with connected fibres. 
 
  Depending on $F$, we define for any $\kappa_F$-algebra $R$
 \begin{align*}
   D_R &\coloneqq \begin{cases}
 				R\pot{\unif} & \textnormal{ if } \cha F = p \\
 				W_{O_F}(R) & \textnormal{ if } \cha F = 0
 			   \end{cases} \\
 D_R^\ast &\coloneqq \begin{cases}
                   R\rpot{\unif} & \textnormal{ if } \cha F = p \\
 				   W_{O_F}(R)[\unif^{-1}] & \textnormal{ if } \cha F = 0.			   
 				  \end{cases} 
 \end{align*}
 We note that $R \mapsto D_R$ and $R \mapsto D_R^\ast$ are sheaves for the big Zariski site over $\kappa_F$. We say that a $D_R$- or $D_R^\ast$-module is locally free for the Zariski topology on $\Spec R$ if the corresponding module sheaf over $\Spec R$ is. It is known that finitely generated $D_R$-modules are locally free if and only if they are locally free for the Zariski topology on $\Spec R$. (See e.g.~\cite[Lem.~8.9]{Kreidl:padicGr}.)

  The positive loop group and the loop group are defined as the group-valued functors on (perfect) $\kappa_F$-algebras given by
\begin{align*}
 L^+G(R) &\coloneqq G(D_R),\ \text{and} \\
 L G(R) &\coloneqq G(D_R^\ast).
\end{align*}
The functors $L^+G$ and $LG$ are representable by an affine scheme and an ind-scheme, respectively. If $\cha F = 0$ they are perfect by definition as we only allowed perfect test objects. For any $L^+G$-torsor $\Gcal$ over a $\kappa_F$-scheme $S$, we let $\Lcal\Gcal\coloneqq LG \times^{L^+G}\Gcal$ denote the pushout; in other words, $\Lcal\Gcal$ is the image of $\Gcal$ under the natural morphism $[L^+ G\backslash *](S) \to [LG\backslash *](S)$ of the classifying stacks. (This construction extends naturally when $\Gcal$ is an $L^+G$-torsor over a $\kappa_F$-stack $S$.)
 
 \begin{defn} \label{defn-local-shtuka}
  Let $S$ be a $\kappa_F$-stack.
  \begin{subenv}
   \item A \emph{local $G$-isoshtuka} over $S$ is a pair $(\Hcal,\varphi)$ where $\Hcal$ is an $LG$-torsor over $S$ and $\varphi\colon \sigma^\ast\Hcal \to \Hcal$ an isomorphism. A \emph{local $G$-shtuka} over $S$ is a pair $(\Gcal,\varphi)$, where $\Gcal$ is an $L^+G$-torsor over $S$ and $\varphi$ is an isomorphism $\varphi\colon \sigma^\ast \Lcal\Gcal \to \Lcal\Gcal$. 
   \item A local $G$-shtuka $(\Gcal,\varphi)$ is called \emph{\'etale} if $\varphi$ restricts to an isomorphism $\sigma^\ast\Gcal \isom \Gcal$. 
    \end{subenv}
 \end{defn}

 If $S = \Spec C$ for an algebraically closed field $C$, any $LG$-torsor is trivial and hence every local $G$-isoshtuka is isomorphic to $(LG,b\sigma^\ast)$ for some $b \in G(D_{C}^\ast)$. This induces a bijection between the isomorphism classes of local $G$-isoshtukas over $C$ and $\sigma$-conjugacy classes $B(F,G)(C)$ in $G(D_{C}^\ast)$. By \cite[Thm.~1.1]{RapoportRichartz:Gisoc}, this set does not depend on $C$ and will simply be denoted by $B(F,G)$. Given a local $G$-isoshtuka $\underline\Hcal$ over a $\kappa_F$-stack $S$ and $b \in G(\breve{F})$, it is a classical result that the geometric points $\sbar$ of $S$ such that $\Hcal_{\sbar} \cong (LG,b\sigma^\ast)$ form a locally closed substack $S^{[b]}$ (\emph{cf.} \cite[Thm.~3.6]{RapoportRichartz:Gisoc}, and see also Prop.~\ref{prop-purity} below for a stronger result). If $S = S^{[b]}$ for some $b$, we say that $\underline\Hcal$ (or $\underline\Gcal$ if $\Hcal = \Lcal\underline\Gcal$) has constant isogeny class.

 \subsection{The Igusa cover for \'etale local $G$-shtukas} Let $S$ be a $\kappa_F$-stack and let $\underline{\Gcal} = (\Gcal,\varphi)$ be an \'etale local $G$-shtuka over $S$. We denote by $\Gcal^\varphi$ the $\varphi$-invariants of $\Gcal$, i.e.\ the difference kernel of $\varphi$ and the canonical projection $\sigma^\ast \Gcal \to \Gcal$. In the literature $\Gcal^\varphi$ is usually called the \emph{(dual) Tate module} of $\underline\Gcal$. 
 
 An equivalent definition (which generalises better in the case of arbitrary local $G$-shtukas) is as the moduli stack of isomorphisms
 \[
  \Ig_{\underline\Gcal}^e \coloneqq \underline\Isom((L^+G_S,\sigma),\underline{\Gcal}) \cong \Gcal^\varphi,
 \]
 which we call the Igusa cover of $\underline\Gcal$. We briefly summarise its properties.
 
 \begin{prop} \label{prop-Igusa-etale}
The Igusa cover  $\Ig_{\underline\Gcal}^e$ is a pro-\'etale $G(O_F)$-torsor over $S$, where we regard $G(O_F)$ as locally constant profinite $\kappa_F$-group scheme. More precisely, the construction of the Igusa cover yields an equivalence of categories between the category of \'etale local $G$-shtukas and the category of $G(O_F)$-torsors.
  \end{prop}
 \begin{proof}
  Let $\check\Tcal_{\underline\Gcal}$ denote the \emph{(dual) Tate functor} in the sense of  \cite[Def.~3.5]{ArastehRad-Hartl:LocGlShtuka}, which is a (covariant) tensor functor from the category of algebraic $\Gscr$-representations over $O_F$ to the category of finite free $O_F$-modules. (The definition in \emph{loc.~cit.} obviously extends to the case when $\cha F = 0$.) Note that $\check\Tcal_{(L^+G_S,\sigma)}$ is naturally isomorphic to the forgetful functor (i.e., the fibre functor), which we denote by $\omega^\circ$. 
  
  By \cite[Prop.~3.6]{ArastehRad-Hartl:LocGlShtuka}, we get a natural isomorphism $\Ig^e_{\underline\Gcal}\cong\Isom^\otimes(\omega^\circ,\check\Tcal_{\underline\Gcal})$. (The proof of \emph{loc.~cit.} holds verbatim when $\cha F = 0$.) Now, one can repeat the proof of \cite[Lem.~6.2]{ArastehRad-Hartl:UnifStack} in our simpler local setting where we consider an \'etale local $G$-shtuka $\underline\Gscr$ in place of a global $\Gfr$-shtuka, and $O_F$ in place of $\OO^{\underline\nu}$. This argument is worked out in details in \cite[\S2.6]{neupert:thesis} when $\cha F = p$ and that $G$ is defined over $\kappa_F$, but the proof does not make use of this assumption.
 \end{proof}
 
  In order to generalise this proposition to more general local $G$-shtukas, we separate its two assertions - representability of the automorphism group of a local $G$-shtuka and its trivilisation over a profinite \'etale cover.

 \subsection{The automorphism group of a local $G$-shtuka} \label{ssect-aut-group}
 Let $S$ be a scheme over $\cl{\kappa}_F $, $\underline\Gcal$ a local $G$-shtuka over $S$ and $b \in G(D_{\cl{\kappa}_F }^\ast)$. We denote by $\GG_b \coloneqq (L^+G,b\sigma)$ the corresponding local $G$-shtuka over $\cl{\kappa}_F $, and let
  \[
   \Gamma_b \coloneqq \Aut(\GG_b) = \{g \in G(O_{\breve F}) \mid g^{-1}b\sigma(g) = b \}
  \]
  be the group of automorphisms over $\cl{\kappa}_F$, equipped with the $\varpi$-adic topology. As a $\cl{\kappa}_F$-scheme, we have $\Gamma_b = \underline\Aut(\GG_b)^\red$. To see this, we first reduce to $\GL_n$ by choosing a faithful representation $\rho\colon G \to \GL_n$ and associate to $\GG_b$ the local shtuka $(D_{\kappa}^n,\rho(b)\sigma_F)$. Now the claim for $F = \QQ_p$ is proven in \cite[Prop~A.1]{HamacherKim:Mantovan}, one easily sees that the proof generalises verbatim to $\cha F = 0$. If $\cha F = p$ we see that the same proof works as long as the underlying scheme is reduced. (This is required by a nilpotency argument in the proof of \cite[Lem.~3.9]{RapoportRichartz:Gisoc}.) By the same argument one sees that the reduced subscheme automorphism group of the associated isoshtuka is represented by $J_b(F) \coloneqq \{g \in G(\Fbreve) \mid g^{-1}b\sigma(g) = b \}$.

 \subsection{Constancy properties of local shtukas}

 We first consider the case $G = \GL_n$. Using the usual equivalence between $\GL_n$-torsors and rank-$n$ vector bundles, one checks that the definition of a local $\GL_n$-shtuka is equivalent to a rank-$n$ local shtuka, which is defined below. (See e.g.\ \cite[\S4]{HartlViehmann:Newton}.) The advantage of local shtukas is that one can use Zink's lemma to subsequently  reduce considerations to the case of effective local shtukas and then to the \'etale case.
 
  \begin{defn}
  Let $R$ be a $\kappa_F$-algebra.
  \begin{subenv}
   \item A local isoshtuka over $R$ is a pair $(N,\varphi)$ where $N$ is a finitely generated $D_R^\ast$-module which is locally free for the Zariski topology on $\Spec R$ and an isomorphism
   $
    \varphi\colon \sigma^\ast N \isom N.
   $
   Similarly, a local shtuka over $R$ is a pair $(M,\varphi)$, where $M$ is a projective $D_R$-module and an isomorphism
   $
    \varphi\colon \sigma^\ast M[\unif^{-1}] \isom M[\unif^{-1}]
   $
   \item A local shtuka $(M,\varphi)$ is called effective if  $\varphi$ restricts to a morphism $\sigma^\ast(M) \to M$. It is called \'etale, if $\varphi$ restricts to an isomorphism $\sigma^\ast(M) \isom M$
   \item A morphism of local shtukas $(M,\varphi) \to (M',\varphi')$ is a $D_R$-linear morphism $f\colon M \to M'$ such that $f \circ \varphi = \varphi' \circ \sigma^\ast f$.  
   \end{subenv} 
 \end{defn}  
 
 The following two useful lemmas are adapted from Zink's article \cite{Zink:SlopeFil}. We briefly explain how to transfer their respective proofs.
 
 \begin{lem}[{cf.~\cite[Lem.~3 \& 4]{Zink:SlopeFil}}] \label{lem-zink}
  Let $\Mund = (M,\varphi)$ be an effective local shtuka over $R$.
  \begin{subenv}
   \item The functor $\Ccal_{\Mund} \coloneqq \underline\Hom((D_R,\sigma^\ast), \Mund)$ is represented by an affine pro-\'etale scheme over $R$.
   \item If the slope-$0$ part of the Newton polygon of $\Mund$ has constant length over $\Spec R$, $\Ccal_{\Mund}$ is profinite \'etale over $R$. If $R$ is moreover perfect, then $M = M^{et} \oplus M^l$, where
   \begin{align*}
    &M^{et} \coloneqq \Ccal_{\Mund} \otimes_{O_F} D_R,\ \text{and} \\
    &M^l \coloneqq \{v \in M \mid \varphi^n(v) \to 0 \textnormal{ as } n \to \infty \textnormal{ for the }\varpi\textnormal{-adic topology}\}.
   \end{align*}   
   In particular, $\Mund^{et} \coloneqq (M^{et},\restr{\varphi}{M^{et}})$ is an \'etale local shtuka and $\Mund^l \coloneqq (M^l,\restr{\varphi}{M^l}) $ has only positive Newton slopes. 
  \end{subenv}  
 \end{lem}
 \begin{proof}
  Consider the functor
  \[
  \Ccal_{\Mund/\unif^n}(A) = \{v \in M \otimes_{D_R} D_{A,n} \mid \varphi(v) = v \},
  \]
  where we also allow $A$ to be non-perfect in the $\cha F = 0$ case. By the same proof as in the equal characteristic case (\emph{cf.} \cite[Prop.~3]{Zink:SlopeFil}), one shows that $\Ccal_{\Mund/\unif^n}$ is represented by an affine \'etale $R$-scheme and hence $\Ccal_\Mund = \varprojlim \Ccal_{\Mund/\unif^n}$ is represented by an affine pro-\'etale $R$-scheme, proving the first part.
  
  Now assume that  the slope-$0$ part of the Newton polygon of $\Mund$ has constant length over $\Spec R$. By the same proof as in the equal characteristic case (\emph{cf.} \cite[p.~84]{Zink:SlopeFil}), we see that for any $n$ the $\Ccal_{\Mund/\unif^n}$ is finite over $\Spec R$ and that the canonical projection $M^{l}/(M^l \cap \unif^n M) \to M/(M^{et}+\unif^n M)$ is a bijection. Taking the limit, we see that $\Ccal_\Mund$ is profinite \'etale over $\Spec R$ and obtain a splitting $M/M^{et} \isom M^l \subset M$, proving $M = M^{et} \oplus M^l$.
 \end{proof}
 
 This lemma can be used to separate the smallest slope part of a given isocrystal and to it over a profinite \'etale cover. It remains to reduce to the situation where we can find an effective local shtuka within a given isoshtukas.
  
 \begin{lem} \label{lem-effective-shtuka}
  Let $\Nund = (N,\varphi)$ be an isoshtuka over $R$ such that all its Newton slopes are non-negative.
  \begin{subenv}
   \item There exists a surjective (perfectly) proper morphism $S \to \Spec R$, and an effective local shtuka $\Mund$ over $S$ such that $\Mund[\unif\iv] \cong \Nund_S$.
   \item If $R$ is absolutely flat (e.g.\ a field) or if $R$ is normal and $N$ is isoclinic of slope $0$, $\Mund$ can already be defined over $R$.
  \end{subenv}
 \end{lem}
 \begin{proof}
  Assume first that $R=K$ is a field and choose a full $D_R$-lattice $M_0 \subset N$. By  Zink's lemma \cite[Lem.~9]{Zink:SlopeFil}, there exists an explicit constant $c \in \NN$ depending only on the Newton polygon of $\Nund$ and a lattice $M_0 \subset M \subset \unif^{-c} M_0$ such that $(M,\varphi)$ is effective. Zink's construction can be applied for any base ring $R$, but the resulting $D_R$-module $M \subset N$ might not be locally free (or finitely presented even). However since it commutes with localisation, we see that it still yields an effective local shtuka if $R$ is an absolutely flat ring. 

  Denote by $\Grass_{N,c}$ the moduli space of locally free submodules $M_0 \subset M \subset \unif^{-c}M_0$. It is represented by a (perfectly) proper $R$-scheme by \cite[Thm.~2.3]{BeauvilleLaszlo:ThetaFunctions} ($\cha F = p$) and \cite[Thm.~8.3]{BhattScholze:AffGr} ($\cha F = 0$), respectively. The condition $\varphi(\sigma^\ast M) \subset M$ is a cut out a closed subscheme $S$, which maps surjectively to $\Spec R$ by above considerations.
  
  If $N$ is isoclinic of slope $0$, then the condition of $M$ being effective is equivalent of $M$ being \'etale; so the geometric fibres of $S \to \Spec R$ must be finite by the argument in \cite[Lem.~2.5]{OortZink:pdiv}. (To give more details, we may replace $R$ with an algebraically closed field extension $\kappa$ of $\FF_q$, which we do. Then $N$ admits a basis in $N^\varphi$, and there exists a full $D_R$-sublattice $M_0'\subset M_0$ spanned by a basis of $N^\varphi$. Now choosing $c'$ so that we have $\varpi^{-c}M_0\subset \varpi^{-c'}M_0'$, it follows that $|S|$ is bounded by the number of $O_F$-lattices  $\Lambda\subset N^\varphi$ satisfying $M_0'^\varphi \subset \Lambda\subset \varpi^{-c'}M_0'^\varphi$, which is finite.) Hence the map $S \to \Spec R$ is (perfectly) finite by Zariski's main theorem. If $R$ is normal, its total ring of fractions $K$ is absolutely flat, hence we get a section $s \in S(K)$. Since $S \to \Spec R$ is integral, this section extends (uniquely) to $R$ by the universal property of  normalisation.
 \end{proof}

 From the above lemmas we now deduce the main constancy result. It was previously deduced in the case of Barsotti-Tate groups over strictly Henselian rings by Oort and Zink \cite[Prop.~3.3]{OortZink:pdiv} and Caraiani and Scholze \cite[Lem.~4.3.15]{CaraianiScholze:ShVar}. We use the same general strategy as their respective proofs.
 
 \begin{prop} \label{prop-constancy-of-isoshtuka}
  Assume that $R$ is perfect and normal, and $\Nund = (N,\varphi)$ is an isoshtuka with constant Newton polygon. Then there exists a faithfully flat profinite \'etale $R$-algebra $R'$ such that $\Nund_{R'}$ is defined over some finite extension $\kappa$ of $\kappa_F$.
 \end{prop} 
 \begin{proof}
  Let $\frac{r}{s}$ denote the smallest Newton slope of $N$ and let $\Phi = \unif^{-r}\varphi^s$. By \cite[Thm.~4.1]{BhattScholze:AffGr} $N$ can be considered as sheaf for the $v$-topology on $\Spec R$. We define the subsheaves
  \begin{align*}
   & N^{et} \coloneqq \Ccal_{(N,\Phi)} \otimes_F D_R^\ast, \ \text{and} \\
   & N^l \coloneqq \{v \in N \mid \Phi^n(v) \to 0 \textnormal{ for } n \to \infty\}.
  \end{align*}
  The canonical morphism $N^{et}\oplus N^l \to N$ pulls back to an isomorphism via a surjective perfectly proper base change by Lem.~\ref{lem-effective-shtuka}~(1) and Lem.~\ref{lem-zink}, so it is an isomorphism by $v$-descent (\emph{cf}.~\cite[Thm.~4.1]{BhattScholze:AffGr}). Thus, we may assume without loss of generality that $N$ is of constant slope $r/s$. 
  
 Since $R$ is normal, there exists an effective local shtuka $\Mund$ such that $\Mund[\unif\iv] \cong (N,\Phi)$ by Lem.~\ref{lem-effective-shtuka}~(2). Let $R \to R'$ be a faithfully flat profinite \'etale morphism simultaneously refining the pro-\'etale coverings $\Ccal_{\Mund}$ and $R \otimes_{\kappa_F} \kappa$, where $\kappa/\kappa_F$ is an extension of degree $s$. Then by Lem.~\ref{lem-zink}, we have 
  \[
   \Nund_{R'} \cong (N^{\Phi = 1},\restr{\varphi}{N^{\Phi = 1}}) \otimes_\kappa R',
  \]
  which proves the proposition.
  %
 \end{proof}

 As a consequence, we obtain the following result for $G$-isoshtukas.
 
 \begin{thm} \label{thm-constancy-of-G-isoshtukas}
  Let $\underline{\Hcal}$ be a $G$-isoshtuka over a perfect normal $\kappa_F$-stack $S$ with constant isogeny class $[b]$. Then there exists a profinite \'etale cover $S' \to S$ such that $\underline\Hcal_{S'} \cong (LG,b\sigma)_{S'}$. In other words the functor
  \[
   X^b_{\underline\Hcal}\colon T \to \Isom((LG,b\sigma)_{T}, \underline\Hcal_{T})
  \]
  on perfect $S$-stacks is represented by a pro-\'etale $J_b(F)$-torsor.
 \end{thm}
 \begin{proof}

 Since we can prove the statement locally we may assume that $S = \Spec R$ is affine. We choose a closed embedding $\rho\colon G \to \GL(V)$ which identifies $G$ with the stabiliser of a line $L \subset t(V)$ inside some tensorial construction. (See \cite[Thm.~1.1]{Broshi:Gtorsors} with a minor correction in \cite[Rmk.~A.19]{ImaiKatoYoucis:Tannaka}.) This induces an equivalence of categories between $LG$-torsors and pairs of locally free $D_R^\ast$-modules $\Vcal, \Lcal \subset t(\Vcal)$ which are \'etale locally isomorphic to $V, L \subset t(V)$. Hence we obtain a pair of isoshtukas $i\colon \underline\Lcal \mono t(\underline\Vcal)$ associated to $\underline\Hcal$. 
 
 By the previous proposition, we may assume that $\underline\Lcal$ and $\underline\Vcal$ have models $\underline\Lcal_0,\underline\Vcal_0$ over $\cl{\kappa}_F$ after replacing $S$ by a profinite \'etale cover. It remains to show that $i$ can also be defined over $\cl{\kappa}_F$, possibly after composition with an automorphism of $\underline\Vcal$. By \cite[Prop.~A.1]{HamacherKim:Mantovan} (\emph{cf}.\ comment in \S\ref{ssect-aut-group} for $\cha F = p$), $\underline\Hom(\underline\Lcal_0,t(\underline\Vcal_0))$ is represented by a finite-dimensional $F$-vector space $H$. Thus $i$ induces a continuous map $f\colon S' \to H \setminus \{0\}$. We note that $\Aut(\underline\Vcal_0)$ acts linearly on  $H$ via post-composition. Since $\underline\Hcal$ has constant isogeny class, $f(S')$ lies inside a single orbit of $\Aut(\underline\Vcal_0)$.  As every action map $\GL(H) \to H, g \mapsto g.h$ splits Zariski-locally, we may factor $f$ as
 \begin{center}
  \begin{tikzcd}
   & \Aut(\underline\Vcal_0) \arrow{d}{j \mapsto t(j) \circ h} \\
   S' \arrow{ru}{f'} \arrow{r}{f} & H \setminus \{0\}
  \end{tikzcd} 
 \end{center}
 for some $h \in H \setminus \{0\}$ after replacing $S'$ by a Zariski cover. Thus $i = h \circ t(j)$, where $j \in \Aut(\underline\Vcal) = C^0(S',\Aut(\underline\Vcal_0))$ corresponds to $f'$. Thus after composition with $t(j)\iv$, the morphism $i=h$ is defined over $\cl{\kappa}_F$.
 \end{proof}
 
As a corollary, we obtain Tate's isogeny theorem for local $G$-shtukas with constant Newton point. This result was originally formulated for $p$-divisible groups, but one can immediately formulate its analogue for $F$-crystals and local shtukas.  Caraiani and Scholze \cite[Rmk.~4.2.17]{CaraianiScholze:ShVar} proved it for $F$-crystals over any integral normal $\FF_p$-scheme, removing the Noetherian hypothesis from the earlier result of Berthelot \cite{Berthelot:Dieudonne}. For equal characteristic local $G$-shtukas for a constant reductive group $G$, Neupert \cite[Thm.~2.7.6]{neupert:thesis} showed Tate's isogeny theorem over Noetherian normal base schemes.

  \begin{prop}[Tate's isogeny theorem for local $G$-shtukas] \label{prop-tates-thm}
  Let $S$ be an integral normal scheme over $\kappa_F$ with generic point $\eta = \Spec K$. Let $\underline\Gcal_1, \underline\Gcal_2$ be two local $G$-shtukas with constant Newton point. Then the restriction map
  \[
   \Isom(\Lcal\underline\Gcal_1,\Lcal\underline\Gcal_2) \to \Isom(\Lcal\underline\Gcal_{1,\eta},\Lcal\underline\Gcal_{2,\eta})
  \]
  is a bijection, which identifies $\Isom(\underline\Gcal_1,\underline\Gcal_2)$ with $\Isom(\underline\Gcal_{1,\eta},\underline\Gcal_{2,\eta})$.
 \end{prop}
 \begin{proof}
  Since $L^+G$ is closed in $LG$, the property of an isomorphism of $LG$-torsors $\Lcal\Gcal_1 \isom \Lcal\Gcal_2$ restricting to an isomorphism $\Gcal_1 \isom \Gcal_2$ is closed; hence it suffices to prove the first assertion.   

 We first claim that $\underline\Isom(\Lcal\Gcal_1,\Lcal\Gcal_2)$ can be represented by an ind-scheme that is ind-integral over $S$. Indeed, the ind-representability and ind-affineness over $S$ follows since $\underline\Isom(\Lcal\underline\Gcal_1,\Lcal\underline\Gcal_2)$ is a closed subfunctor of the $LG$-torsor $\underline\Isom(\Lcal\Gcal_1,\Lcal\Gcal_2)$ where $\Lcal\Gcal_i$ is the underlying $LG$-torsor of $\Lcal\underline\Gcal_i$. It now remains to show that $\underline\Isom(\Lcal\underline\Gcal_1,\Lcal\underline\Gcal_2)$ is universally specialising over $S$, which follows from Thm.~\ref{thm-constancy-of-G-isoshtukas}.

Now fix an isomorphism $\varphi\colon \Lcal\underline\Gcal_{1,\eta} \to \Lcal\underline\Gcal_{2,\eta}$, which gives rise to a $K$-point in $\underline\Isom(\Lcal\underline\Gcal_1,\Lcal\underline\Gcal_2)$. By ind-integrality, the closure $T$ of $\varphi$ in $\underline\Isom(\Lcal\underline\Gcal_1,\Lcal\underline\Gcal_2)$ is birational and integral over $S$. By normality of $S$, there exists a section $S \to T$, which corresponds to the isomorphism $\Lcal\underline\Gcal_1 \to \Lcal\underline\Gcal_2$ extending $\varphi$. 
 \end{proof}

We can strengthen the result above to obtain a slightly more general statement for local isoshtukas.
 \begin{lem} \label{lem-Tates-thm}
 Let $S$ be an integral normal scheme over $\kappa_F$ with generic point $\eta = \Spec K$. Let $\Nund_1$ and $\Nund_2$ be local shtukas (with possibly different ranks), and let $\lambda_{1}$ and $\lambda_{2}$ respectively denote the largest slopes of the Newton polygon of $\Nund_1$ and $\Nund_2$.  Suppose that one of the following conditions holds:
 \begin{enumerate}
     \item\label{lem-Tates-thm-constant} The Newton polygons of $\Nund_1$ and $\Nund_2$ are both constant.
     \item\label{lem-Tates-thm-purity} The base scheme $S$ is perfect, and the slope $\leq\lambda_2$ part of the Newton polygon of $\Nund_1$ and the slope $\leq\lambda_1$ part of the Newton polygon of $\Nund_2$ are constant.
 \end{enumerate}
 Then the restriction map
\[\Hom(\Nund_1,\Nund_2) \to \Hom(\Nund_{1,\eta},\Nund_{2,\eta})\]
is a bijection.
 \end{lem}
\begin{proof}
    If the Newton polygons of $\Nund_1$ and $\Nund_2$ are constant, then the  proof of Prop.~\ref{prop-tates-thm} can be adapted to show the lemma in case~(\ref{lem-Tates-thm-constant}) since $\underline\Hom(\Nund_1,\Nund_2)^{\rm red}$ is represented by a locally profinite group by the same argument as in \S~\ref{ssect-aut-group}.

    In case~(\ref{lem-Tates-thm-purity}), we claim that we have
    \[\Nund_1 = \Nund_1^{\leqslant \lambda_2}\oplus \Nund_1^{>\lambda_2} \quad\textnormal{and}\quad \Nund_2 = \Nund_2^{\leqslant \lambda_1}\oplus\Nund_2^{>\lambda_1},\]
    where $\Nund_i^{\leqslant\lambda_j}$ and $\Nund_{i}^{>\lambda_j}$ are respectively the slope $\leqslant\lambda_j$ part and the slope $>\lambda_j$ part of $\Nund_i$. (In fact, the first paragraph of the proof of Prop.~\ref{prop-constancy-of-isoshtuka} actually shows that if $\Nund$ is a local isoshtuka over a perfect scheme such that the first break point of the Newton polygon is constant, then the smallest slope part of $\Nund$ is a direct summand of $\Nund$. We repeatedly apply this to $\Nund_1$ for all slopes $\leq \lambda_2$, and similarly for $\Nund_2$.)
    
    Now, for any homomorphism $\Nund_1\to\Nund_2$ of local shtukas, its kernel contains $\Nund_1^{>\lambda_2}$ and its image is contained in $\Nund_2^{\leqslant\lambda_1}$; i.e., we have a natural bijection 
    \[\Hom(\Nund_1,\Nund_2) \riso \Hom(\Nund_1^{\leqslant\lambda_2},\Nund_2^{\leqslant\lambda_1}).\]
    The lemma in case~(\ref{lem-Tates-thm-purity}) now follows from case~(\ref{lem-Tates-thm-constant}) as $\Nund_1^{\leqslant\lambda_2}$ and $\Nund_2^{\leqslant\lambda_1}$ have constant Newton polygons by assumption.
\end{proof}
 %

 \subsection{The Igusa cover for local $G$-shtukas} \label{ssect-igusa-cover}
 Let $\kappa/\kappa_F$ be an algebraic field extension. For a fixed element $b \in G(D_\kappa^\ast)$ we denote by $\GG_b = (L^+G_{\kappa},b\sigma)$ the associated local $G$-shtuka over $\kappa$. Now let $S$ be a $\kappa$-stack and let $\underline\Gcal \coloneqq (\Gcal, \varphi)$ be a local $G$-shtuka over $S$. We define the \emph{central leaf} of $\underline\Gcal$ as the set
   \[
     C^{b}_{\underline\Gcal} \coloneqq \{s \in S \mid \underline{\Gcal}_{\overline{\kappa(s)}} \cong \GG_{b,\overline{\kappa(s)}}\}.
   \]
 We will later see that it is closed inside the Newton stratum of $b$ and can thus be equipped with the structure of a reduced substack. We define the \emph{Igusa cover} as  the $\kappa$-stack solving the moduli problem
  \[
   \Ig_{\underline\Gcal}^b(T) \coloneqq \{(P,j) \mid P \in S(T), j\colon \GG_{b,T} \isom \underline{\Gcal}_T \}.
  \]

 \begin{prop} \label{prop-igusa-cover}
  In the situation of above definition the following holds.
  \begin{subenv}
   \item\label{prop-igusa-cover-affine} The canonical projection $p\colon \Ig_{\underline\Gcal}^b \to S$ is an affine morphism.
   \item\label{prop-igusa-cover-Newton} The morphism $p$ factors through the Newton stratum $S^{[b]}$ of $S$ associated to the $\sigma$-conjugacy class of $b$. The induced morphism $\Ig_{\underline\Gcal}^b \to S^{[b]}$ is integral.
   \item\label{prop-igusa-cover-cen-leaf} The image of $p$ equals $C_{\underline{\Gcal}}^b$. In particular, $C_{\underline{\Gcal}}^b$ is closed in $S^{[b]}$.
  \end{subenv}  
 \end{prop}
 \begin{proof}
  Note that it suffices to prove the above statements after passing to the perfection. Thus we assume that $S$ is perfect.
  
  Since $\underline\Isom (L^+G,\Gcal) \to S$ is an $L^+G$-torsor, it is an affine morphism. Since $\Ig_{\underline\Gcal}^b = \ker (j^{\rm univ} \circ b\sigma, \varphi \circ j^{\rm univ}) \subset \underline\Isom(L^+G,\Gcal)$ is closed, it follows that $\Ig_{\underline\Gcal}^b \to S$ is affine as well, proving (\ref{prop-igusa-cover-affine}).
  
  The first part of (\ref{prop-igusa-cover-Newton}) is obvious. Since we already know that $\Ig_{\underline\Gcal}^b  \to S^{[b]}$ is affine, it remains to show that it is universally closed. By \cite[\href{https://stacks.math.columbia.edu/tag/0CLW}{Lem.~0CLW}]{AlgStackProj}, we have to check that it satisfies the existence part of the valuative criterion. Thus we consider a commutative diagram
  \begin{center}
   \begin{tikzcd}
    \Spec K \arrow{r} \arrow{d} & \Ig_{\underline\Gcal}^b \arrow{d} \\
    \Spec B \arrow{r} & S^{[b]}
   \end{tikzcd}
  \end{center}
  where $B$ is a valuation ring and $K$ its fraction field. By Prop.~\ref{prop-tates-thm} we can extend the isomorphism $j_K\colon \GG_{b,K} \to \underline{\Gcal}_K$ to an isomorphism $j\colon \GG_{b,B} \rightarrow \underline{\Gcal}_B$, defining a lift $\Spec B \to \Ig_{\underline\Gcal}^b$, proving (\ref{prop-igusa-cover-Newton}).
  
  To prove (\ref{prop-igusa-cover-cen-leaf}), denote by $C$ the set-theoretic image of $\Ig_{\underline\Gcal}^b$ in $S^{[b]}$. 
  Since $\Ig_{\underline\Gcal}^b$ is integral over $S^{[b]}$, a point $s\in S^{[b]}$ lies in $C$ if and only if we have $\underline{\Gcal}_{\overline{\kappa(s)}} \cong \GG_{b,\overline{\kappa(s)}}$. Hence, the central leaf $C_{\underline\Gcal}^b$ is the set-theoretic image of $\Ig_{\underline\Gcal}^b$, which is integral over $S^{[b]}$, and thus $C_{\underline{\Gcal}}^b$ is closed in $S^{[b]}$.
 \end{proof}

 \subsection{Finite-level Igusa covers}\label{ssect-finite-level-igusa-cover}
  To construct Igusa cover with finite or partial level structure, we start by considering the Tate module of an \'etale local $G$-shtuka. Assume that $\underline\Gcal$ is an \'etale local $G$-shtuka over a $\kappa_F$-stack $S$, then $\Ig^e_{\underline\Gcal}$ is a pro-\'etale $G(O_F)$-torsor by Prop.~\ref{prop-Igusa-etale} Hence for any closed subgroup $K \subset G(O_F)$, we define the Igusa cover $\Ig_{\underline\Gcal,K}^e$ of $K$-level as the associated $G(O_L)/K$-fibration over $S$. In particular, if the index $[G(O_F):K]$ is finite, then $\Ig_{\underline\Gcal,K}^e \to S$ is a finite morphism.
  
 For arbitrary $b \in G(\Fbreve)$, we had seen in \S\ref{ssect-aut-group} that $\Gamma_b = \Aut(\GG_b)^{\rm red}$, hence we get a canonical $\Gamma_b$-action on $\Ig_{\underline\Gcal}^b$. By construction, the canonical projection $r \colon \Ig_{\underline\Gcal}^b \to C^b_{\underline\Gcal}$ is $\Gamma_b$-invariant and $\Gamma_b$ acts simply transitively on its geometric fibres. To any closed subgroup $K \subset \Gamma_b$ we associate the finite level Igusa cover $\Ig^b_{\underline\Gcal,K}$ which we define as the GIT-type quotient $\Ig^b_{\underline\Gcal} \sslash K$ below. Since $r$ is affine, we may write $\Ig^b_{\underline\Gcal} = \underline\Spec \Ascr$ where $\Ascr$ is a quasi-coherent $\Oscr_{C^b_{\underline\Gcal}}$-algebra. We define
  \[
   \Ig^b_{\underline\Gcal,K} = \Ig^b_{\underline\Gcal}\sslash K \coloneqq \underline\Spec \Ascr^K.
  \]
 If $[\Gamma_b:K] < \infty$, or equivalently if $K$ is open, we also call $\Ig^b_{\underline\Gcal,K}$ a finite Igusa cover. For $K' \subset K$, the embedding $\Ascr^K \subset \Ascr^{K'}$ induces a morphism $r_{K,K'}\colon \Ig^b_{\underline\Gcal,K'} \to \Ig^b_{\underline\Gcal,K}$; thus the finite level Igusa covers $\{\Ig^b_{\underline\Gcal,K}\}_{K \subset \Gamma_b}$ form a projective system.
 
 \begin{lem} \label{lem-igusa-tower}
 Let $S$ be an  $\kappa_F$-stack, $\underline\Gcal$  a local $G$-shtuka over $S$ and $b \in G(\breve{F})$.
  \begin{subenv}
   \item\label{lem-igusa-tower-finite-level} If $K' \trianglelefteq K$ is an open normal subgroup, then $r_{K,K'}$ is a finite Galois cover with Galois group $K/K'$.
   \item\label{lem-igusa-tower-proet} The canonical morphism $\Ig^b_{\underline\Gcal} \to \varprojlim_{K\, \text{open}} \Ig^b_{\underline\Gcal,K}$ is an isomorphism, and the canonical projection $r_K \colon \Ig^b_{\underline\Gcal} \to \Ig^b_{\underline\Gcal,K}$ is a pro-\'etale $K$-torsor.
   \item The structure morphism $\Ig^b_{\underline\Gcal,\Gamma_b} \to C^b_{\underline\Gcal}$ is a universal homeomorphism.
  \end{subenv}
 \end{lem}
 \begin{proof}
  Since all assertions are local for the flat topology, we may assume that $S$ is a scheme. The first assertion now follows from usual GIT since $K/K'$ acts freely on $\Ig^b_{\underline\Gcal,K}$. The first part of the second assertion is equivalent to $\Ascr = \varinjlim \Ascr^K$, which is true since the action of $\Gamma_b$ is continuous. By (\ref{lem-igusa-tower-finite-level}) the morphism $r_K$ is a profinite \'etale $K = \varprojlim K/K'$-torsor, proving the second part. Finally, $\Ig^b_{\underline\Gcal,\Gamma_b} \to C^b_{\underline\Gcal}$ is universally bijective by our considerations on the $\Gamma_b$-action above and also integral by construction, hence universally closed. Thus it is a universal homeomorphism.
 \end{proof}
 
 In the general case where $\kappa$ is an algebraic field extension of $\kappa_F$, the scheme $\Aut(\GG_b)^\red$ is given by the profinite group $\Gamma_b$ equipped with the Galois descent datum induced by the Frobenius action on $G(O_{\breve F})$. If $K \subset \Gamma_b$ is a closed $\Gal(\kbar/k)$-stable subgroup, one checks that the Galois descent datum on $\Ig^b_{\underline\Gcal}$ induces a Galois descent datum on $ \Ig^b_{\underline\Gcal,K}$. Thus $\Ig^b_{\underline\Gcal,K}$ is defined over $\kappa$.
 
 From the next section on, we will work exclusively with the case $\cha F= p$. We finish this chapter with an outlook how the above constructions can be applied in the $p$-adic case. 
 
 \subsection{Application to Barsotti-Tate groups and Shimura varieties} \label{ssect-applications}
  Let $X$ be a {\BT} over an $\FF_p$-scheme $S$ of height $n$. By a result of Gabber, the Dieudonn\'e functor defines an equivalence of categories between the categories of \BT s over $S^\perf$  and local shtukas $(M,\varphi)$ satisfying $M \subset \varphi(M) \subset pM$. (See also \cite[Thm.~A.1]{Lau:DisplayCrystals}.) Thus if $\underline\Gscr$ denotes the corresponding local $\GL_n$-shtuka, then we have $\Ig^b_{\underline\Gcal} = \underline\Isom(\XX_b,X)^\perf$ where $\XX_b$ is the {\BT} over $\overline\FF_p$ with Dieudonn\'e-module $(\breve\ZZ_p^n,b\sigma)$. In particular, Prop.~\ref{prop-igusa-cover}~(3) recovers the classical statement of Oort that central leaves are closed inside their Newton stratum (\cite[Thm.~2.2]{Oort:Foliations}), removing the requirement that $S$ is excellent.

 Another application is to purity. In \cite[Main~Theorem~B]{Vasiu:CrystallineBoundedness} Vasiu proves that the Newton polygon stratification associated to an $F$-crystal over an $\FF_p$-scheme $S$ satisfies the purity property; i.e.\ the embedding of a (locally closed) Newton stratum is an affine morphism. More generally, one may consider the purity property for the union $S^{[b]_i}$ of all Newton strata where the Newton polygons are all on or below a given Newton polygon $\nu(b)$ and contain a fixed break point $(i,\nu(b)(i))$ of $\nu(b)$. 
  The purity property for a single Newton stratum $S^{[b]}$ can be recovered from the purity property for $S^{[b]_i}$ as $S^{[b]}$ is the intersection of $S^{[b]_i}$'s for all break points of $\nu(b)$. 
  Note that the notion of break point $i$ for $\nu(b)$ and $S^{[b]_i}$ has a group-theoretic generalisation for $F$-isocrystals with $G$-structure on $S$, where certain relative roots for $G$ serve as break points for $\nu(b)$; \emph{cf.} \cite[p.~503]{Viehmann:NewtonStratLoopGp}, \cite[Def.~2.3]{Hamacher:DeforSpProdStr}.
 
Given an $F$-isocrystal over a locally Noetherian $\FF_p$-scheme $S$, Yang  \cite{Yang:Purity} proved that the complement of $S^{[b]_i}$ in $S^{\leqslant[b]}$ is of pure codimension one (which is weaker than being affine) or empty. This statement has been generalised to $F$-isocystals with additional $G$-structure by the first named author in \cite{Hamacher:DeforSpProdStr}. By generalising Vasiu's approach, Viehmann proved in \cite{Viehmann:NewtonStratLoopGp} that for local $G$-shtuka over an integral local Noetherian scheme $S$, the subschemes $S^{[b]_i}$ are affine over $S$ if $G$ is a split reductive group. Using our results above, we can remove the restrains from above results.

\begin{prop} \label{prop-purity}
 Let $\underline\Hcal$ be a local $G$-isoshtuka over a $\kappa_F$-scheme $S$. Then  $S^{[b]_i}$ is a locally closed affine subscheme of $S$.
\end{prop}
\begin{proof}
 We assume without loss of generality that $S = \Spec R$ is affine and perfect, and that $S = S^{\leqslant[b]}$. By \cite[Prop.~1]{Hamacher:DeforSpProdStr} it suffices to consider the case $G = \GL_n$; denote by $\underline{N}$ the local isoshtuka corresponding to $\underline\Hcal$, and choose a local shtuka $\Mund$ such that $\Mund[\unif\iv] = \Nund$. By the argument \cite[\S4.1]{Viehmann:NewtonStratLoopGp} we may further reduce to the case that the break point equals $(1,0)$. Here we use the notion of upper convex Newton polygons, i.e.\ $S^{[b]_{1}}$ contains all Newton strata \emph{above} the Newton polygon of $[b]$ which contain the break point $(1,0)$.
 
 Note that $S^{[b]_{1}}$ is the image of $\underline\Hom_S((D_R,\sigma),\Mund) \setminus \{0\} $ in $S$, or equivalently that of
 \[
  S' = \{ (f,g) \in \underline\Hom_S((D_R,\sigma),\Mund) \times \underline\Hom_S(\Mund,(D_R,\sigma)) \mid f \circ g = \id \}.
 \]
 Since $S'$ is affine and satisfies the valuative criterion over its image $S^{[b]_{1}}$ by case~(\ref{lem-Tates-thm-purity}) of Lem.~\ref{lem-Tates-thm}, it follows that the morphism $S' \to S^{[b]_{1}}$ is integral and $S^{[b]_{1}}$ is affine.
\end{proof}
  
  We now consider the case that $S$ is the special fibre of the integral model of a Hodge type Shimura variety $\Sh(\Gsf,\Xsf)_\Ksf$ as constructed by Kisin and Pappas in \cite{KisinPappas:ParahoricIntModel}. By construction, we obtain a principally polarised abelian scheme $(A,\lambda)$ over $S$, which is moreover equipped with a family $\Psi$-invariant tensors $(t_\alpha)$ in the display $P(A[p^\infty])$ by \cite[Prop.~1]{HamacherKim:Mantovan}. As above, the triple $(A[p^\infty], \lambda, (t_\alpha))$ corresponds to a local $\Gsf_{\ZZ_p}$-shtuka $\underline\Gcal$ over $S^\perf$. (See also proof of \cite[Cor.~4.12]{HamacherKim:Mantovan}.) Hence $\Ig_{\underline\Gcal}^b$ recovers the construction of the Igusa variety in \cite[\S6.1]{HamacherKim:Mantovan}, removing the hypothesis that $b$ is decent.  We may even extend the definition of the Igusa variety to Shimura varieties of abelian type with hyperspecial level structure at $p$ as the crystalline canonical model constructed by Lovering \cite[(3.4.8)]{Lovering:GCrys} defines a local $\Gsf_{\ZZ_p}^c$-shtuka over $S^\perf$. (To explain, the crystalline canonical model in \emph{loc.~cit.} is an \emph{$F$-crystal with $\Gsf^c_{\ZZ_p}$-structure} on the $p$-adically completed integral canonical model in the sense of  \cite[(2.4.7)]{Lovering:GCrys}, which integrally corresponds to the pro-\'etale $\Gsf^c_{\ZZ_p}(\ZZ_p)$ torsor of the tower of Shimura varieties; \emph{cf.} \cite[Thm.~3.5.1]{Lovering:GCrys}. It consists of a formal $\Gsf_{\ZZ_p}^c$-torsor with some extra structure, and the desired local $\Gsf_{\ZZ_p}^c$-shtuka can be obtained by pulling it back to a formal $\Gsf^c_{\ZZ_p}$-torsors over $\Spf W(S^\perf)$ and identifying it with a $L^+\Gsf^c_{\ZZ_p}$-torsor over $S^{\perf}$ by the same proof as \cite[Prop.~2.4]{ArastehRad-Hartl:LocGlShtuka}.) 
  
  As a consequence we give another proof of the following results announced by Shen and Zhang; \emph{cf.} \cite[(5.4.5)]{ShenZhang:ShVarStrat}.

 \begin{prop}\label{prop-Shen-Zhang}
  Central leaves  (called canonical central leaves in \cite{ShenZhang:ShVarStrat}) in Shimura varieties of abelian type with hyperspecial level structure are non-empty and closed inside the Newton stratum. Moreover, each central leaf is a union of connected components of the ``adjoint'' central leaf of the induced $\Gsf^{\mathrm{ad}}_{\ZZ_p}$-shtuka containing it.
 \end{prop}
 \begin{proof}
  The non-emptiness is a consequence of \cite[Thm.~0.3]{Kisin:LanglandsRapoport} together with non-emptiness of unramified affine Deligne--Lusztig varieties. In \emph{loc.~cit.}, Kisin constructs a bijection 
  \[
   \Sscr_{K_p}(\Gsf,\Xsf)(\cl\FF_p) \bij \bigsqcup_\phi \varprojlim_\phi I_\phi(\QQ) \backslash X_p(\phi) \times X^p(\phi)/\Ksf^p.
  \]
  Here the left hand side denotes $\cl\FF_p$-valued points of the Shimura variety with infinite level away from $p$ (the level of away from $p$ is irrelevant to tour question) and $X_p(\phi)$ denotes an affine Deligne--Lusztig variety; i.e.,\ is of the form
  \[
   X_p(\phi) = \{g \in G(\breve\QQ_p)/\breve\Ksf_p \mid g\iv b_\phi \sigma(g) \in \breve\Ksf_p p^{-\mu} \breve\Ksf_p \}, 
  \]
 where $\breve\Ksf_p\coloneqq \Gsf_{\ZZ_p}(\breve\ZZ_p)$ and $\mu$ is a suitable cocharacter as in \cite[\S3.3]{Kisin:LanglandsRapoport}. Note that $X_p(\phi)$ is known to be non-empty. (See e.g.~\cite[Thm.~5.2]{Gashi:KRconj}.)
  Then we claim that the isomorphism class of the local shtuka over a point of the Shimura variety with corresponding element $g_p \in X_p(\phi)$ is isomorphic to  $\GG_{g_p\iv b_\phi \sigma(g_p)}$: For Shimura varieties of Hodge type this follows from the construction, the relevant step being \cite[Prop.~1.4.4]{Kisin:LanglandsRapoport}. To conclude the general case we note that Kisin and Lovering use the same twisting operation to obtain their result for abelian-type Shimura varieties from the Hodge type case in \cite[\S4.7]{Kisin:LanglandsRapoport}, \cite[\S3.4]{Lovering:GCrys}. Since we obtain all possible isomorphism classes this way, we conclude non-emptiness.

  The closedness was proven in Prop.~\ref{prop-igusa-cover}~(3). To prove the second assertion, it therefore remains to prove that every adjoint central leaf only contains finitely many central leaves; in other words, we want to show that \emph{any fibre of the map}
\begin{equation}\label{eq-shen-zhang}
    (\breve\Ksf^cp^{-\mu^c}\breve \Ksf^c_p)/_{\breve\Ksf^c_p\text{-}\sigma\text{ conj}} \to \Gsf^{\ad}(\breve\QQ_p)/_{\breve\Ksf_p^{\ad}\text{-}\sigma\text{ conj}}
\end{equation}
\emph{is finite}, where $\breve\Ksf^c_p\coloneqq\Gsf^c_{\ZZ_p}(\breve\ZZ_p)$, $\breve\Ksf^{\ad}_p\coloneqq\Gsf^{\ad}_{\ZZ_p}(\breve\ZZ_p)$ and $\mu^c$ is the $\Gsf^c$-valued cocharacter induced by $\mu$. Indeed, examining the proof of \cite[Lem.~4.2.1]{ShenZhang:ShVarStrat}, one can observe that any fibre of the map \eqref{eq-shen-zhang} is a principal homogeneous space under the action of $\coker(\Lcal\colon\Zsf^c(\breve\ZZ_p) \to \Zsf^c(\breve\ZZ_p))$ where $\Zsf^c$ is the centre of  $\Gsf^{c}$ and $\Lcal(t) \coloneqq t^{-1}\sigma(t)$ is the Lang isogeny. 
The proof of \emph{loc.~cit.} also shows that $\coker(\Lcal)$ is finite as $\Lcal$ is surjective on the neutral connected component $\Zsf^{c,0}(\breve\ZZ_p)$. This shows that any fibre of \eqref{eq-shen-zhang} is finite.
 \end{proof}

\section{Moduli stacks of global $\Gscr$-shtukas} \label{sect-moduli}
In this and the following sections, we will work with the same setup as in \S\ref{sect-intro}; in particular, $\Gscr$ is a smooth affine group scheme over $C$ with connected fibres and reductive generic fibre $\Gsf$. We start by recalling the definition of global $\Gscr$-shtukas and the related moduli stacks. The principal references for this section are the papers of Arasteh Rad and Hartl \cite{ArastehRad-Hartl:UnifStack} and Bieker \cite{Bieker:IntModels}.

\subsection{Modifications of $\Gscr$-bundles}
 We consider the stack of $\Gscr$-bundles $\Bun_\Gscr$ over the curve $C$. Let $S$ be an $\f$-scheme, $T \subset C_S$ a finite subscheme over $S$, and $\Vscr,\Vscr' \in \Bun_\Gscr(S)$ two $\Gscr$-bundles over $C_S$. By definition, a $T$-modification of $\Vscr$ in $\Vscr$ is an isomorphism
 \[
  \phi\colon \restr{\Vscr}{C_S \setminus T} \isom \restr{\Vscr}{C_S \setminus T}.
 \]
 We consider the case that $T$ has a nice parametrisation over $S$ in the sense that $T = \bigcup_{i=1}^n \Gamma_{x_i}$ is the union of graphs of morphisms $x_i\colon S \to C$.
 
\begin{defn}\label{def-unbounded-shtukas}
 Let $\Ibf = (I_1,\dotsc,I_k)$ be a partition of a finite set $I$. We denote by $\Hecke_{\Gscr,\Ibf}$ the stack over $\f$ whose $S$-valued points are tuples $((x_i)_{i\in I},(\Vscr_j)_{j=0}^k, (\phi_j)_{j=1}^k)$ where
 \begin{bulletlist}
  \item $x_i \in C(S)$ for all $i \in I$ called legs,
  \item $\Vscr_j \in \Bun_\Gscr(S)$ for all $j=0,\dotsc,k$ and
  \item $\phi_j\colon \restr{\Vscr_{j}}{C_S \setminus \bigcup_{i \in I_j} \Gamma_{x_i}} \isom \restr{\Vscr_{j-1}}{C_S \setminus \bigcup_{i \in I_j} \Gamma_{x_i}}$ for all $j=1,\dots,k$.
 \end{bulletlist}

 We also define the \emph{stack of unbounded global $\Gscr$-shtukas} by \[\Xscr_{\Gscr,\Ibf} \coloneqq \ker(\sigma\circ\pr_0, \pr_k \colon \Hecke_{\Gscr,\Ibf} \rightrightarrows \Bun_\Gscr),\] 
 and an object of $\Xscr_{\Gscr,\Ibf}(S)$ is called a \emph{global $\Gscr$-shtuka} over $S$. More concretely, a $\Gscr$-shtuka above $S$ is a tuple $\underline\Vscr_\bullet=((x_i)_{i\in I},(\Vscr_j)_{j=0}^k, (\phi_j)_{j=0}^k)$ where
 \begin{bulletlist}
  \item $((x_i)_{i\in I},(\Vscr_j)_{j=0}^k, (\phi_j)_{j=1}^k)\in\Hecke_{\Gscr,\Ibf}(S)$ and
  \item $\phi_0 \colon \sigma^\ast \Vscr_0 \isom \Vscr_k$.
 \end{bulletlist}
Given a $\Gscr$-shtuka as above, we denote by $\tau_j\colon \sigma^\ast\Vscr_j \to \Vscr_j$ the composition $\phi_{j+1} \circ \dotsb \circ \phi_k \circ \phi_0 \circ \sigma^\ast \phi_1 \circ \dotsb \circ \sigma^\ast \phi_{j}$ for $j=0,\dotsc,k-1$.
\end{defn} 
 
By \cite[Thm.~3.15]{ArastehRad-Hartl:UnifStack} $\Xscr_{\Gscr,\Ibf}$ is a locally ind-finite-type ind-Deligne--Mumford stack. So in order to apply the results in \S\ref{sect-Igusa-Covers}, we need a systematic way to extract from $\Xscr_{\Gscr,\Ibf}$ a family of closed Deligne--Mumford substacks (not just ind-Deligne--Mumford substacks), which are interesting and useful so that the Igusa varieties that we construct over them are also meaningful. 

This leads to the theory of \emph{(global) bounds}. A local version of the theory was introduced by Arasteh~Rad and Hartl \cite[Def.~4.5]{ArastehRad-Hartl:LocGlShtuka}, from which various versions of global theory have been developed, such as Arasteh~Rad--Habibi \cite[Def.~3.1.3]{ArastehRad-Habibi:LocMod}, Bieker \cite[Def.~2.2.1]{Bieker:IntModels} and Hartl--Xu \cite[Def.~2.6.3]{HartlXu:Unif}. 

Unfortunately, the definition of global bounds is quite involved, and a complete review would require lengthy technical digressions. 
We will avoid reviewing the general theory by exclusively working with two specific classes of bounds in the sense of \cite[Def.~2.2.1]{Bieker:IntModels}, which we introduce below. 

\begin{notation}\label{not-BD}
Recall that the \emph{Beilinson--Drinfeld grassmannian} $\Gr_{\Gscr,\Ibf}$ is the fibre of $\pr_0\colon\Hecke_{\Gscr,\Ibf}\to\Bun_\Gscr$ at the $\FF_q$-point corresponding to the trivial $\Gscr$-torsor (\emph{cf.} \cite[Def.~2.1.1]{Bieker:IntModels}). It is an ind-scheme ind-quasi-projective over $C^I$ that comes with the action of the iterated global positive loop group $L^+_{C^I}\Gscr$ (\emph{cf.} \cite[Def.~2.1.3]{Bieker:IntModels}). 

We will consider the following two classes of $L^+_{C^I}\Gscr$-stable quasi-projective closed subschemes $Z$ of $\Gr_{\Gscr,\Ibf}\times_{C^I} \widetilde C^I_Z$, where $\widetilde C^I_Z$ is a (generically \'etale) finite cover of $C^I$ depending on $Z$; namely,
\begin{enumerate}
    \item \emph{Beilinson--Drinfeld Schubert varieties} (\emph{cf.} \revise{\cite[Def.~3.1.3]{ArastehRad-Habibi:LocMod} and} 
    \cite[Def.~2.4.1]{Bieker:IntModels}), and
    \item \emph{``Arasteh~Rad--Hartl bounds''} (\emph{cf.} \cite[\S3]{ArastehRad-Hartl:UnifStack}) and \cite[Ex.~2.4.2(4)]{Bieker:IntModels}).
\end{enumerate}

To introduce Beilinson--Drinfeld Schubert varieties, we need some notations.
Fix a maximal $\scl{\Fsf}$-torus $\Tsf\subset \Gsf$, and a Borel subgroup $\Bsf$ containing $\Tsf$. Let $X_\ast(\Tsf)_+$ denote the monoid of dominant coweights corresponding to the choice of $\Bsf$. For each $\lambda\in X_\ast(\Tsf)_+$, let $\Fsf_\lambda$ denote the \emph{reflex field}; that is, the field of definition of the geometric conjugacy class of $\Gsf$ containing $\lambda$. Let $\widetilde C_\lambda$ be the finite generically \'etale cover of $C$ corresponding to $\Fsf_\lambda$. (Note that $\Fsf_\lambda$ is contained in a finite extension $\widetilde\Fsf/
\Fsf$ splitting $\Gsf$. In particular, if $\Gsf$ is split then we have $\Fsf_\lambda=\Fsf$ for any $\lambda\in X_\ast(\Tsf)_+$.)

To $\blambda = (\lambda_i)_{i\in I} \in (X_\ast(\Tsf)_+)^I$, we associate a closed quasi-projective subscheme
\begin{equation}
    Z(\blambda) \subset \Gr_{\Gscr,\Ibf}\times_{C^I} \widetilde C^I_{Z(\blambda)}\qquad\big(\text{where }\widetilde C^I_{Z(\blambda)}\coloneqq \prod_{i\in I}\widetilde C_{\lambda_i}\big),
    ,
\end{equation}
which is stable under the action of $L^+_{C^I}\Gscr\times_{C^I}\widetilde C^I_Z$. We call $Z(\blambda)$ the \emph{Beilinson--Drinfeld Schubert variety} for $\blambda$. (See \cite[Def.~2.4.1]{Bieker:IntModels} for more details.)

To define Arasteh~Rad--Hartl bounds, we choose a faithful representation $\rho\colon \Gscr\mono \SL(\Vscr)$ for some rank-$N$ vector bundle $\Vscr$ on $C$, such that $\SL(\Vscr)/\Gscr$ is quasi-affine. (The existence of such a  $\rho$  follows from \cite[Prop.~2.2(b)]{ArastehRad-Hartl:UnifStack}.) To any $I$-tuple $\blambda^{\SL_N}=(\lambda^{\SL_N}_i)_{i\in I}$ of dominant coweights for $\SL_N$, we associate a closed quasi-projective subscheme
\begin{equation}
    Z_{\blambda^{\SL_N}} \subset \Gr_{\Gscr,\Ibf},
\end{equation}
which is stable under the $L^+_{C^I}\Gscr$-action. (See \cite[Ex.~2.4.2(4)]{Bieker:IntModels} for more details, where $Z_{\blambda^{\SL_N}}$ is denoted as $\Gr^{\Ibf,\lesssim\blambda^{\SL_N}}_{\Gscr,C}$.) We call $Z_{\blambda^{\SL_N}}$ an \emph{Arasteh~Rad--Hartl bound}, as it was originally introduced in \cite[\S3]{ArastehRad-Hartl:UnifStack}. Although its definition involves the auxiliary choice of $\rho$, we have  $\Gr_{\Gscr,\Ibf} = \bigcup_{\blambda^{\SL_N}}Z_{\blambda^{\SL_N}}$ in general (\emph{cf.} \cite[proof of Thm.~3.15]{ArastehRad-Hartl:UnifStack}). Note that the sign conventions for $\blambda^{\SL_N}$ in \cite[Ex.~2.4.2(2)]{Bieker:IntModels} and \cite[Eq.~(3.2)]{ArastehRad-Hartl:UnifStack} are opposite, and we follow the former.

If $Z = Z(\blambda)$ is a Beilinson--Drinfeld Schubert variety, then we let $\bpi\coloneqq (\pi_i)_{i\in I}\colon \widetilde C^I_Z\to C^I$ denote the natural projection, with $\pi_i\colon \widetilde C_{\lambda_i}\to C$ denoting the $i$th component of $\bpi$. If $Z=Z_{\blambda^{\SL_N}}$ is an Arasteh~Rad--Hartl bound, then we set $\widetilde C^I_Z = C^I$ and let $\bpi = (\pi_i)_{i\in I}\colon \widetilde C^I_Z\to C^I$ denote the identity map. 

By repeating the same construction as \cite[Def.~3.1.3]{Bieker:IntModels} for $\Hecke_{\Gscr,\Ibf}$ in place of $\Xscr_{\Gscr,\Ibf}$, a choice of $ Z $ as above defines a closed algebraic substack
\begin{equation}\label{eq-Bieker-bound}
\Hecke_{\Gscr,\Ibf}^{ Z } \subset \Hecke_{\Gscr,\Ibf}\times_{C^I}\widetilde C^I_Z.
\end{equation}
The construction uses an alternative description of $\Gr_{\Gscr,\Ibf}$ in terms of the Beauville--Laszlo descent; see Rmk.~2.1.2 and Def.~3.1.3 in \cite{Bieker:IntModels} for more details. We review the construction in Rmk.~\ref{rmk:bounds-via-BL}.
\end{notation}

\begin{rmk}
    A (global) bound is actually defined to be a ``compatible'' family of $L^+_{C^I}\Gscr$-stable closed subschemes of the Beilinson--Drinfeld grassmannian defined over various finite covers of $C^I$ (\emph{cf.} \cite[Def.~2.2.1]{Bieker:IntModels}). However, for the bound defined by a closed subscheme $Z\subset \Gr_{\Gscr,\Ibf}\times_{C^I}\widetilde C^I_Z$ that we consider, $\widetilde C^I_Z$ and $Z$ are respectively the \emph{reflex scheme} and the representative over it in the sense of \cite[Def.~2.6.3]{HartlXu:Unif}; this is clear for Arasteh~Rad--Hartl bounds $Z= Z_{\blambda^{\SL_N}}$ as $\widetilde C^I_Z = C^I$, and for a Beilinson--Drinfeld Schubert variety $Z = Z(\blambda)$ this claim follows from \cite[Lem.~2.3.5]{Bieker:IntModels}. For our purpose (namely, to prove Thm.~\ref{th-infinite-level-igusa-var}), we may just work with the representative $Z$ over $\widetilde C^I_Z$.
\end{rmk}

\begin{defn}\label{def-bounded-shtukas}
Given $ Z $ as above, we write
\begin{equation}\label{eq-bounded-shtukas}
    \Xscr_{\Gscr,\Ibf}^ Z  \coloneqq \ker(\sigma \circ\pr_0,  \pr_k \colon \Hecke_{\Gscr,\Ibf}^{ Z } \rightrightarrows \Bun_\Gscr),
\end{equation}
which is known to be a closed Deligne--Mumford substack of $\Xscr_{\Gscr,\Ibf}\times_{C^I}\widetilde C^I_Z$. (This follows from the same proof as \cite[Thm.~3.1.7]{ArastehRad-Habibi:LocMod}.) 
We call an object in $\Xscr_{\Gscr,\Ibf}^ Z (S)$ a \emph{$\Gscr$-shtuka over $S$ bounded by $ Z $}, which can be concretely described as a tuple $(\underline\Vscr_\bullet, (\widetilde x_i)_{i\in I})$ where $(\widetilde x_i)_{i\in I}\in\widetilde C^I_Z(S)$ and $\underline\Vscr_\bullet\in\Xscr_{\Gscr,\Ibf}(S)$ with legs at $(\pi_i(\widetilde x_i))_{i\in I}$, such that forgetting $\phi_0$ we get an $S$-point of $\Hecke_{\Gscr,\Ibf}^{ Z }$.
\end{defn}

Although we chose to exclusively work with certain classes of bounds, our argument extends to a broader class of bounds or even an alternative theory of bounds (such as \cite[Def.~2.6.3]{HartlXu:Unif}), provided that certain mild conditions for the proof of Thm.~\ref{th-infinite-level-igusa-var} are satisfied. See Rmk.~\ref{rmk-igusa-independence-of-choice} for further discussions.

 \subsection{The localisations of a global $\Gscr$-shtuka} \label{ssect-gl-loc}
 We fix a point $x \in |C|$, and set $\Gscr_x \coloneqq \Gscr_{O_x}$ where $O_x \subset \Fsf_x$ is the ring of integers. Given a global $\Gscr$-shtuka $\underline\Vscr_\bullet = ((x_i)_{i\in I},(\Vscr_j)_{j=0}^{k}, (\phi_j)_{j=0}^{k})$ \reviselong{such that each $x_i$ is either disjoint from $x$ or factors through $\Spf O_x$,} we define the \emph{localisation} of $\underline\Vscr_\bullet$ at $x$ as 
 \begin{equation}\label{eq-loc-bullet}
     \loc{\underline\Vscr_\bullet}{x} \coloneqq \restr{(({\Vscr_j})_{j=0}^{k}, (\phi_j)_{j=0}^{k})}{S \hat\times \Spf O_x}.
 \end{equation}
Similarly, for any $j = 0,\dotsc,k$ we write 
\begin{equation}\label{eq-loc-j}
    \loc{\underline\Vscr_j}{x} = (\loc{\Vscr_j}{x},\tau_{j,x}) \coloneqq \restr{(\Vscr_j,\tau_j)}{S \hat\times \Spf O_x},
\end{equation}
where $\tau_j$'s are introduced in Def.~\ref{def-unbounded-shtukas}.

 Following \cite[\S3.2]{neupert:thesis}, we can view $\loc{\underline\Vscr_j}{x}$ as a local $\Res_{O_x/\f\pot{\varpi_x}} \Gscr_x$-shtuka as follows. By \cite[Prop.~3.2.4]{neupert:thesis}, we have an equivalence of categories
 \[
   (\Gscr_x-{\rm Tors}/S \times_{\FF_q} \Spf O_x) \cong (L^+\Gscr_x-{\rm Tors}/S_{\FF_x})
 \]
 such that $\tau_{j,x}$ induces an isomorphism of $L\Gscr_x$-torsors $F_{S_{\FF_x}/\FF_x}^\ast \Lcal \loc{\Vscr_j}{x} \cong \Lcal \loc{\Vscr_j}{x}$, where $F_{S_{\FF_x}/\FF_x} = \Frob_S \times \id_{\FF_x}$ denotes the relative Frobenius. In order to apply the results from the previous section to $\loc{\underline\Vscr_j}{x}$, we repeat the proof of \cite[Lem.~3.2.1]{neupert:thesis} to obtain the following compatible equivalences of categories
 \begin{align*}
   (L^+\Gscr_x-{\rm Tors}/S_{\FF_x}) &\cong (L^+\Res_{O_x/\FF_x\pot{\unif}} \Gscr_x-{\rm Tors}/S) \quad\textnormal{and} \\
   (L\Gscr_x-{\rm Tors}/S_{\FF_x})& \cong (L\Res_{O_x/\FF_x\pot{\unif}} \Gscr_x-{\rm Tors}/S).
  \end{align*}
 Thus we can view $\loc{\underline\Vscr_j}{x}$ as a local $\Res_{O_x/\f\pot{\varpi_x}} \Gscr_x$-shtuka.

 Note that $\phi_j$ induces an isogeny $\loc{\underline\Vscr_{j}}{x} \to  \loc{\underline\Vscr_{j-1}}{x}$ for any $j$; i.e., the localisation $\loc{\underline\Vscr_\bullet}{x}$ could be thought of as an \emph{isogeny chain} of $\loc{\underline\Vscr_{j}}{x}$'s. If $x$ does not lie in the image of any leg, then $\loc{\underline\Vscr_{j}}{x}$ is \'etale and $\phi_j$ is an isomorphism for any $j$. 
 Now suppose that there is a unique $I_{j_0}$ such that for any $i\notin I_{j_0}$ the corresponding leg $x_i$ of $\underline\Vscr_\bullet$ is disjoint from $x$; i.e., $\varphi_j$ is an isomorphism for any $j\ne j_0$. Then we have either $\loc{\underline\Vscr_{j}}{x} \cong \loc{\underline\Vscr_{0}}{x}$ or $\loc{\underline\Vscr_{j}}{x} \cong \sigma^\ast\loc{\underline\Vscr_{0}}{x}$, so $\loc{\underline\Vscr_\bullet}{x}$ can be reconstructed from $\loc{\underline\Vscr_{0}}{x}$. In this case, we also call $\loc{\underline\Vscr_0}{x}$ the \emph{localisation} of $\underline\Vscr_\bullet$ at $x$. 
 
 \begin{rmk}\label{rmk-gl-loc}
  Of course, it would be more aesthetic if $\loc{\underline\Vscr_j}{x}$ were a local $\Gscr_x$-shtuka. This is not possible in general unless the following extra conditions are satisfied:
  \begin{enumerate}
      \item $S$ is a $\FF_x$-scheme, and
      \item at most one of the legs $(x_i)$ meets $x$.
  \end{enumerate}
  Indeed, in the favourable situation when both of these conditions are satisfied, the construction given in \cite[Def.~5.4]{ArastehRad-Hartl:LocGlShtuka} associates a local $\Gscr_x$-shtuka over $S$ to $\loc{\underline\Vscr_j}{x}$. By \cite[Rmk.~5.5]{ArastehRad-Hartl:LocGlShtuka} there exists a fully faithful functor from category  of local $\Gscr_x$-shtukas to the category of local $\Res_{ O _x/\f\pot{\varpi_x}} \Gscr_x$-shtukas, which transfers Arasteh Rad and Hartl's construction to Neupert's. 

  Let us explain in more details the special case when $S=\Spec\cl\FF_q$ and $k=1$. Let $y_1,\cdots,y_d\in C(\cl\FF_q)$ be the geometric points above $x$ with $d\coloneqq \deg(x)$, and we order them so that $y_j=\sigma^{j-1}\circ y_1$ for any $j=1,\dotsc,d$.  Choose an isomorphism
  \begin{equation}\label{eq-gl-loc-WeilRes}
    (L^+\Res_{O_x/\FF_q\pot{\unif}}\Gscr_x)_{\cl\FF_q} \cong (L^+\Gscr_x)_{\cl\FF_q}^d,
  \end{equation}
  so that the $j$th factor on the right hand side corresponds to $y_j$. Now consider a global $\Gscr$-shtuka $\underline\Vscr_\bullet\coloneqq((\overline x_i),\Vscr_0,\phi_0)$ over $\cl\FF_q$, and suppose that each $y_j$ appears at most once as a leg. Then the underlying $L^+\Res_{O_x/\FF_q\pot{\unif}}\Gscr_x$-torsor of $\loc{\underline\Vscr_0}{x}$ is trivial, and we have
  \begin{equation}\label{eq-gl-loc-b-tuple}
      \loc{\underline\Vscr_0}{x} \cong \big((L^+\Gscr_x)_S^{d}, \underline b\sigma \big) \quad\textnormal{where }\underline b = (b_1,b_2,\dotsc,b_d).
  \end{equation}
  If $y_j$ is \emph{not} a leg then we may even choose a trivialisation so that $b_j=1$. In particular, if we started with a $\Gscr$-shtuka such that $y_1$ is the only leg over $x$ then we may choose a trivialisation so that we have
  \begin{equation}\label{eq-gl-loc-b}
      \underline b = (b_1,1,\dotsc ,1),
  \end{equation}
where $b_1$ is uniquely determined up to $\Gscr_x(\breve O_x)$-$\sigma^{d}$ conjugacy. It turns out that the local $\Gscr_x$-shtuka constructed in \cite[Def.~5.4]{ArastehRad-Hartl:LocGlShtuka} is isomorphic to $((L^+\Gscr_x)_S,b_1\sigma^{d})$.
 \end{rmk}
  
  \subsection{Tate modules and level structure}
  In this article we use the notion of adelic level structure as defined in \cite[\S3.4]{neupert:thesis}. This works similar as in the case of (infinite-level) moduli spaces of abelian varieties, where one defines level structures by trivialising Tate modules. Recall that we defined the Tate module of an \'etale $G$-shtuka $\underline{\Gcal} = (\Gcal,\varphi)$ over a $\kappa_F$-scheme $S$ by $\Ig_{\underline{\Gcal}}^e = \underline{\Isom}((L^+G,\sigma), \underline\Gcal)$, which defines an equivalence of categories between \'etale $G$-shtukas and $G(O_F)$-torsors for the pro-\'etale topology by Prop.~\ref{prop-Igusa-etale}. Given a global shtuka $\underline\Vscr_\bullet = ((x_i)_{i\in I},(\Vscr_j)_{j=0}^k, (\phi_j)_{j=0}^k)$ and a closed point $x \in C \setminus \bigcup \image x_j$  we call $\Ig_{\loc{\underline\Vscr_{0}}{x}}^e$ the \emph{$x$-adic Tate-module} of $\underline\Vscr$. Since $\loc{\underline\Vscr_{0}}{x}$ is a local $\Res_{O_x/\FF_q\pot{\varpi_x}}\Gscr_x$-shtuka, $\Ig_{\loc{\underline\Vscr_{0}}{x}}^e$ is a $\Gscr(O_x)$-torsor.

  We continue to use the notation introduced in \S\ref{not-BD}, and write $\Ig^e_{\loc{\underline\Vscr^{\rm univ}}{x}}$ for the $x$-adic Tate module of the universal global $\Gscr$-shtuka over $\restr{\Xscr^ Z _{\Gscr,\Ibf}}{\bpi^{-1}((C\setminus \{x\})^I)}$.
 
 \begin{defn} \label{def-level-str-finite}
  Let $T =\{t_1,\dotsc,t_n\} \subset C$ be a finite set of closed points.
  \begin{subenv}
   \item We define the moduli stack of global $\Gscr$-shtukas with infinite level structure at $T$ by $\Xscr_{\Gscr,\Ibf,T}^ Z  \coloneqq \Ig^e_{\underline\Vscr^{\rm univ}[t_1^\infty]}\times_{\Xscr_{\Gscr,\Ibf}^ Z } \dotsb  \times_{\Xscr_{\Gscr,\Ibf}^ Z } \Ig^e_{\underline\Vscr^{\rm univ}[t_n^\infty]}$. Its $S$-valued points are tuples $(\underline\Vscr_\bullet,(\eta_t)_{t \in T})$, where
  \begin{bulletlist}
   \item $\underline\Vscr_\bullet \in \Xscr_{\Gscr,\Ibf}^ Z (S)$ is such that the legs are in $\bpi^{-1}((C\setminus T)^I)$,
   \item $\eta_t\colon (L^+\Res_{ O _x/\f\pot{\varpi_t}} \Gscr,\sigma) \isom \loc{\underline\Vscr_0}{t}$ is an isomorphism for $t \in T$.
  \end{bulletlist}
  \item Let $\Ksf \subset \Ksf_0 \coloneqq \prod_{x \in |C|} \Gscr(O_x)$ be a closed subgroup of the form $\Ksf = \prod_{x \not\in T} \Gscr(O_x) \times K_T$. We define the moduli stack of global $\Gscr$-shtukas with $\Ksf$-level structure as
  \[
   \Xscr_{\Gscr,\Ibf,\Ksf}^ Z  \coloneqq \Xscr_{\Gscr,\Ibf,T}^ Z  \sslash K_T,
  \]
  using the same GIT-quotient construction as in \S\ref{ssect-finite-level-igusa-cover}. 
  If $[\Ksf_0:\Ksf] < \infty$, we say that the level structure is finite.
  \end{subenv}
 \end{defn}
  
  \begin{rmk}
   \begin{subenv}
    \item Note that $\Xscr_{\Gscr,\Ibf,T}^ Z $ is a $\prod_{t\in T} \Gscr(O_t)$-torsor for the pro-\'etale topology on $\Xscr_{\Gscr,\Ibf}^ Z  \times_{C^I} (C\setminus T)^I$ by Prop.~\ref{prop-Igusa-etale}. Thus, $\Xscr_{\Gscr,\Ibf,\Ksf}^ Z $ is a $\Ksf_0/\Ksf$-fibration over $\Xscr_{\Gscr,\Ibf}^ Z  \times_{C^I} (C\setminus T)^I$. In particular, it is finite if the level structure is finite. 
    
    \item In  \cite{neupert:thesis} a slightly different language is used to describe the above moduli problems. Fix a connected $\f$-scheme $S$, a geometric point $\sbar$ and a global $\Gscr$-shtuka $\underline{\Vscr}_\bullet \in  \Xscr_{\Gscr,\Ibf}^ Z (S)$ whose legs are disjoint from $T$. For $t \in T$, we may consider the $\Gscr(O_t)$-torsor $\loc{\Vscr_0}{t}^{\varphi}$ as a continuous morphism $\pi_1(S,s) \to\Gscr(O_t) = \Aut \loc{\underline\Vscr_{0}}{t} $. Then the moduli stack of global $\Gscr$-shtukas with infinite level structure at $x$ parametrises global $\Gscr$-shtukas $\underline{\Vscr}_\bullet$ together with a $\pi_1(S,s)$-invariant point in $\prod_{t \in T} \loc{\Vscr_0}{t}_{\sbar}^{\varphi}$. One checks that $\Xscr_{\Gscr,\Ibf,\Ksf}^ Z $ parametrises  $\underline{\Vscr}_\bullet$ together with a $\pi_1(S,s)$-invariant $K_T$-orbit in $\prod_{t \in T} \loc{\Vscr_0}{t}_{\sbar}^{\varphi}$.
   \end{subenv}   
  \end{rmk}
  
  Another commonly used way to express moduli structure is given by Varshavsky's construction. Let $D \subset C$ be a finite subscheme and $ \underline{\Vscr}_\bullet$ a global $\Gscr$-shtuka over an $\f$-scheme $S$, whose legs are disjoint \revise{from}
  $D$. Then a level structure on $S$ is a compatible tuple of trivialisations $\restr{\Vscr_j}{D} \cong \Gscr_D$.  
  
 \begin{prop}
  Let $D$ be as above and let $\Ksf_D \coloneqq \ker(\prod_{x \in |C|} \Gscr(O_x)\epi \Gscr(O_D))$. Then  $\Xscr_{\Gscr,\Ibf,\Ksf_D}^ Z $ parametrises global $\Gscr$-shtukas bounded by $ Z $ with $D$-structure.
 \end{prop}  
 \begin{proof}
  This is proven in \cite[Thm.~6.5]{ArastehRad-Hartl:UnifStack}.
 \end{proof}

 \subsection{HN truncation}\label{ssect-HN} 
  In general, the moduli stack of $\Gscr$-bundles (and hence, moduli stacks of bounded $\Gscr$-shtukas) may not be quasi-compact. In order to be able to count points, we define a well-behaved filtration by open substacks which are quasi-compact modulo the action of the central torus. 
  
  Let $H$ be a split semisimple group over $\FF_p$, and we regard it as a constant group over $C$. Let $\Lambda^+$ denote the monoid of \emph{dominant} rational coweights of $H$, with the usual partial ordering given by positive coroots. Given any $\mu \in \Lambda^+$ we obtain an open algebraic substack $\Bun_H^{\leqslant \mu} \subset \Bun_H$ of finite type over $\FF_q$ which parametrises $H$-bundles whose HN-polygon is less than or equal to $\mu$ (\cite[Thm.~2.1]{Schieder:HN-str}). If $\Hscr$ is a reductive group over $C$ that is Zariski-locally isomorphic to $H$, then we have $\Bun_\Hscr\cong\Bun_H$. (The main example is when $H = \SL_N$ and $\Hscr = \SL(\Vscr)$ where $\Vscr$ is a rank-$N$ vector bundle on $C$.)

  Set $\Gscr^{\rm ad}\coloneqq \Gscr/\Zscr_\Gscr$ where $\Zscr_\Gscr$ is the flat closure of the centre $\Zsf_\Gsf$ of $\Gsf$. Note that $\Gscr^{\rm ad}$ is a smooth affine group scheme over $C$ with connected fibres and generic fibre $\Gsf^{\ad} = \Gsf/\Zsf_\Gsf$. By \cite[Prop.~2.2, Thm.~2.6]{ArastehRad-Hartl:UnifStack}, there exists a faithful representation $\rho\colon \Gscr^{\rm ad} \mono \Hscr$ for some semisimple group $\Hscr$ over $C$ that is Zariski-locally isomorphic to $H$, such that the induced morphism
\begin{equation}\label{eq:Functoriality-BunG}
    \Bun_{\Gscr^{\rm ad}} \xrightarrow{\rho_\ast} \Bun_{\Hscr} \cong \Bun_{H}
\end{equation}
   is representable, quasi-affine and of finite presentation. (Indeed, \emph{loc.~cit.} actually shows that we may take $H=\SL_N$ and $\rho\colon \Gscr^{\ad}\mono \SL(\Vscr)$ for some rank-$N$ vector bundle $\Vscr$ on $C$.) We fix any such an embedding $\rho$, and let $\Bun_{\Gscr^{\rm ad}}^{\leqslant\mu}$ denote the preimage of $\Bun_{H}^{\leqslant\mu}$ via \eqref{eq:Functoriality-BunG}. 
%
 
\begin{prop}[\emph{Cf.} {\cite{ArastehRad-Hartl:UnifStack}}]\label{prop:HN-qc}
The stack $\Bun_{\Gscr^{\ad}}^{\leqslant\mu}$ is quasi-compact.
Furthermore, for any bound $ Z^{\rm ad} $ for $\Gscr^{\ad}$ as in \S\ref{not-BD}, the stack $\Xscr^{ Z^{\rm ad} ,\leqslant\mu}_{\Gscr^{\ad},\Ibf}$ is quasi-compact.
\end{prop}
\begin{proof}
Note that the forgetful map $\Xscr^{ Z^{\rm ad} ,\leqslant\mu}_{\Gscr^{\ad},\Ibf} \to \Bun_{\Gscr^{\ad}}^{\leqslant\mu}$ is of finite type by \cite[Theorem~3.15]{ArastehRad-Hartl:UnifStack}, so it suffices to show that $\Bun_{\Gscr^{\ad}}^{\leqslant\mu}$ is quasi-compact. This is true as both the map $\Bun_{\Gscr^{\ad}}\to \Bun_H$ and $\Bun_H^{\leqslant\mu}$ are quasi-compact.
\end{proof}

 Let $\Bun_{\Gscr}^{\leqslant\mu}$ denote the preimage of $\Bun_{\Gscr^{\rm ad}}^{\leqslant\mu}$, which is an open substack of $\Bun_{\Gscr}$.  For any $Z\subset \Gr_{\Gscr,\Ibf}\times_{C^I}\widetilde C^I_Z$ as in \S\ref{not-BD}, we similarly define open substacks 
\[
\Xscr^{ Z ,\leqslant\mu}_{\Gscr,\Ibf}\subset\Xscr^{ Z }_{\Gscr,\Ibf}
\]
 by requiring $\Vscr_0\in\Bun _{\Gscr}^{\leqslant\mu}$. We call these open substacks the \emph{Harder--Narasimhan truncation} or the \emph{HN truncation}.

\begin{rmk}
%
Note that $\Bun_{\Gscr}^{\leqslant\mu}$ and $\Xscr^{ Z ,\leqslant\mu}_{\Gscr,\Ibf}$ may not be quasi-compact in general, unless the centre $\Zsf_\Gsf$ of $\Gsf$ is totally anisotropic, since the natural map $\Bun_{\Gscr}^{\leqslant\mu}\to\Bun_{\Gscr^{\rm ad}}^{\leqslant\mu}$ may not be quasi-compact. Indeed, $\Bun_{\Gscr}^{\leqslant\mu}$ and $\Xscr^{ Z }_{\Gscr,\Ibf}$ may have infinitely many connected components in general. We define the HN truncations via $\Bun_{\Gscr^{\rm ad}}$ so that the resulting open substacks are stable under the action of $\Zsf_\Gsf(\Fsf)\backslash\Zsf_\Gsf(\AA)$ described in \S\ref{ssect-Xi}. We will also show in \S\ref{ssect-Xi} that the quotient of $\Xscr^{ Z ,\leqslant\mu}_{\Gscr,\Ibf}$ by some suitable discrete subgroup of $\Zsf_\Gsf(\Fsf)\backslash\Zsf_\Gsf(\AA)$ is quasi-compact.
\end{rmk}

 \begin{rmk}
     If $\Gscr$ is a split reductive group, then we may set $H = \Gscr^{\rm ad}$ and our definition of HN truncation coincides with \cite[\S2.1]{Varshavsky:shtuka}. However, for general $\Gscr$ there may not be a canonical or preferred choice of $\rho$ in \eqref{eq:Functoriality-BunG}. 
\end{rmk}
 
\begin{prop}\label{prop-schematic}
    Given any $\mu\in\Lambda^+$, the moduli stack $\Xscr^{ Z ,\leqslant\mu}_{\Gscr,\Ibf,\Ksf_D}$ can be represented by a scheme if $|D|$ is sufficiently large compared to $\mu$.
    In particular, $\Xscr^ Z _{\Gscr,\Ibf,T}$ can be represented by a scheme for any non-empty finite subset $T\subset |C|$. 
\end{prop}
\begin{proof}
The second claim easily follows from the first, so let us prove the first claim.

If $\Gscr=H$ is a split reductive group, then the proposition is proved in \cite[Prop.~2.16 a)]{Varshavsky:shtuka}.
For general $\Gscr$, let us choose an embedding $\rho\colon\Gscr^{\ad}\mono \Hscr$ as before (so we also get $\rho\colon\Gsf^{\ad}\mono H_{\Fsf}$), and an Arasteh~Rad--Hartl bound $ Z '$ for $H$ (in the sense of \S\ref{not-BD}) so that the natural map $\Bun_\Gscr\to\Bun_H$ induces $\Xscr^{ Z }_{\Gscr,\Ibf}\to \Xscr^{ Z '}_{H,\Ibf}$. Such $Z'$ exists by the ind-stack structure of $\Xscr_{H,\Ibf}$ obtained in \cite[Thm.~3.15]{ArastehRad-Hartl:UnifStack}.

For any finite non-empty subscheme $D\subset C$, we set $\Ksf'_D\coloneqq\ker(\prod_{x \in |C|} H(O_x)\to H(O_D))$. Then the natural map 
\[
\Xscr^{ Z ,\leqslant\mu}_{\Gscr,\Ibf,\Ksf_D} \to  \Xscr^{ Z ',\leqslant\mu}_{H,\Ibf,\Ksf'_D}\times_{\Xscr^{ Z '\leqslant\mu}_{H,\Ibf}}\Xscr^{ Z ,\leqslant\mu}_{\Gscr,\Ibf}
\]
is representable by a \emph{finite} morphism of schemes as it corresponds to pushing forward a $\Gscr(O_D)$-torsor to a $H(O_D)$-torsor. Therefore, $\Xscr^{ Z ,\leqslant\mu}_{\Gscr,\Ibf,\Ksf_D}$ is representable by a scheme if $\Xscr^{ Z ',\leqslant\mu}_{H,\Ibf,\Ksf'_D}$ is representable by a scheme, which follows if $|D|$ is large enough compared to $\mu$.
\end{proof}
 
 \section{Global $\Gscr$-shtukas with full level structure} \label{sect-full-level}
 
  In the previous section, we defined the moduli space of global $\Gscr$-shtukas over $C\setminus T$ with infinite level structure at $T$ when $T$ is finite. As a next step, we extend this definition to infinite sets $T$, especially when $|C|\setminus T$ is finite. We then construct the \emph{infinite-level Igusa varieties} and establish basic properties; \emph{cf.} Thm.~\ref{th-infinite-level-igusa-var}.
  
 \subsection{Level structure at infinitely many places} \label{ssect-infinite-level-structure}
 We fix a \emph{non-empty} subset $T \subseteq |C|$  such that the locally ringed space $(C\setminus T, \restr{\Oscr_C}{C\setminus T})$ is a scheme.  (This is automatically satisfied when $T$ is finite or cofinite.) We denote by $\Tbar$ the complement of $T$  in $|C|$. We are mostly interested in the case where $\Tbar$ is a fixed set of legs in the moduli problem of global $\Gscr$-shtukas.

 \begin{defn}\label{df-infinite-level-structure}
   We define the moduli space of global $\Gscr$-shtukas with infinite level structure at $T$ as the infinite fibre product over ${\Xscr_{\Gscr,\Ibf}^ Z }$
   \[
    \Xscr_{\Gscr,\Ibf,T}^ Z  \coloneqq \prod_{t \in T} \Ig^e_{\loc{\underline\Vscr_0^{\rm univ}}{t}},
   \]
   which solves a moduli problem analogous to Def.~\ref{def-level-str-finite}. 
 \end{defn}
 We claim that $ \Xscr_{\Gscr,\Ibf,T}^ Z$ can be represented by a scheme and the natural projection $ \Xscr_{\Gscr,\Ibf,T}^ Z \to  \Xscr_{\Gscr,\Ibf}^ Z$ can be represented by an affine morphism. In fact, choosing $t_0\in T$, the definition of $ \Xscr_{\Gscr,\Ibf,T}^ Z$ can be rewritten as the following infinite fibre product of \emph{schemes} over $\Ig^e_{\loc{\underline\Vscr_0^{\rm univ}}{t_0}}$
 \[
    \prod_{t \in T\setminus\{t_0\}} (\Ig^e_{\loc{\underline\Vscr_0^{\rm univ}}{t}}\times_{\Xscr_{\Gscr,\Ibf}^ Z } \Ig^e_{\loc{\underline\Vscr_0^{\rm univ}}{t_0}}),
 \]
 where each direct factor is \emph{affine} over $\Ig^e_{\loc{\underline\Vscr_0^{\rm univ}}{t_0}}$. (Recall that Igusa covers are schemes that are affine over ${\Xscr_{\Gscr,\Ibf}^ Z }$.) This shows that $ \Xscr_{\Gscr,\Ibf,T}^ Z$ is a scheme that is affine over $\Ig^e_{\loc{\underline\Vscr_0^{\rm univ}}{t_0}}$ by \cite[\href{https://stacks.math.columbia.edu/tag/0CNH}{Tag 0CNH}]{AlgStackProj}, and thus it is affine over $\Xscr^Z_{\Gscr,\Ibf}$ as well. 
 
 Next we aim to show that the infinite level structure negates every choice we had to make at the places in $T$. This requires some preparation; namely, an adelic version of the Beauville--Laszlo descent for $\Bun_\Gscr$ (\emph{cf.} Prop.~\ref{prop-adelic-BL}).
 
 \subsection{Classical Beauville-Laszlo descent} 
We begin with  the classical Beauville--Laszlo descent for $\Bun_\Gscr$ (Cor.~\ref{cor-lemma-5.1}) as a preparation for its adelic version. Although there are similar results in the literature such as \cite[Thm.~2.3.4]{BeilinsonDrinfeld:Hitchin} and \cite[Lem.~5.1]{ArastehRad-Hartl:LocGlShtuka}, we are not aware of the result that we need for the proof of Prop.~\ref{prop-adelic-BL}. (See Rmk.~\ref{rmk-lemma-5.1} for more details.) So we provide full details although they must be well known to experts.

 We consider the following general setup. Let $X$ be an affine scheme and let $D \subset X$ be an effective Cartier divisor.   We denote by $\Dhat = (D,\hat\Oscr_{X,D})$ the completion of $D$ in $X$. It is a famous result of Beauville and Laszlo that the category of vector bundles over $X$ is equivalent to the category of triples $(\Escr',\Ecal,\varphi)$, where $\Escr'$ and $\Ecal$ are vector bundles over $X \setminus D$ and $\Spec\hat\Oscr_{X,D}$, respectively, and $\varphi\colon \restr{\Escr'}{\Spec\hat\Oscr_{X,D} \setminus D} \isom \restr{\Ecal}{\Spec\hat\Oscr_{X,D} \setminus D}$ is an isomorphism. In order to relate a global shtuka to its localisations, we have to ``sheafify the result over $\Dhat$'' so that the resulting sheaves naturally admit an interpretation in terms of torsors for some loop group and positive loop group; \emph{cf.} \eqref{eq-formal-and-loop-torsor}.

Now let $X$ be a scheme. We consider the set $\Bfr$ of all affine open $U = \Spec R \subset X$ such that $D \cap U = V(\unif)$ for a regular element $\unif \in R$ and set $\Bfr_D \coloneqq \{ U \cap D \mid U \in \Bfr\}$. Note that $\Bfr$ and hence $\Bfr_D$ is a basis of topology, hence the restriction defines an equivalence of categories between sheaves on $D$ and sheaves on $\Bfr_D$. (See e.g.~\cite[\href{https://stacks.math.columbia.edu/tag/009O}{Lem.~009O}]{AlgStackProj}.) We denote $\hat\Oscr_{X,D}^\circ$
the sheaf on $D$ such that for $V = \Spec R \cap D \in \Bfr_D$ we have
\[
 \hat\Oscr_{X,D}^\circ(V) \coloneqq \hat\Oscr_{X,D}(V)[\unif\iv] = \hat{R}[\unif\iv],
\]
 where $\hat R$ is the $\unif$-adic completion of $R$. Since localisation is exact this construction indeed defines a sheaf on $\Bfr_D$, and hence on $D$.

\begin{defn} \label{def-analytification}

 The topologically ringed space $\Dhat^\circ\coloneqq (D,\hat\Oscr_{X,D}^\circ)$ is called the \emph{punctured formal neighbourhood} of $D$ in $X$.  Moreover, for any vector bundle $\Fscr$ on $ X \setminus D$ we define  its \emph{analytification} $\Fscr^{\rm an}$ as the sheaf on $\Dhat^\circ$ obtained by extending the following sheaf on $\Bfr_D$:
 \begin{equation} \label{eq-vb}
 (\Spec R) \cap D \mapsto \Fscr(\hat{R}[\unif\iv]), \quad \text{for any }\Spec R\in \Bfr.
 \end{equation}
\end{defn}

 Alternatively, we can define $\Fscr^{\rm an}$ as follows. Since the analytification is local with respect to Zariski open covers of $X$ in $\Bfr$, we restrict to the case that $X = \Spec R$ is affine and that $D=V(\unif)$ for a regular element $\unif \in R$. Thus $\Dhat = \Spf \Rhat$, where $\Rhat$ denotes the $\unif$-adic completion of $R$ and $X \setminus D = \Spec R[\unif^{-1}]$.  We denote by $M$ the finite projective $R[\unif^{-1}]$-module corresponding to $\Fscr$. Then
  \begin{equation}\label{eq-vb-local}
   \Fscr^{\rm an} = M \hat\otimes_{R[\unif\iv]} \hat\Oscr_{X,D}^\circ.
  \end{equation}
  Note that the tensor product of analytified vector bundles can be taken in the category of presheaves (i.e.\ without sheafifying) since the tensor product preserves the sheaf property of $\hat\Oscr_{X,D}^\circ$ by flatness of $M$.

 The following lemma allows us to replace the glueing data over $\Spec \Rhat[\unif\iv]$ in \cite{Beauville-Laszlo:DescentLemma} by glueing data of $\hat\Oscr^{\circ}_{X,D}$-modules. The advantage is that the latter globalises with respect to the Zariski topology on $D$.
 
 \begin{lem}\label{lem-46}
  Let $X=\Spec R$ be an affine scheme, and set $D = V(\unif) \subset X$ for a regular element $\varpi \in R$. Then $M \mapsto \widetilde{M}^{\rm an} \coloneqq M \hat\otimes_{\Rhat[\unif\iv]} \hat\Oscr_{X,D}^\circ$ is a fully faithful functor from the category of finite projective modules over $\Rhat[\unif\iv]$ to the category of flat $\hat\Oscr_{X,D}^\circ$-modules.
 \end{lem}
 \begin{proof}
  Let $M,N$ be two finite projective $\Rhat[\unif\iv]$-modules. We want to show that the natural morphism   $   \Hom(M,N) \to \Hom(\widetilde{M}^{\rm an},\widetilde{N}^{\rm an})  $ is bijective. Injectivity is clear from considering global sections. To show surjectivity, note that the restriction of global sections of $\widetilde M^{\rm an}$ to an open $U \subset D$ is given by 
  \[
   M = \widetilde{M}^{\rm an}(D) \to \widetilde{M}^{\rm an}(U) = M \otimes_{\Rhat[\unif\iv]} \hat\Oscr_{X,D}^\circ(U);
  \]
  \emph{cf.}\ \eqref{eq-vb-local}.
   Hence $\widetilde{M}^{\rm an}(U)$ is generated by global sections, and each morphism $\widetilde{M}^{\rm an} \to \widetilde{N}^{\rm an}$ is uniquely determined by its effect on the global sections. Therefore the natural map $\Hom(M,N) \to \Hom(\widetilde{M}^{\rm an},\widetilde{N}^{\rm an})$ is an isomorphism.
 \end{proof}
 
  Now let $\Escr$ be a vector bundle over $X$. We can associate  two sheaves of $\hat\Oscr_{X,D}^\circ$-modules to it; namely, $(\restr{\Escr}{\Dhat})[\unif^{-1}]\coloneqq\restr{\Escr}{\Dhat}\otimes_{\hat\Oscr_{X,D}}\hat\Oscr_{X,D}^{\circ}$ and $(\restr{\Escr}{X \setminus D})^{\rm an}$. We also get a natural isomorphism $(\restr{\Escr}{X \setminus D})^{\rm an} \isom (\restr{\Escr}{\Dhat})[\unif^{-1}]$ using the description in \eqref{eq-vb-local}. 
 Hence, we can associate to each vector bundle $\Escr$ over $X$ the triple $(\restr{\Escr}{X \setminus D},\restr{\Escr}{\Dhat},\varphi_\Escr)$, where $\varphi_\Escr$ is the canonical isomorphism above. Altogether, we obtain the following variant of the Beauville--Laszlo descent lemma.

 \begin{prop}[Beauville--Laszlo] \label{prop-BL}
  The functor $\Escr\mapsto (\restr{\Escr}{X \setminus D},\restr{\Escr}{\Dhat},\varphi_\Escr)$ defines an equivalence of categories between the category of vector bundles on $X$ and the category of triples $ (\Escr',\Ecal,\varphi)$ where $\Escr',\Ecal$ are vector bundles over $X \setminus D$ and $\Dhat$, respectively, and $\varphi\colon\Escr'^{\rm an}\to \Ecal[\unif^{-1}] $ is an isomorphism.
 \end{prop}
 \begin{proof}
  Since the construction of the functor commutes with the restriction morphisms for opens in $\Bfr$, we may assume that $X = \Spec R$ is affine and $D = V(\unif)$ for a regular element $\unif \in R$. Since for a finite $\Rhat$-module being locally free with respect to the Zariski topology on $\Spf \Rhat$ is the same as being locally free with respect the Zariski topology on $\Spec \Rhat$ and $\varphi$ corresponds to an isomorphism of $\Rhat[\unif\iv]$-modules by Lem.~\ref{lem-46} above, this assertion is identical to \cite[\S3,~Th\'eor\`eme]{Beauville-Laszlo:DescentLemma}.
 \end{proof}
 
 We now apply Prop.~\ref{prop-BL} to obtain a local description of the stack of $\Gscr$-bundles (Cor.~\ref{cor-BL-for-torsors}). Set $X = C_S$ and $D = \{x\} \times S \cong S_{\FF_x}$ for $x \in |X|$ and some $\FF_q$-scheme $S$. In particular, we get $\hat\Oscr_{X,D}(U) = (\Oscr_D(U))\pot{\unif_x}$ for every open subset $U \subset D$, and hence we have $\hat\Oscr_{X,D}^\circ(U) = (\Oscr_D(U))\rpot{\unif_x}$ whenever $U$ is quasi-compact. We define the analytification of a given $\Gscr$-torsor $\Vscr'$ on $X \setminus D$ as a sheaf on the \'etale site of $D$ given by 
 \[
  \Vscr'^{\rm an}\colon \Spec R \mapsto \Vscr'(R\rpot{\unif_x}).
 \]
 To check that this is indeed a sheaf, recall that by \cite[Thm.~1.1]{Broshi:Gtorsors} there exists a vector bundle $V$ over $C$ and a line bundle $L \subset V^\otimes$ such that $\Gscr = \underline \Aut(V,L)$. (See \cite[Rmk.~A.19]{ImaiKatoYoucis:Tannaka} for a minor correction.) Thus we get an equivalence between the stack of $\Gscr$-torsors and the stack of $(V,L)$-twists by \cite[Cor.~1.4]{Broshi:Gtorsors}. Here a $(V,L)$-twist means a vector bundle $\Escr$ together with a line bundle $\Lscr\subset\Escr^\otimes$ such that $(\Escr,\Lscr)$ is \'etale-locally isomorphic to $(V,L)$. Now let $(\Escr',\Lscr')$ be the $(V,L)$-twist associated to a $\Gscr$-torsor $\Vscr'$ over $X\setminus D$, i.e.\ $\Vscr' = \underline\Isom((V,L),(\Escr',\Lscr'))$. Then the presheaf
 \[
  \Vscr'^{\rm an} = \underline\Isom((V^{\rm an},L^{\rm an}),(\Escr'^{\rm an},\Lscr'^{\rm an})),
 \]
  is a sheaf, which can be seen from the description of the analytification of vector bundles; see \eqref{eq-vb-local} and the subsequent discussion.

 \begin{cor} \label{cor-BL-for-torsors}
  In the situation above, the groupoid $\Bun_\Gscr(S)$ of $\Gscr$-torsors over $X$ is equivalent to the groupoid of triples $(\Vscr',\Vcal,\varphi)$, where $\Vscr'$ and $\Vcal$ are $\Gscr$-torsors over $X \setminus D$ and $\Dhat$, respectively, and $\varphi\colon \Vscr'^{\rm an} \to (\Gscr^{\rm an} \times^{{\Gscr}|_{\Dhat}} \Vcal)$ is an isomorphism.
 \end{cor} 
 \begin{proof}
   Given a triple $(\Vscr',\Vcal,\varphi)$ as above we may glue the associated $(V,L)$-twists to a (unique) pair $(\Escr,\Lscr)$ over $X$ by Prop.~\ref{prop-BL}. It remains to show that this is again a $(V,L)$-twist, or equivalently that
   \[
    \Vscr \coloneqq \underline\Isom((V,L),(\Escr,\Lscr))
   \]
   is an fppf cover of $X$. There exists a finitely presented $\FF_q$-scheme $S_0$ over $S$ such that $\Escr$ and $\Lscr$ are defined over $S_0$. Replacing $S$ by $S_0$, we may assume that $S$ is noetherian. Since $\Vscr \subset \underline\Isom(V,\Escr)$ is a closed subscheme, $\Vscr$ is of finite presentation over $X$. Moreover, since $\restr{\Vscr}{X\setminus D}$ is flat over $X$ it remains to show that $\restr{\Vscr}{\Spec \Oscr_{C,x} \times S}$ is flat. Since $\restr{\Vscr}{\Dhat}$ is flat, this follows by the local flatness criterion. (See eg.\ \cite[Thm.~22.3]{Matsumura:crt}.)
 \end{proof}
 
 In the setup of Cor.~\ref{cor-BL-for-torsors} above, note that 
 $\restr{\Gscr}{\Dhat} \cong (L^+\Gscr_x)_{S_{\FF_x}}$ and $\Gscr^{\rm an} \cong (L\Gscr_x)_{S_{\FF_x}}$. We apply \cite[3.2.1,~3.2.4]{neupert:thesis} to obtain equivalences of categories
  \begin{align}\label{eq-formal-and-loop-torsor}
   (\Gscr-{\rm Tors}/\Dhat) \cong (L^+\Res_{O_x/\f\pot{\unif_x}} \Gscr_x-{\rm Tors}/S),\ \text{and}  \\
   (\Gscr^{\rm an}-{\rm Tors}/\Dhat^{\circ}) \cong (L\Res_{\Fsf_x/\f\rpot{\unif_x}} \Gscr_x-{\rm Tors}/S).\notag
  \end{align}

  For any $\FF_q$-scheme $S$, let $\Bun_{{\Gscr}|_{C \setminus \{x\}}}(S)$ denote the groupoid of $\Gscr$-torsors over $(C\setminus\{x\})_S$.
  Given $\Vscr' \in\Bun_{{\Gscr}|_{C \setminus \{x\}}}(S) $, the above equivalences identify $\Vscr'^{\rm an}$ with the sheaf $\Vscr_x'^{\rm an} \colon \Spec R \mapsto \Vscr'(R\rpot{\unif_x})$, and $\Gscr^{\rm an} \times_{{\Gscr}_{\Dhat}} \Vcal$ with $\Lcal\Vcal$, essentially by tracing through the definitions. Altogether, we obtain
 
 \begin{cor} \label{cor-lemma-5.1}
  Let $S$ be an $\FF_q$-scheme and $x \in |C|$. The groupoid $\Bun_\Gscr(S)$ is equivalent to the groupoid of triples $(\Vscr',\Vcal,\varphi)$, where $\Vscr' \in\Bun_{{\Gscr}|_{C \setminus \{x\}}}(S) $, $\Vcal$ is a $L^+\Res_{O_x/\f\pot{\unif_x}}\Gscr_x$- torsor over $S$  and $\varphi\colon \Vscr_x'^{\rm an} \to \Lcal\Vcal$ is an isomorphism.
 \end{cor}
 
 \begin{rmk}\label{rmk-lemma-5.1}
  Our result is a slight improvement of \cite[Lem.~5.1]{ArastehRad-Hartl:LocGlShtuka}, where they assumed a priori that $\Vscr'$ extends to $X$. Our construction of $\Vscr_x'^{\rm an}$ is different from \emph{loc.~cit}, but yields the same result. Note that in the proof of Prop.~\ref{prop-adelic-BL} we apply Cor.~\ref{cor-lemma-5.1} to a triple  $(\Vscr',\Vcal,\varphi)$ where $\Vscr'$ is not a priori known to extend to $C_S$.
 \end{rmk}

 \begin{rmk}\label{rmk:bounds-via-BL}
     Let $X = C_S$ for an $\FF_q$-scheme $S$, and consider the divisor $D=\bigcup_{i\in I}\Gamma_{x_i}$ where $x_i\in C(S)$ for each $i\in I$. Then, one can similarly define the analytification $(-)^{\rm an}$ for $\Gscr$-torsors over $C_S\setminus D$, and the statement of Cor.~\ref{cor-BL-for-torsors} also holds for $\Bun_\Gscr(S)$ with respect to this choice of $D$ by the identical proof. (Note that we have $\restr{\Gscr}{\Dhat}\cong L^+_{C^I}\Gscr\times_{C^I}S$ and $\Gscr^{\rm an}\cong L_{C^I}\Gscr\times_{C^I}S$, where $L^+_{C^I}\Gscr$ and $L_{C^I}\Gscr$ are defined in \cite[Def.~2.1.3]{Bieker:IntModels}.)

     As a consequence, $\Gr_{\Gscr,\Ibf}(S)$ is equivalent to the category of following tuples
     \begin{equation}\label{eq:BD-Gr-via-BL}
       \underline\Vcal_\bullet \coloneqq \big( (x_i)_{i\in I}, (\Vcal_j)_{j=0}^k, (\varphi_j)_{j=1}^k, \hat\epsilon \big) ,
     \end{equation}
    where $x_i\in C(S)$, $\Vcal_j$ is a $\Gscr$-torsor over $\hat D$ with $D = \bigcup_{i\in I}\Gamma_{x_i}$, $\varphi_j\colon \Gscr^{\rm an} \isom \Gscr^{\rm an}\times^{\Gscr|_{\hat D}}\Vcal_j$ is a trivialisation of the  $\Gscr^{\rm an}$-torsor induced by $\Vcal_j$ (for $j\ne 0$), and $\hat\epsilon \colon \Vcal_0 \isom \restr{\Gscr}{\Dhat}$ is a trivialisation of $\Vcal_0$. In fact, given $(\underline\Vscr_\bullet,\epsilon)\in\Gr_{\Gscr,\Ibf}(S)$ where
 \[
\underline{\Vscr}_\bullet\coloneqq ((x_i)_{i\in I},(\Vscr_j)_{j=0}^k, (\phi_j)_{j=1}^k) \in \Hecke_{\Gscr,\Ibf}(S)
\] 
and $\epsilon\colon \Vscr_0\isom \Gscr_{C_S}$, we obtain an object $\underline\Vcal_\bullet$ as in \eqref{eq:BD-Gr-via-BL} by the ``restriction to $\Dhat$''; namely, we set $\Vcal_j\coloneqq \restr{\Vscr_j}{\hat D}$, $\varphi_j \coloneqq (\phi_j)^{\rm an}$ and $\hat\epsilon\coloneqq \restr{\epsilon}{\Dhat}$. Its quasi-inverse is given by the Beauville--Laszlo descent for $\Bun_\Gscr(S)$ by setting $\Vscr_j' = \Gscr_{C_S\setminus D}$ for each $j$. (Note that the Beauville--Laszlo descent obtained in \cite[Lem.~5.1]{ArastehRad-Hartl:LocGlShtuka} suffices for this purpose, as $\Gscr_{C_S\setminus D}$ is clearly extendable to $C_S$.)

Now, let us fix a bound $Z\subset \Gr_{\Gscr,\Ibf}\times_{C^I}\widetilde C^I_Z$ as in \S\ref{not-BD}, and consider $(\underline\Vscr_\bullet,(\widetilde x_i)_{i\in I})\in (\Hecke_{\Gscr,\Ibf}\times_{C^I}\widetilde C^I_Z)(S)$; i.e. $(\widetilde x_i)_{i\in I}$ is an $S$-point of $ \widetilde C^I_Z$ that projects onto the legs of $\underline\Vscr_\bullet$. Then, whether $(\underline\Vscr_\bullet,(\widetilde x_i)_{i\in I})$ lies in $\Hecke^Z_{\Gscr,\Ibf}(S)$ is determined by its ``restriction to $\Dhat$''. In fact, the restriction of $\underline\Vscr_\bullet$ to $\Dhat$ almost defines a tuple $\underline\Vcal_\bullet$ as in \eqref{eq:BD-Gr-via-BL} except $\hat\epsilon$. By replacing $S$ by some \'etale covering $S'$ however, one can trivialise $\restr{\Vscr_0}{\Dhat}$, and the choices of $\hat\epsilon$ form a principal homogeneous spaces for $L^+_{C^I}\Gscr(S')$; \emph{cf.} \cite[Lem.~3.4]{HainesRicharz:TestFtnRes}. Thus we get an $S'$-point in $\Gr_{\Gscr,\Ibf}\times_{C^I}\widetilde C^I_Z$ up to $L^+_{C^I}\Gscr$-action, and if this $L^+_{C^I}\Gscr$-orbit lies in $Z$ then we say that $(\underline\Vscr_\bullet,(\widetilde x_i)_{i\in I})\in\Hecke^Z_{\Gscr,\Ibf}(S)$.
 \end{rmk}

 \subsection{Adelic loop groups}
  
  With $T$ as in \S\ref{ssect-infinite-level-structure}, we define the ad\`eles (away from $\Tbar$) as the restricted product
 \[
  \AA^{\Tbar} \coloneqq  \sideset{}{'}\prod_{x \in T} (\Fsf_x,O_x),
 \]  
 which contains $\hat O_T\coloneqq \prod_{x\in T}O_x$ as an open compact subring. For any $\FF_q$-algebra $R$, we set
 \[\AA^{\Tbar}(R)\coloneqq (R\hat\otimes_{\FF_q}\hat O_T )\otimes_{\hat O_T}\AA^{\Tbar}.\]

  We define the adelic loop group as fpqc-sheaf on $\FF_q$-schemes given by
 \begin{align*}
  L_{\AA^{\Tbar}}\Gsf(R) &\coloneqq \sideset{}{'}\prod_{x \in T} (L\Res_{O_x/\f\pot{\unif_x}} \Gscr_x(R),L^+\Res_{O_x/\f\pot{\unif_x}} \Gscr_x(R)).
  \shortintertext{As the notation suggests, the sheaf $ L_{\AA^{\Tbar}}\Gsf$ depends only on $\Gsf$ as}
   L_{\AA^{\Tbar}}\Gsf(R) &= \sideset{}{'}\prod_{x \in T} (\Gscr(R \hat\otimes_\f \Fsf_x), \Gscr(R \hat\otimes_\f O_x)) \\
   &= \Gscr(\sideset{}{'}\prod_{x \in T}(R \hat\otimes_\f \Fsf_x, R \hat\otimes_\f O_x)) \\
   &= \Gsf(\AA^{\Tbar}(R)),
   \shortintertext{where $R \hat\otimes_\f \Fsf_x$ is a short-hand notation for $(R \hat\otimes_\f O_x)\otimes_{O_x}\Fsf_x$. Moreover, $L_{\AA^{\Tbar}}\Gsf$ is represented by an ind-scheme as it can be written as}
    L_{\AA^{\Tbar}}\Gsf &= \varinjlim_{T_0 \subset T \textnormal{ finite}} (\prod_{x \in T_0} L\Res_{O_x/\f\pot{\unif_x}} \Gscr_x \times \prod_{x \in T \setminus T_0} L^+\Res_{O_x/\f\pot{\unif_x}} \Gscr_x).
 \end{align*}
  
  We now generalise the notion of Beauville-Laszlo descent data as follows. Fix an $\f$-scheme $S$. For a $\Gsf$-torsor $\Vscr'$ over $(C\setminus T)_S$ we define a pseudo-$L_{\AA^{\Tbar}}\Gsf$-torsor $\Vscr_{\AA^{\Tbar}}'^{\rm an}$ on $S$ by 
 \[
  \Vscr_{\AA^{\Tbar}}'^{\rm an}\colon \Spec R \mapsto \Vscr'(\AA^{\Tbar}(R)).
 \]
 And given a family $(\Vcal_t)_{t \in T}$ of $L^+\Gscr_t$-torsors over $S$ we define a $L_{\AA^{\Tbar}}\Gsf$-torsor 
 \[\Lcal_{\AA^{\Tbar}}((\Vcal_t)_{t\in T}) \coloneqq L_{\AA^{\Tbar}}\Gsf \times^{\prod_{t\in T} L^+\Gscr_t} \prod_{t\in T} \Vcal_t.\] 
Let $\Bun_{\Gscr,T}(S)$ be the groupoid of triples $(\Vscr',(\Vcal_t)_{t \in T},\varphi)$ where $\Vscr' \in \Bun_{{\Gscr}|_{C \setminus T}} (S)$, each $\Vcal_t$ is an $L^+\Gscr_t$-torsor and $\varphi\colon \Vscr_{\AA^{\Tbar}}'^{\rm an} \isom \Lcal_{\AA^{\Tbar}}((\Vcal_t)_{t\in T})$ is an isomorphism.  
 One easily checks that if $\Vscr'$ and $(\Vcal_t)_{t\in T}$ are the restrictions of a common $\Gscr$-bundle $\Vscr \in \Bun_\Gscr(S)$, then there exists a canonical isomorphism $\varphi_\Vscr\colon \Vscr_{\AA^{\Tbar}}'^{\rm an} \isom \Lcal_{\AA^{\Tbar}}((\Vcal_t)_{t\in T})$, thus we get a natural functor
 \[\res_T\colon \Bun_\Gscr(S) \to \Bun_{\Gscr,T}(S).\]

 When $T$ is finite, then we have $\AA^{\bar T}=\prod_{t\in T}\Fsf_t$ and $L_{\AA^{\bar T}\Gsf} = \prod_{t\in T}L\Gsf_t$ with $\Gsf_t \coloneqq \Gsf_{\Fsf_t}$. Therefore, given $(\Vscr',(\Vcal_t)_{t \in T},\varphi)\in\Bun_{\Gscr,T}(S)$ we have 
\begin{equation}\label{eq-fin-BL-descent-datum}
    \Vscr'^{\rm an}_{\AA^{\bar T}} = \prod_{t\in T}\Vscr'^{\rm an}_{t}\quad\textnormal{and}\quad\varphi = (\varphi_t)_{t\in T}\colon \prod_{t\in T}\Vscr'^{\rm an}_{t}\to \prod_{t\in T}\Lcal\Vcal_t,
\end{equation}
where $\Vscr'^{\rm an}_{t}$ is the $L\Gsf_t$-torsor restricting $\Vscr'$ and  $\varphi_t\colon \Vscr'^{\rm an}_{t} \cong \Lcal\Vcal_t$ is an isomorphism of $L\Gsf_t$-torsors. In particular, the objects in $\Bun_{\Gscr,T}(S)$ are classical Beauville--Laszlo descent data and it is not hard to deduce from Cor.~\ref{cor-lemma-5.1} that $\res_T$ is an equivalence of categories.

If $T$ is not finite, then the functor $\res_T$ is not necessarily an equivalence but it satisfies the following weaker property.
 
 \begin{prop} \label{prop-adelic-BL}
  Let $S$ be an $\FF_q$-scheme and let $T \subset |C|$ as in \S\ref{ssect-infinite-level-structure}. Then the functor $\res_T\colon \Bun_\Gscr(S) \to \Bun_{\Gscr,T}(S)$ is fully faithful. If $(\Vscr',(\Vcal_t)_{t\in T},\varphi) \in \Bun_{\Gscr,T}(S)$ is such that $\Vcal_t$ is trivial for all $t\in T$, then there exists a $\Vscr \in \Bun_{\Gscr}(S)$ such that $\res_T(\Vscr) = (\Vscr',(\Vcal_t)_{t\in T},\varphi)$.
 \end{prop}
 \begin{proof}
%
We first show that $\res_T$ is fully faithful. It is faithful since $\Vscr \mapsto \restr{\Vscr}{C \setminus T}$ is. To see that it is full, choose $\Vscr,\Wscr \in\Bun_\Gscr(S)$ and let $(f',(f_t)_{t\in T})\colon \res_T(\Vscr) \to \res_T(\Wscr)$ be an isomorphism.  We write $\res_T(\Vscr) = (\Vscr',(\Vcal_t)_{t\in T},\varphi_\Vcal)$ and $ \res_T(\Wscr) = (\Wscr',(\Wcal_t)_{t\in T},\varphi_\Wcal)$. Since $\Vscr$ and $\Wscr$ are of finite presentation, we have
  \begin{align*}
   \underline\Hom(\Vscr,\Wscr)(C\setminus T) &= \underline\Hom(\Vscr,\Wscr)(\varprojlim_{T_0 \subset T \textnormal{ finite}} C \setminus T_0) \\ 
   &= \varinjlim_{T_0 \subset T \textnormal{ finite}} \underline\Hom(\Vscr,\Wscr)(C \setminus T_0).
  \end{align*}
  In other words, there exists a finite set $T_0 \subset T$ such that $f'$ extends to an isomorphism $\tilde{f}\colon \restr{\Vscr}{(C\setminus T_0)_S} \isom \restr{\Wscr}{(C\setminus T_0)_S}$. In order to show that $\res_{T}(\tilde{f}) = (f',(f_t)_{t\in T\setminus T_0})$, 
  we consider the following commutative diagram of the components at $t\in T\setminus T_0$ of $\varphi_{\Vscr}$, $\varphi_{\Wscr}$ and $(f',(f_t)_{t\in T})$:
  \begin{center}
   \begin{tikzcd}
    \Vscr_{t}'^{\rm an} \ar{r}{\varphi_{\Vscr,t}} \ar{d}[swap]{f_{t}'^{\rm an}} &\Lcal\Vcal_t \ar{d}{\Lcal(f_t)} \\
    \Wscr_{t}'^{\rm an} \ar{r}{\varphi_{\Wscr,t}} &\Lcal\Wcal_t 
   \end{tikzcd}.
  \end{center}
  Hence $f_t$'s are already uniquely determined by $f'$. Since we have $\restr{\ftilde}{(C\setminus T)_S} = f'$ by construction, it thus follows that $\res_{T}(\tilde{f}) = (f',(f_t)_{t\in T\setminus T_0})$. Since the proposition is true for $T_0$ being a finite subset of $|C|$ by Cor.~\ref{cor-lemma-5.1}, there exists a morphism $f\colon \Vscr\to\Wscr$ such that  $\res_{T_0}(f) =(\ftilde,(f_t)_{t \in T_0})$ and hence $\res_T(f) = (f',(f_t)_{t \in T})$.

  It remains to show that if there exists a trivialisation $\epsilon_t\colon \Vcal_t \isom (L^+\Gscr_t)_S$ for each $t\in T$ then $(\Vscr',(\Vcal_t)_{t\in T},\varphi) \in \Bun_{\Gscr,T}(S)$ lies in the essential image of $\res_T$. For this we assume without loss of generality that $S= \Spec R$ is affine. First note that $\Vscr'$ extends to a $\Gscr$-torsor $\Wscr$ on $C_R$. Indeed, since the stack of $\Gscr$-torsors is locally finitely presented, it extends to a torsor $\Wscr'$ over an open subset $(C\setminus T_0)_R$ with $T_0\subset T$. By Cor.~\ref{cor-lemma-5.1}, there exists a $\Gscr$-torsor on $C_R$ \[\Wscr \coloneqq \res_{T_0}^{-1}(\Wscr',(\Vcal_t)_{t \in T_0}, (\varphi_t)_{t\in T_0})\]  extending $\Vscr'$ as desired. 
  
 Now, observe that we have the following $\Gsf(\AA^{\Tbar}(R))$-equivariant bijection
  \[\xymatrix@1{
	\Wscr(\AA^{\Tbar}(R)) \cong \Vscr'(\AA^{\Tbar}(R)) \ar[r]^-{\cong}_-{\varphi} & \Lcal_{\AA^{\Tbar}}((\Vcal_t)_{t\in T}) (R) \ar[r]^-{\cong}_-{(\epsilon_t)}& (L_{\AA^{\Tbar}}\Gsf)(R) = \Gsf(\AA^{\Tbar}(R))}.
  \]
  In particular, the set $\Wscr(\AA^{\Tbar}( R ))$ is nonempty and the above bijection induces an isomorphism $\varphi^{\rm alg}\colon \Wscr_{\AA^{\Tbar}(R)} \isom \Gsf_{\AA^{\Tbar}(R)}$. Since $\Wscr$ and $\Gsf$ are of finite presentation, there exists a finite subset $T'\subset T$ such that for any $t\in T\setminus T'$ the isomorphism $\Lcal\Wscr_t\isom (L\Gsf_t)_R$ induced by $\varphi^{\rm alg}$ restricts to $\Wscr_t\isom (L^+\Gsf_t)_R$, and thus we have $\res_{T\setminus T'}(\restr{\Wscr}{C\setminus T'}) \cong (\Vcal',(\Vcal_t)_{t\in T\setminus T'},\restr{\varphi}{\AA^{T' \cup\Tbar}})$. As $T'$ is finite, there exists $\Vscr \in \Bun_\Gscr(R)$ such that $\res_{T'}(\Vscr) \cong (\restr{\Wscr}{C\setminus T'},(\Vcal_t)_{t\in T'}, (\varphi_t)_{t \in T'})$. Then we have $\res_T(\Vscr) \cong (\Vscr',(\Vcal_t)_{t\in T},\varphi)$ as desired.
 \end{proof}

 \subsection{An adelic description of global $\Gscr$-shtukas} \label{ssect-local-global}
 We now apply the above considerations to global $\Gscr$-shtukas. By Cor.~\ref{cor-lemma-5.1}, the functor
 \[
   \underline{\Vscr}_\bullet \mapsto (\restr{\underline{\Vscr}_\bullet}{C\setminus\{x\}},\loc{\underline{\Vscr}_\bullet}{x},can.)
 \] 
 defines an equivalence from the groupoid of global $\Gscr$-shtukas over a scheme $S$ to the groupoid of triples ($\underline\Vscr'_\bullet,\Vcal_\bullet,\varphi)$ where $\underline\Vscr'_\bullet$ is a global $\Gscr$-shtuka with $C$ replaced by $C\setminus\{x\}$, $\Vcal_\bullet$ is an ``isogeny chain'' of local $ \Res_{O_x/\f\pot{\unif_x}} \Gscr_x$-shtukas defined analogously to \eqref{eq-loc-bullet} and $\varphi: \underline{\Vscr'}_{\bullet,x}^{\rm an} \isom \Lcal\underline{\Vcal}_\bullet$ is an isomorphism of a chain of isoshtukas. In particular, any compatible family of quasi-isogenies of local $\Res_{O_x/\f\pot{\unif_x}} \Gscr_x$-shtukas $\rho\colon \loc{\underline{\Vscr}_\bullet}{x} \to \underline\Vcal'_\bullet$ can be extended uniquely to a quasi-isogeny $\underline{\Vscr}_\bullet \to \underline\Vscr_\bullet'$ of global $\Gscr$-shtukas inducing identity away from $x$. (See also \cite[Prop.~5.7]{ArastehRad-Hartl:LocGlShtuka}.)
 
 Similarly, for any $T \subset |C|$ as in \S\ref{ssect-infinite-level-structure}, Prop.~\ref{prop-adelic-BL} yields a fully faithful functor 
  \[
   \underline{\Vscr}_\bullet \mapsto (\restr{\underline{\Vscr}_\bullet}{C\setminus T},(\loc{\underline{\Vscr}_\bullet}{x})_{x\in T},can.),
 \] 
 which induces an equivalence of categories if $T$ is finite. In general, its essential image contains those triples such that the  underlying $L\Res_{O_x/\f\pot{\unif_x}}\Gscr_x$-torsors of $\loc{\underline{\Vscr}_\bullet}{x}$ are trivial. As a  consequence, we obtain the following moduli description of $\Xscr_{\Gscr,\Ibf,T}$ introduced in Def.~\ref{df-infinite-level-structure}.
 
 \begin{prop} \label{prop-infinite-level}
Let $T \subset |C|$ be as in \S\ref{ssect-infinite-level-structure}. We define $\Xscr_{{\Gscr}|_{C\setminus T},\Ibf}^ Z $ the stack fibred in groupoids over $\f$ whose $S$-valued points are tuples $((\widetilde x_i)_{i\in I}, (\Vscr_j')_{j=0}^k, (\phi'_j)_{j=0}^k, \eta_0)$ given by
 \begin{bulletlist}
  \item $\widetilde x_i \in \pi_i^{-1}(C\setminus T) (S)$ for all $i \in I$, so $(\widetilde x_i)_{i\in I}\in (\widetilde C^I_Z\setminus\bpi^{-1}(T))(S)$;
  \item $\Vscr_j'$ is a $\restr{\Gscr}{C \setminus T}$-bundle on $(C \setminus T)_S$ for $j=0,\dotsc,k$;
  \item $\phi_j'\colon \restr{\Vscr'_{j}}{(C\setminus T)_S \setminus \bigcup_{i \in I_j} \Gamma_{\pi_i(\widetilde x_i)}} \isom \restr{\Vscr'_{j-1}}{(C \setminus T)_S \setminus \bigcup_{i \in I_j} \Gamma_{\pi_i(\widetilde x_i)}}$ for all $j=1,\dots,k$ such that  the restriction of $((\Vscr_j')_{j=0}^k, (\phi'_j)_{j=1}^k)$ to the formal completion along $\bigcup_{i\in I}\Gamma_{\pi_i(\widetilde x_i)}$ together with $(\widetilde x_i)_{i\in I}$ satisfies the same condition defining the closed substack $\Hecke^ Z _{\Gscr,\Ibf}\subset \Hecke^{ Z }_{\Gscr,\Ibf}\times_{C^I}\widetilde C^I_Z$, explained in Rmk.~\ref{rmk:bounds-via-BL};
  \item $\phi_0'\colon \sigma^\ast\Vscr'_0 \isom \Vscr'_k$;
  \item $\eta_{0}\colon (L_{\AA^{\Tbar}}\Gsf,\sigma) \isom ((\Vscr'_0)_{\AA^{\Tbar}}^{\rm an},(\tau'_0)_{\AA^{\Tbar}}^{\rm an})$, where $\tau'_0 = \phi_1'\circ\cdots\circ\phi_k'\circ\phi_0'$.
  \end{bulletlist}
  Then $\Xscr_{{\Gscr}|_{C\setminus T},\Ibf}^ Z $ and $\Xscr_{{\Gscr},\Ibf,T}^ Z $ are canonically isomorphic. In particular, $ \Xscr_{{\Gscr},\Ibf,T}^ Z $ does not depend on the choice of the integral model $\Gscr$ outside $C \setminus T$.
\end{prop}
\begin{proof}
 It follows directly from \S\ref{ssect-local-global} that $\Xscr_{{\Gscr},\Ibf,T}^ Z $ is canonically isomorphic to the stack whose $S$-valued points are tuples  \[((\widetilde x_i)_{i\in I}, (\Vscr_j')_{j=0}^k, (\phi'_j)_{j=0}^k, (\underline\Vcal_{0,t})_{t \in T}, (\eta_t)_{t\in T}, (\psi_t)_{t \in T})\] where
 \begin{bulletlist}
  \item $(\widetilde x_i)$, $(\Vscr_j')$, $(\phi_j')$ are as above,
  \item $\underline\Vcal_{0,t}$ is an \'etale local $\Res_{O_t/\f\pot{\varpi_t}} \Gscr_t$-shtuka together with a trivialisation $\eta_t\colon (\Res_{O_t/\f\pot{\varpi_t}} \Gscr_t,\sigma) \isom \underline \Vcal_{0,t}$ for every $t \in T$,
  \item $\psi\colon ((\Vscr'_0)_{\AA^{\Tbar}}^{\rm an},(\tau'_0)_{\AA^{\Tbar}}^{\rm an}) \isom \Lcal_{\AA^{\Tbar}}(\underline\Vcal_{0,t})$ is an isomorphism of $\Gsf_{\AA^{\Tbar}}$-isoshtukas.
 \end{bulletlist}
 It remains to show that the datum $(\underline\Vcal_{0,t},\eta_t,\psi_t)$ is equivalent to the datum $\eta_{0,t}$, where $\eta_{0,t}\colon (L\Res_{\Fsf_x/\FF_x\rpot{\unif_t}}\Gsf_t,\sigma) \isom ((\Vscr'_0)_{t}^{\rm an},(\tau'_0)_{t}^{\rm an})$ is the component of $\eta_0$ at $t$. Obviously, we can obtain $\eta_{0}$ from the triple by defining $\eta_{0} \coloneqq  \psi^{-1}\circ\Lcal_{\AA^{\Tbar}}((\eta_t)_{t\in T}) $. Conversely, given $\eta_{0,t}$ we define the triple by $\underline\Vcal_{0,t} \coloneqq (L^+\Res_{O_t,\f\pot{\varpi_t}}\Gscr_t,\sigma)$, $\psi_t =\eta_{0,t}^{-1}$ and $\eta_t \coloneqq \id$.
\end{proof}

\begin{rmk}\label{rmk-infinite-level}
 We obtain the following interesting alternative description of (finite) level structure. Assume that $T$ is finite and $\Gscr'$ another smooth affine group scheme over $C$ equipped with an isomorphism $\Gscr'_{C\setminus T} \cong \Gscr$ that identifies $\Gscr'(\prod_{x\in C} O_x)$ with a subgroup $\Ksf' \subset \Ksf$. Then  $\Xscr_{\Gscr,\Ibf,\Ksf'}^ Z $ is the restriction of $\Xscr_{\Gscr', \Ibf}^ Z $ to $C \setminus T$.
\end{rmk}

 \subsection{Igusa varieties over moduli stacks of global $\Gscr$-shtukas} \label{ssect-infinite-level-igusa}
 
  We fix a tuple of geometric points $\widetilde\xbf = (\widetilde x_i)_{i\in I} \in \widetilde C^I_Z(\cl\FF_q)$, and let $x_i\in |C|$ denote the underlying closed point of $\pi_i(\widetilde x_i)\in C(\cl\FF_q)$. \emph{We  allow $x_i =  x_{i'}$ only when we have $i,i'\in I_j$ for a common $I_j$.} Let $\xbf$ denote the set $\{x_1,\dotsc,x_m\} \subset |C|$ of the underlying closed points of $\{\pi_i(\widetilde x_i)\}_{i\in I}$ (removing duplicate points). We fix an element 
  \[\bbf = (b_x)_{x\in\xbf} \in \prod_{x\in\xbf} \Gsf(\Fsf_x \hat\otimes_{\FF_q} \cl\FF_q).\] 
  We may (and do) choose each $b_x$ to lie in $\Gsf(\Fsf_{x}\otimes_{\FF_q}\kappa)$ for some common finite subextension $\kappa$ of $\cl\FF_q$. By increasing $\kappa$ if necessary, we assume that each $\widetilde x_i$ is defined over $\kappa$. If the underlying closed points of $\{\pi_i(\widetilde x_i)\}_{i\in I}$ are  pairwise \revise{distinct, then} we may regard each $b_{x_i}$ as an element in $\Gsf(\breve\Fsf_{x_i})$ via the projection $\Gsf(\Fsf_{x_i}\otimes_{\FF_q}\kappa)\to \Gsf(\breve\Fsf_{x_i})$ corresponding to $\pi_i(\widetilde x_i)$; \emph{cf.} Rmk.~\ref{rmk-gl-loc} and \eqref{eq-gl-loc-b}.
  
   For any set $T \subseteq |C| \setminus \xbf$ as in \S\ref{ssect-infinite-level-structure}, we consider
  \begin{equation}\label{eq-Igusa}
   \Ig_{\Gscr,\Ibf,\widetilde\xbf,T}^{ Z ,\bbf} \coloneqq \Xscr_{{\Gscr},\Ibf,T,\widetilde\xbf}^ Z  \times_{\Xscr^ Z _{\Gscr,\Ibf,\widetilde\xbf}} \prod_{x\in\xbf} \Ig^{b_x}_{\underline\Vscr_0^{\rm univ}[x^\infty]}, 
   \end{equation}
  where $\Xscr_{{\Gscr},\Ibf,T,\widetilde\xbf}^ Z $ and $\Xscr^ Z _{\Gscr,\Ibf,\widetilde\xbf}$ are respectively the fibres of $\Xscr_{{\Gscr},\Ibf,T}^ Z $ and $\Xscr^ Z _{\Gscr,\Ibf}$ over $\widetilde\xbf\in \widetilde C^I_Z(\cl\FF_q)$, and $\underline\Vscr_0^{\rm univ}[x^\infty]$ is the localisation of $\underline\Vscr_\bullet^{\rm univ}$ at $x$ as in \eqref{eq-loc-j}.  Most importantly, if $T = |C| \setminus \xbf$ then we simply write $\Ig_{\Gscr,\Ibf,\widetilde\xbf}^{ Z ,\bbf}$ for $\Ig_{\Gscr,\Ibf,\widetilde\xbf,|C|\setminus\xbf}^{ Z ,\bbf}$. 
  
 By Prop.~\ref{prop-schematic}, $\Ig_{\Gscr,\Ibf,\widetilde\xbf,T}^{ Z ,\bbf}$ is representable by a (possibly empty) $\cl\FF_q$-scheme for any non-empty $T$. Note also that $\Ig_{\Gscr,\Ibf,\widetilde\xbf,T}^{ Z ,\bbf}$ can be defined over some finite field~$\kappa$ if we have $b_x\in\Gsf(\Fsf_x\otimes_{\FF_q}\kappa)$ and $\widetilde x_i$ is defined over $\kappa$ for each $i\in I$.

Now we are ready to obtain the basic properties of $\Ig_{\Gscr,\Ibf,\widetilde\xbf}^{ Z ,\bbf}$.

\begin{thm}\label{th-infinite-level-igusa-var}
\begin{subenv}
    \item\label{th-infinite-level-igusa-var:proet} There exists a reduced locally closed substack $C^{ Z ,\bbf}_{\Gscr,\Ibf,\widetilde\xbf}\subset \Xscr^ Z _{\Gscr,\Ibf,\widetilde\xbf}$ whose geometric points consist of those that can be lifted to $\Ig_{\Gscr,\Ibf,\widetilde\xbf}^{ Z ,\bbf}$. Moreover, the natural projection $\Ig_{\Gscr,\Ibf,\widetilde\xbf}^{ Z ,\bbf}\to \Xscr_{\Gscr,\Ibf,\widetilde\xbf}^{ Z }$ factors through $C^{ Z ,\bbf}_{\Gscr,\Ibf,\widetilde\xbf}$, and the induced map on the perfections
    \[\Ig_{\Gscr,\Ibf,\widetilde\xbf}^{ Z ,\bbf,\perf}\to C^{ Z ,\bbf,\perf}_{\Gscr,\Ibf,\widetilde\xbf}\]
    is a pro-\'etale Galois cover if $\Ig_{\Gscr,\Ibf,\widetilde\xbf}^{ Z ,\bbf}$ is non-empty.
    \item\label{th-infinite-level-igusa-var:nonempty} The scheme $\Ig_{\Gscr,\Ibf,\widetilde\xbf}^{ Z ,\bbf}$ is non-empty if and only if $\Ig_{\Gscr,\Ibf,\widetilde\xbf}^{ Z ,\bbf}(\cl\FF_q)$ is non-empty. 
    \item\label{th-infinite-level-igusa-var:moduli} Suppose that $\Ig_{\Gscr,\widetilde\xbf}^{ Z ,\bbf}(\cl\FF_q)$ is non-empty. Then the scheme $\Ig_{\Gscr,\Ibf,\widetilde\xbf}^{ Z ,\bbf}$ represents the functor which associates to an $\cl\FF_q$-scheme $S$ the set of isomorphism classes of tuples $(\Vscr_0', \tau'_0, \eta^{\xbf}_{0},(\eta_{0,x})_{x\in\xbf})$ where 
 \begin{bulletlist}
  \item $\Vscr_0'$ is a $\Gsf$-bundle on $\Spec \Fsf \times_{\Spec \FF_q} S$ for  $j=0,\dotsc,k$.
  \item $\tau_0'\colon \sigma^\ast\Vscr'_0 \isom \Vscr'_0$ is an isomorphism of $\Gsf$-bundles,
  \item $\eta_{0}^\xbf\colon (L_{\AA^{\xbf}}\Gsf,\sigma) \isom ((\Vscr_0')_{\AA^\xbf}^{\rm an},(\tau_0')_{\AA^\xbf}^{\rm an})$ is an isomorphism of $L_{\AA^{\xbf}}\Gsf$-isoshtukas,
  \item $\eta_{0,x}\colon (L\Res_{\Fsf_{x}/\f\rpot{\unif_{x}}}\Gsf,b_x\sigma) \isom ((\Vscr_0')_{x}^{\rm an},(\tau_0')_{x}^{\rm an})$ is an isomorphism of local $\Res_{\Fsf_{x}/\f\rpot{\unif_{x}}} \Gsf$-isoshtukas for every $x\in\xbf$.
  \end{bulletlist}
  \item\label{th-infinite-level-igusa-var:indep}  
  Suppose that $\Ig_{\Gscr,\Ibf,\widetilde\xbf}^{ Z ,\bbf}(\cl\FF_q)$ is non-empty. Then, up to canonical isomorphism, $\Ig_{\Gscr,\Ibf,\widetilde\xbf}^{ Z ,\bbf}$ depends only on $\Gsf$, $\xbf$ and the $\Gsf(\Fsf_{x}\hat\otimes_{\FF_q}\cl\FF_q)$-$\sigma$-conjugacy class of the $b_x$ for each $x\in\xbf$, but not on the choices of $\Gscr$, $\Ibf$, $\bbf$, $ Z $ and $\widetilde\xbf$ that make $\Ig_{\Gscr,\Ibf,\widetilde\xbf}^{ Z ,\bbf}$ non-empty. 
  
  \reviselong{Furthermore, up to canonical isomorphism, $C^{ Z ,\bbf}_{\Gscr,\Ibf,\widetilde\xbf}$ depends on the same data as $\Ig_{\Gscr,\Ibf,\widetilde\xbf}^{ Z ,\bbf}$, except that it depends on  $\Gscr$ and the $\Gscr(O_{x_i}\hat\otimes_{\FF_q}\cl\FF_q)$-$\sigma$-conjugacy class of $b_x$ for each $x\in\xbf$ in place of $\Gsf$ and the $\Gsf(\Fsf_{x_i}\hat\otimes_{\FF_q}\cl\FF_q)$-$\sigma$-conjugacy class.}
  \item\label{th-infinite-level-igusa-var:bound} Fix $\Gscr$, $\Ibf = (I_j)_{j=1}^k$, $\xbf$ and $\bbf$, and choose a partition $I = \bigsqcup_{x\in \xbf}I_x$ such that $I_x$ is contained in some $I_j$ for each $x\in \xbf$. Then, there exists a bound $Z$ considered in \S\ref{not-BD} and $\widetilde\xbf = (\widetilde x_i)_{i\in I} \in \widetilde C^I_Z(\cl\FF_q)$, such that
  \begin{bulletlist}
      \item $\Ig_{\Gscr,\Ibf,\widetilde\xbf}^{ Z ,\bbf}$ is non-empty, and
      \item $\widetilde x_i$ lies over $x\in\xbf$ if and only if $i\in I_x$.
  \end{bulletlist}
  \end{subenv}
    \end{thm}
\begin{proof} 
    Let $C^{ Z ,\bbf}_{\Gscr,\Ibf,\widetilde\xbf}\subset \Xscr^ Z _{\Gscr,\Ibf,\widetilde\xbf}$ be the reduced intersection of the central leaves of $\underline\Vscr_0^{\rm univ}[x^\infty]$'s with respect to $b_x$'s for $x\in\xbf$. Then claim~(\ref{th-infinite-level-igusa-var:proet}) follows from Prop.~\ref{prop-igusa-cover}(\ref{prop-igusa-cover-cen-leaf}) and Lem.~\ref{lem-igusa-tower}(\ref{lem-igusa-tower-proet}) as the (possibly infinite) direct product of pro-\'etale Galois covers is again pro-\'etale Galois.

    If $\Ig_{\Gscr,\Ibf,\widetilde\xbf}^{ Z ,\bbf}$ is non-empty then so is $C_{\Gscr,\Ibf,\widetilde\xbf}^{ Z ,\bbf}$, which is locally of finite type over $\cl\FF_q$. Thus, from (\ref{th-infinite-level-igusa-var:proet}) we get $\Ig_{\Gscr,\Ibf,\widetilde\xbf}^{ Z ,\bbf}(\cl\FF_q)\ne\emptyset$, which is the non-trivial implication of (\ref{th-infinite-level-igusa-var:nonempty}).
    
	Observe that the natural projection induces isomorphisms on geometric fibres
	\begin{equation}\label{eq:xbf-tilde-to-xbf-bar}
	(\Gr_{\Gscr,\Ibf}\times_{C^I}\widetilde C^I_Z)_{\widetilde\xbf} \isom \Gr_{\Gscr,\Ibf,\overline\xbf}\quad\text{and}\quad
	 (\Xscr_{\Gscr,\Ibf}\times_{C^I}\widetilde C^I_Z)_{\widetilde\xbf} \isom \Xscr_{\Gscr,\Ibf,\overline\xbf},
	\end{equation}
	where $\overline\xbf\coloneqq \bpi(\widetilde\xbf)$. Thus, we shall view the geometric fibre $Z_{\widetilde\xbf}$ of $Z$ at $\widetilde\xbf$ as a closed subscheme of $\Gr_{\Gscr,\Ibf,\overline\xbf}$, and $\Xscr^{ Z }_{\Gscr,\Ibf,\widetilde\xbf}$ as a closed substack of $\Xscr_{\Gscr,\Ibf,\overline\xbf}$.
	
    To show claim (\ref{th-infinite-level-igusa-var:moduli}), we first make $\Gr_{\Gscr,\Ibf,\overline\xbf}$ more explicit as follows. Let $L^+_{\overline\xbf}\Gscr$ denote the fibre of  $L^+_{C^I}\Gscr$ at $\overline\xbf$, where $L^+_{C^I}\Gscr$ is the positive global loop group \cite[Def.~2.1.3]{Bieker:IntModels}. Let $\overline x_i\coloneqq \pi_i(\widetilde x_i)$ denote the $i$th component of $\overline\xbf$. Then we have
    \begin{align}
        \label{eq:infinite-level-igusa-var:loop-gp}L^+_{\overline\xbf}\Gscr(\cl\FF_q) &\cong \prod_{i\in I_{/\sim}}\Gscr(O_{\overline x_i})\subseteq\prod_{x\in\xbf} \Gscr(O_{\overline x}\hat\otimes_{\FF_q}\cl\FF_q)\quad \text{and}\\
        \label{eq:infinite-level-igusa-var:Gr}\Gr_{\Gscr,\Ibf,\overline\xbf}(\cl\FF_q) &\cong \prod_{i\in I_{/\sim}}\Gsf((\breve\Fsf)_{\overline x_i})/\Gscr(O_{\overline x_i})\subseteq\prod_{x\in\xbf}\Gsf(\breve\Fsf_{x_i}\hat\otimes_{\FF_q}\cl\FF_q)/\Gscr(\breve O_{x}\hat\otimes_{\FF_q}\cl\FF_q),
    \end{align}
    where $O_{\overline x_i}$ is the completed local rings $C_{\cl\FF_q}$ at $\overline x_i$ with fraction field $(\breve\Fsf)_{\overline x_i}$, and $x_i\in\xbf$ is the closed point underlying $\overline x_i$. Here, we set $i\sim i'$ for $i,i'\in I$ if and only if $\overline x_i = \overline x_{i'}$ in $C(\cl\FF_q)$. Using
    the isomorphism $\breve O_{x}\hat\otimes_{\FF_q}\cl\FF_q \cong \prod_{\overline x}O_{\overline x}$
    where the product runs through the geometric point $\overline x$ over $x$, the middle terms in \eqref{eq:infinite-level-igusa-var:loop-gp} and \eqref{eq:infinite-level-igusa-var:Gr} can be viewed as direct factors of the right most terms. The isomorphism in \eqref{eq:infinite-level-igusa-var:loop-gp} is clear from the definition \cite[Def.~2.1.3]{Bieker:IntModels}, and the bijection \eqref{eq:infinite-level-igusa-var:Gr} follows from the description of $\Gr_{\Gscr,\Ibf}$ in terms of the Beauville--Laszlo descent lemma as in \cite[Rmk.~2.1.2]{Bieker:IntModels} (see also Rmk.~\ref{rmk:bounds-via-BL}), together with the condition imposed on $\widetilde\xbf$ at the beginning of \S\ref{ssect-infinite-level-igusa}. 
    Under these identifications, the action of $L^+_{\overline\xbf}\Gscr(\cl\FF_q)$ on $\Gr_{\Gscr,\Ibf,\overline\xbf}(\cl\FF_q)$ can be interpreted as the ``left multiplication'' action.

    Now, given $(\underline\Vscr_\bullet,(\eta_t)_{t\notin \xbf},(\eta_x)_{x\in\xbf})\in \Ig_{\Gscr,\widetilde\xbf}^{ Z ,\bbf}(S)$, we obtain a tuple $(\Vscr',\tau_0',\eta_0^\xbf,(\eta_{0,x})_{x\in\xbf})$ as in (\ref{th-infinite-level-igusa-var:moduli}) via the following recipe:
    \begin{bulletlist}
        \item $(\Vscr'_0,\tau'_0)\coloneqq \restr{(\Vscr_0,\tau_0)}{\Spec\Fsf\times_{\Spec\FF_q} S}$;
        \item $\eta^{\xbf}_0 \coloneqq \Lcal_{\AA^\xbf}((\eta_t)_{t\notin\xbf})$;
        \item  $\eta_{0,x}\coloneqq \Lcal\eta_x$ for $x\in\xbf$.
    \end{bulletlist}
    (To define $\eta^{\xbf}_0$ and $(\eta_{0,x})_{x\in\xbf}$, we use the canonical isomorphisms $((\Vscr_0')_{\AA^\xbf}^{\rm an},(\tau_0')_{\AA^\xbf}^{\rm an})\cong \Lcal_{\AA^\xbf}((\loc{\underline\Vscr_0}{t})_{t\notin \xbf})$ and $((\Vscr_0')_{x}^{\rm an},(\tau_0')_{x}^{\rm an})\cong \Lcal\loc{\underline\Vscr_0}{x}$ for $x\in\xbf$, respectively.)
    
    Conversely, given a tuple $(\Vscr_0', \tau_0', \eta^{\xbf}_{0},(\eta_{0,x})_{x\in\xbf})$ over $S$  as in (\ref{th-infinite-level-igusa-var:moduli}) one gets a $\Gscr$-shtuka $\underline\Vscr_\bullet \in \Xscr_{\Gscr,\Ibf,\overline\xbf}(S)$ by repeating the argument of Prop.~\ref{prop-infinite-level}. (To produce $\Vscr_0\in\Bun_\Gscr(S)$, we apply Prop.~\ref{prop-adelic-BL} to the following triple for $T = |C|$
    \begin{equation}\label{eq:infinite-level-igusa-var:adelic-BL}
        (\Vscr_0',(L^+\Res_{O_t/\FF_t\pot{\varpi_t}}\Gscr_t)_{t\in|C|},(\eta_0^{\xbf,-1},(\eta_{0,x}^{-1})_{x\in\xbf}))\in\Bun_{\Gscr,|C|}(S),
    \end{equation}
    and recursively build $\Vscr_j$ and $\phi_j$ so that we have isomorphisms $\eta^\xbf$ and $(\eta_{x})_{x\in\xbf}$.) The resulting $\Gscr$-shtuka $\underline\Vscr_\bullet$ is equipped with the following isomorphism 
\begin{equation}\label{eq:infinite-level-igusa-var:BL-at-xbf}
    \eta_x \colon (L^+\Res_{O_x/\FF_x\pot{\varpi_x}}\Gscr_x, b_x\sigma)\riso \loc{\underline\Vscr_0}{x}
\end{equation}
for each $x\in |C|$, where we set $b_x = 1$ if $x\notin\xbf$. In fact, we obtain $\eta_x$ by restricting $\eta_{0,x}$ (resp. $\eta^\xbf$) to $L^+\Res_{O_x/\FF_x\pot{\varpi_x}}\Gscr_x$ if $x\in\xbf$ (resp. $x\notin \xbf$). 

Now, by restricting the resulting $\Gscr$-shtuka $\underline\Vscr_\bullet$ to the formal neighbourhood $\Dhat\cong\bigsqcup_{i\in I_{/\sim}}(S\hat\times_{\cl\FF_q}\Spf\breve O_{x_i})$ of $D=\bigcup_{i\in I}\Gamma_{\overline x_i}$ and trivialising $\restr{\Vscr_0}{\Dhat}$ via $\hat\epsilon =(\eta_x^{-1})_{x\in\xbf}$, we can associate to $(\underline\Vscr_\bullet,(\eta_x)_{x\in\xbf})$ an $S$-point of $\Gr_{\Gscr,\Ibf,\overline\xbf}$  by \cite[Rmk.~2.1.2]{Bieker:IntModels} (\emph{cf.} Rmk.~\ref{rmk:bounds-via-BL}). 
%
Furthermore, by inspecting $(\eta_x)_{x\in\xbf}$ \eqref{eq:infinite-level-igusa-var:BL-at-xbf}, it follows that this $S$-point actually factors through the following $\cl\FF_q$-point defined by $\bbf$:
\begin{equation}\label{eq:bbf-orbit}
    \big(b_x\cdot\Gscr(\breve O_x\hat\otimes_{\FF_q}\cl\FF_q)\big)_{x\in\xbf} \in \Gr_{\Gscr,\Ibf,\overline\xbf}(\cl\FF_q).
\end{equation}
 (If $\overline x$ is a geometric point over $x\in\xbf$ not lying under any leg $\widetilde x_i$, then the projection $\Gsf(\breve\Fsf_x\hat\otimes_{\FF_q}\cl\FF_q)/\Gscr(\breve O_x\hat\otimes_{\FF_q}\cl\FF_q) \epi \Gsf((\breve\Fsf)_{\overline x})/\Gscr(O_{\overline x})$ sends the coset of $b_x$ to the identity coset. Thus, \eqref{eq:bbf-orbit} defines an $\cl\FF_q$-point of $\Gr_{\Gscr,\Ibf,\overline\xbf}$ via \eqref{eq:infinite-level-igusa-var:Gr}.) 

Now by Rmk.~\ref{rmk:bounds-via-BL} and \eqref{eq:xbf-tilde-to-xbf-bar}, the previous paragraph shows that the $\Gscr$-shtuka $\underline\Vscr_\bullet$ over $S$ constructed from a tuple as in (\ref{th-infinite-level-igusa-var:moduli}) lies in $\Xscr^Z_{\Gscr,\Ibf,\widetilde\xbf}(S)$ if and only if (the $L^+_{\overline\xbf}\Gscr$-orbit of) the point in \eqref{eq:bbf-orbit} is contained in $Z_{\widetilde\xbf}$. Claim (\ref{th-infinite-level-igusa-var:moduli}) now follows, since the non-emptiness of $\Ig^{Z,\bbf}_{\Gscr,\Ibf,\widetilde\xbf}(\cl\FF_q)$ implies that $Z_{\widetilde\xbf}$ contains the point \eqref{eq:bbf-orbit}.
    
    Note that the moduli description of $\Ig^{Z,\bbf}_{\Gscr,\Ibf,\widetilde\xbf}$ given in (\ref{th-infinite-level-igusa-var:moduli}) only depends on $\Gsf$, $\xbf$ and the isomorphism class of the local $\Res_{\Fsf_{x}/\FF_q\rpot{\varpi_{x}}}\Gsf_{x}$-isoshtuka over $\cl\FF_q$ associated to $b_x$ for each $x\in\xbf$. \reviselong{From this, claim~}(\ref{th-infinite-level-igusa-var:indep})\reviselong{ for $\Ig^{Z,\bbf}_{\Gscr,\Ibf,\widetilde\xbf}$ follows.     
    Furthermore,} by the ind-stack structure of $\Xscr_{\Gscr,\Ibf}$ \cite[Thm.~3.15]{ArastehRad-Hartl:UnifStack}, there exists an Arasteh~Rad--Hartl bound $Z\subset \Gr_{\Gscr,\Ibf}$ such that a $\Gscr$-shtuka  $\underline\Vscr_\bullet\in\Xscr_{\Gscr,\Ibf,\overline\xbf}$ constructed from a tuple as in (\ref{th-infinite-level-igusa-var:moduli}) is bounded by $Z$, which proves  (\ref{th-infinite-level-igusa-var:bound}).

%
\begin{revisebox}
    \indent It remains to show claim~(\ref{th-infinite-level-igusa-var:indep}) for central leaves. As independence of the bound is clear from the case of Igusa varieties, we may assume that $Z$ is a sufficiently large Arasteh~Rad--Hartl bound and $\widetilde\xbf = \overline\xbf$. Then by definition, the reduced substack $C^{Z,\bbf}_{\Gscr,\Ibf,\overline\xbf}\subset\Xscr_{\Gscr,\Ibf,\overline\xbf}$ depends on the isomorphism class of the local $L^+\Res_{O_x/\FF_x\pot{\varpi_x}}\Gscr_x$-shtuka over $\cl\FF_q$ given by $ b_x$ for each $x\in\xbf$. It  therefore remains to verify that $C^{Z,\bbf}_{\Gscr,\Ibf,\overline\xbf}$ is canonically independent of the choice of the moduli ind-stack $\Xscr_{\Gscr,\Ibf,\overline\xbf}$; namely, the choice of $I$, $\overline\xbf$ and $\Ibf$. 
    
    First, for fixed $I$ and $\overline\xbf$, both $\Xscr_{\Gscr,\Ibf,\overline\xbf}$ and its substack $C^{ Z ,\bbf}_{\Gscr,\Ibf,\overline\xbf}$ are canonically independent of the choice of partition $\Ibf$ of $I$ satisfying the condition in claim~(\ref{th-infinite-level-igusa-var:bound}), which is a standard consequence of the Beauville--Laszlo descent (\emph{cf.} Cor.~\ref{cor-BL-for-torsors}). Thus, we may replace $\Ibf$ with the coarsest partition $(I)$ of $I$. Next, choose $I'$ and $\overline \xbf'\in C(\cl\FF_q)^{I'}$ so that $\overline \xbf'$ consists of the $\cl\FF_q$-points of $C$ lying above some $x\in\xbf$, each taken exactly once. Then we obtain a closed immersion of ind-stacks
    \[
    \begin{tikzcd}
        \Xscr_{\Gscr,(I),\overline\xbf} \arrow[r,hookrightarrow] & \Xscr_{\Gscr,(I'),\overline\xbf'},
    \end{tikzcd}
    \]
    obtained by removing duplicate legs and adding legs not already contained in $\overline\xbf$. 
    Then this immersion sends $C^{Z,\bbf}_{\Gscr,(I),\overline\xbf}$ isomorphically onto $C^{Z',\bbf}_{\Gscr,(I'),\overline\xbf'}$ for some large enough Arasteh~Rad--Hartl bound $Z'$, which can be seen by inspecting geometric points via the Beauville--Laszlo descent. This shows claim (\ref{th-infinite-level-igusa-var:indep}) for central leaves and thus completes the proof of Theorem~\ref{th-infinite-level-igusa-var}.
\end{revisebox}
\end{proof}
\begin{rmk}\label{rmk-igusa-variety-over-finite-fields}
      Recall that $\Ig^{ Z ,\bbf}_{\Gscr,\Ibf,\widetilde\xbf}$ admits a model over a finite field $\kappa$ if each leg $\widetilde x_i$ is defined over $\kappa$ and $b_x\in\Gsf(\Fsf_x\otimes_{\FF_q}\kappa)$. By the same proof, this $\kappa$-model of $\Ig^{ Z ,\bbf}_{\Gscr,\Ibf,\widetilde\xbf}$ has essentially the same moduli description as Thm.~\ref{th-infinite-level-igusa-var}(\ref{th-infinite-level-igusa-var:moduli}), where $S$ is a $\kappa$-scheme (instead of an $\cl\FF_q$-scheme).
\end{rmk}

Now, Thm.~\ref{th-infinite-level-igusa-var} justifies the following notation.
\begin{defn}\label{def-infinite-level-igusa-variety}
    If $\Ig_{\Gscr,\Ibf,\widetilde\xbf}^{ Z ,\bbf}$ is non-empty then we write
    \begin{equation}
        \Ig_{\Gsf,\xbf}^{\bbf}\coloneqq\Ig_{\Gscr,\Ibf,\widetilde\xbf}^{ Z ,\bbf} \quad\textnormal{and} \quad C^\bbf_{\Gscr,\xbf}\coloneqq C^{ Z ,\bbf}_{\Gscr,\Ibf,\widetilde\xbf}
    \end{equation}
    and call them the \emph{Igusa variety} and \emph{central leaf} for $\bbf$, respectively.
\end{defn}

Conceptually, the central leaf $C^\bbf_{\Gscr,\xbf}$ should be the reduced intersection of the central leaves for $b_x$'s for $x\in\xbf$ in the \emph{unbounded} moduli stack $\Xscr_{\Gscr,\Ibf,\overline\xbf}$ with fixed legs $\overline\xbf$. But since $\Xscr_{\Gscr,\Ibf,\overline\xbf}$ is an \emph{ind}-Delighe--Mumford stack, to apply results in \S\ref{sect-Igusa-Covers} we should fix a suitable bound $Z$ and construct the central leaf in the resulting closed \emph{Deligne--Mumford substack} $X^Z_{\Gscr,\Ibf,\widetilde\xbf}$ (for some $\widetilde\xbf\in\widetilde C^I_Z(\cl\FF_q)$). 

Although the Igusa variety only depends on $\Gsf$, $\xbf$ and the $\sigma$-conjugacy class of $\bbf$ up to canonical isomorphism, the natural projection 
\begin{equation}\label{eq-igusa-projection}
\begin{tikzcd}
    \Ig^\bbf_{\Gsf,\xbf}\cong \Ig_{\Gscr,\Ibf,\widetilde\xbf}^{ Z ,\bbf} \arrow{r} & \Xscr^{ Z }_{\Gscr,\Ibf,\widetilde\xbf}
\end{tikzcd}
\end{equation}
clearly depends on the auxiliary choices suppressed from the notation $\Ig^\bbf_{\Gsf,\xbf}$. On the one hand, the moduli description of $\Ig^\bbf_{\Gsf,\xbf}$ given in Thm.~\ref{th-infinite-level-igusa-var}(\ref{th-infinite-level-igusa-var:moduli}) is much more flexible than the moduli description in terms of bounded $\Gscr$-shtukas with certain infinite level structure. On the other hand, for potential applications, one may need to make suitable auxiliary choices, including the integral model $\Gscr$ and the bound $Z$, to project $\Ig^\bbf_{\Gsf,\xbf}$ to an \emph{interesting} moduli stack $ \Xscr^{ Z }_{\Gscr,\Ibf,\widetilde\xbf}$. 

\begin{rmk}\label{rmk-igusa-independence-of-choice}
We use the notation from the proof of Thm.~\ref{th-infinite-level-igusa-var}. For any $i\in I_{/\sim}$ let $b_{\overline x_i}\in\Gsf(\breve{\Fsf}_{x_i})$ denote the image of $b_{x_i}$ under the projection $\Gsf(\breve\Fsf_{x_i}\hat\otimes_{\FF_q}\cl\FF_q)\epi \Gsf((\breve{\Fsf})_{\overline x_i})$. Then, the $L^+_{\overline\xbf}\Gscr$-orbit of the point \eqref{eq:bbf-orbit}, which we denote as $S_\bbf\subset \Gr_{\Gscr,\Ibf,\overline\xbf}$, can be written as follows:
    \begin{equation}\label{eq:S-bbf}
        S_\bbf(\cl\FF_q) = \prod_{i\in I_{/\sim}}\Gscr(O_{\overline x_i}) b_{\overline x_i}\Gscr(O_{\overline x_i})/\Gscr(O_{\overline x_i}).
    \end{equation}

The key observation in the proof of Thm.~\ref{th-infinite-level-igusa-var} is that if a $\Gscr$-shtuka $\underline\Vscr_\bullet$ admits an infinite Igusa level structure $(\eta_{0,x})_{x\in\xbf}$ (over some profinite \'etale cover of $S$), then $\underline\Vscr_\bullet$ lies in the preimage of $[L^+_{\overline\xbf}\Gscr\backslash S_\bbf]$ under the local model morphism 
\begin{equation}
    \begin{tikzcd}
    \Xscr_{\Gscr,\Ibf,\overline\xbf} \arrow{r} & {[L^+_{\overline\xbf}\Gscr\backslash\Gr_{\Gscr,\Ibf,\overline\xbf}]}
\end{tikzcd}
\end{equation}
(\emph{cf.} \cite[Prop.3.4.2]{Bieker:IntModels}).
In particular, \emph{any bound $Z$ whose fibre $Z_{\widetilde\xbf}$ at a fixed leg $\widetilde\xbf$ over $\overline\xbf$ contains $S_\bbf$} can be used to construct the \emph{correct} central leaf in $\Xscr_{\Gscr,\Ibf,\overline\xbf}$ and the \emph{correct} Igusa variety. 
Conceptually, our setting naturally aligns with the bounds of \emph{finite type}’’ introduced by Hartl and Y. Xu \cite[Def.~2.6.1, 2.6.3]{HartlXu:Unif}.
\end{rmk}

Using Rmk.~\ref{rmk-igusa-independence-of-choice}, one can get the following variant of Thm.~\ref{th-infinite-level-igusa-var}(\ref{th-infinite-level-igusa-var:bound}) for Beilinson--Drinfeld Schubert varieties when $\Gscr$ is parahoric at each $x\in\xbf$.
%
\begin{lem}\label{lem:igusa-admissible}
    Fix $\xbf = \{x_i\}_{i\in I}\subset |C|$, and suppose that $\Gscr_{x_i}$ is a parahoric group scheme over $O_{x_i}$ for each $x_i\in\xbf$. 
    Choose $b_i\in \Gsf(\breve\Fsf_{x_i})$ and $\lambda_i \in X_\ast(\Tsf)_+$ for each $i\in I$ such that the $\sigma^{\deg(x_i)}$-conjugacy class of $b_i$ is $\lambda_i$-admissible.

    Choose a partition $\Ibf = (I_j)_{j=1}^k$ of $I$, and let $Z(\blambda)\subset\Gr_{\Gscr,\Ibf}\times_{C^I}\widetilde C^I_Z$ be the Beilinson--Drinfeld Schubert variety for $\blambda\coloneqq(\lambda_i)_{i\in I}$. Fix a point $\widetilde\xbf = (\widetilde x_i)_{i\in I}\in \widetilde C^I_{Z(\blambda)}(\cl\FF_q)$ so that each $\widetilde x_i$ lies above $x_i$.
    
    Then, there exists $\bbf' = (b'_i)_{i\in I}$ where each $b'_i$ is $\Gsf(\breve\Fsf_{x_i})$-$\sigma^{\deg(x_i)}$-conjugate to $b_i$, such that $\Ig^{Z(\blambda),\bbf'}_{\Gscr,\Ibf,\widetilde\xbf}$ is non-empty; i.e., we have $\Ig^\bbf_{\Gsf,\xbf}\cong \Ig^{Z(\blambda),\bbf'}_{\Gscr,\Ibf,\widetilde\xbf}$.
\end{lem}    
Note that we implicitly identify $b_i$ and $b_i'$ with their respective images under $\Gsf(\breve\Fsf_{x_i})\to\Gsf(\breve\Fsf_{x_i}\hat\otimes_{\FF_q}\cl\FF_q)$ via $\overline x_i=\pi_i(\widetilde x_i)$, as in \eqref{eq-gl-loc-b}. Furthermore, if $b_i$ and $b_i'$ are $\Gsf(\Fsf_{x_i})$-$\sigma^{\deg(x_i)}$-conjugate, then the corresponding elements in $\Gsf(\breve\Fsf_{x_i}\hat\otimes_{\FF_q}\cl\FF_q)$ are $\sigma$-conjugate, so we get an isomorphism $\Ig^{\bbf}_{\Gsf,\xbf} \cong \Ig^{\bbf'}_{\Gsf,\xbf}$.
\begin{proof}
    As $\widetilde x_i$'s lie above pairwise distinct closed points of $C$, we have 
    \[Z(\blambda)_{\widetilde\xbf}\cong \prod_{i\in I} Z(\lambda_i)_{\widetilde x_i},\] 
    where $Z(\lambda_i)\subset \Gr_{\Gscr,\{i\}}\times_C\widetilde C_{\lambda_i}$ is the single-legged  Beilinson--Drinfeld Schubert variety associated to $\lambda_i$. Furthermore, $Z(\lambda_i)$ contains all the $\lambda_i$-admissible affine Schubert cells; i.e., those associated to $\lambda_i$-admissible elements in the Iwahori Weyl group. (To see this, it suffices to check that the point of $\Gr_{\Gscr,\{i\},\overline x_i}(\cl\FF_q)$ given by the ``translation'' element of $X_\ast(\Tsf)_{\Gal(\cl{\breve\Fsf}_{x_i}/\breve\Fsf_{x_i})}$ lifts to a point in $Z(\lambda_i)_{\widetilde x_i}(\cl\Fsf_x)$, and such a lift can be constructed in the ind-finite ind-scheme $\Gr_{\Tscr,\{i\}}$ for a parahoric group scheme $\Tscr$ over $C$ with generic fibre $\Tsf$, using \cite[Lem.2.21]{Richarz:AffGr}.)
    
Now, by the converse of Mazur's inequality \cite[Thm.~A]{He:KottwitzRapoportConj} (which can be applied to \emph{any} connected reductive group $\Gsf_{\Fsf_{x_i}}$ by the note at the end of the introduction \cite[p.~1128]{He:KottwitzRapoportConj}), there exists $b_i'\in \Gsf(\breve\Fsf_{x_i})$ that is $\sigma^{\deg(x_i)}$-conjugate to $b_i$ such that $\Gscr(\breve O_{x_i})b_i'\Gscr(\breve O_{x_i})/\Gscr(\breve O_{x_i})$ is a $\lambda_i$-admissible affine Schubert cell. Hence, $S_{\bbf'}$ is contained in $Z(\blambda)_{\widetilde\xbf}$, so $\Ig^{Z(\blambda),\bbf'}_{\Gscr,\Ibf,\widetilde\xbf}$ is non-empty as explained in Rmk.~\ref{rmk-igusa-independence-of-choice}.
\end{proof}

 \begin{notation} \label{notation-Hecke-action}
 Let $\AA_{\breve\Fsf}$ be the restricted product of the completions of $\breve\Fsf$ at all places, so we have $L_\AA\Gsf(\cl\FF_q) = \Gsf(\AA_{\breve\Fsf})$. 
 To simplify the notation, we may also see $\bbf = (b_x)_{x\in\xbf}$ as an element of $\Gsf(\AA_{\breve\Fsf}) = \prod_{v\in |C|} \Gsf(\Fsf_v\hat\otimes_{\FF_q}\cl\FF_q)$, by setting $b_v = 1$ for $v\notin\xbf$. Thus, we may replace the datum $(\eta_0^\xbf,(\eta_x)_{x\in\xbf})$ with an isomorphism 
 \begin{equation}\label{eq:adelic-Igusa-level-str}
      \eta_0\colon (L_\AA \Gsf,\bbf\sigma) \isom ((\Vscr_0')_{\AA}^{\rm an},(\tau_0')_{\AA}^{\rm an}).
 \end{equation}
 In particular, we get a canonical action of $\Aut(L_\AA \Gsf, \bbf\sigma)$ on $\Ig_{\Gsf,\xbf}^{\bbf}$, given by pre-composing $\eta_0$. We call this action the \emph{Hecke action}.
 
 By replacing $b_{x}$ up to $\Gsf(\Fsf_{x}\hat\otimes_{\FF_q}\cl\FF_q)$-$\sigma$-conjugacy if necessary, we assume that every $b_{x}$ is decent; that is, for some integer $r$ dividing $\deg(x)$ for all $x\in\xbf$ we have
 \begin{equation} \label{eq-decent}
 (b_{x}\sigma)^r = (r \cdot \nu)(b_{x})(\varpi_{x}) \sigma^r.
 \end{equation}
 In particular, we have $ b_{x} \in \Gsf(\Fsf_{x} \otimes_{\FF_q} \FF_{q^r})$ and the Igusa variety admits a model over $\FF_{q^r}$, which we fix. We denote by $\Fr$ the geometric Frobenius automorphism of $\Ig_{\Gsf,\xbf}^{\bbf}$ over $\FF_{q^r}$ with respect to this $\FF_{q^r}$-model. It is easy to see that \eqref{eq-decent} implies that $J_{b_{x}}(F) \subset G(\Fsf_{x} \otimes \FF_{q^r})$. Thus, the Hecke action of
\begin{equation}
 \Aut(L_\AA \Gsf, \bbf\sigma) = \Gsf(\AA^{\xbf}) \times \prod_{x\in\xbf} J_{b_{x}}(\Fsf_{x}) \eqqcolon \JJ_\bbf
\end{equation}
 commutes with $\Fr$.
 \end{notation}

 \subsection{Controlling the centre}\label{ssect-Xi}

 Let us choose a discrete subgroup $\Xi\subset \Zsf_\Gsf(\AA)$ such that the quotient  $\Zsf_\Gsf(\Fsf)\backslash\Zsf_\Gsf(\AA)/\Xi$ is compact. In the following, we assume that $\Xi$ is torsion-free, which can always be achieved by replacing $\Xi$ with a finite index subgroup. Indeed, if $i\colon\GG_m^t\hookrightarrow\Zsf_\Gsf$ is a closed embedding onto the maximal split torus, then for any id\`eles $\xi_1,\dotsc,\xi_t\in \AA^\times$ with positive degree, the subgroup $\Xi\coloneqq \langle i(\xi_1 e_1),\dotsc , i(\xi_t e_t)\rangle\subset \Zsf_\Gsf(\AA)$ is discrete and cocompact in $\Zsf_\Gsf(\Fsf)\backslash\Zsf_\Gsf(\AA)$. It is easy to see that any discrete subgroup $\Xi\subset\Zsf_\Gsf(\AA)$ disjoint from $\Zsf_\Gsf(\Fsf)$  and  cocompact in $\Zsf_\Gsf(\Fsf)\backslash\Zsf_\Gsf(\AA)$ should contain such a subgroup.
 
  Note that $\Zsf_\Gsf(\AA)$ acts on $\Ig^\bbf_{\Gsf,\xbf}$ (viewing it as a subgroup of $\JJ_\bbf$). We set
  \begin{equation}\label{eq-Igusa-Xi}
  \Ig_{\Gsf,\xbf,\Xi}^{\bbf} \coloneqq \Ig_{\Gsf,\xbf}^{\bbf}\sslash\Xi \quad\text{and}\quad
  \Ig_{\Gsf,\xbf,\Xi\cdot\Ksf}^{\bbf} \coloneqq  \Ig_{\Gsf,\xbf,\Xi}^{\bbf}\sslash \Ksf ,
  \end{equation}
  using the same GIT-type quotient construction as in \S\ref{ssect-finite-level-igusa-cover}. Here, $\Ksf\subseteq \prod_{t\notin\xbf}\Gscr(O_t) \times \prod_{x\in\xbf}\Gamma_{b_x}$ is an open compact subgroup, which intersects trivially with any discrete torsion-free subgroup such as $\Xi$.
    Since $\Xi$ commutes with the Hecke action, the natural action of $\JJ_\bbf$ on $ \Ig_{\Gsf,\xbf}^{\bbf}$ descends to the $\Xi$-quotient $ \Ig_{\Gsf,\xbf,\Xi}^{\bbf}$, and ditto for the finite-level Hecke correspondences.

Let $\Gscr^{\ad}$ be as in \ref{ssect-HN}, and set $\bbf^{\rm ad} \coloneqq (b_x^{\ad})_{x\in\xbf}$ where $b_x^{\ad}\in\Gsf^{\ad}(\Fsf_{x}\hat\otimes_{\FF_q}\cl\FF_q)$ is the image of $b_x$. 
Choose a projection $\Ig^\bbf_{\Gsf,\xbf}\to \Xscr^Z_{\Gscr,\Ibf,\widetilde\xbf}$ as in \eqref{eq-igusa-projection}, and an Arasteh~Rad Hartl bound $Z^{\rm ad}\subset \Gr_{\Gscr^{\rm ad},\Ibf}$ such that the natural map $\Xscr^Z_{\Gscr,\Ibf}\to \Xscr_{\Gscr^{\rm ad},\Ibf}$ factors through $\Xscr^{ Z^{\rm ad} }_{\Gscr^{\rm ad},\Ibf}$. Lastly, we make a suitable auxiliary choice as in \S\ref{ssect-HN}, and define
\begin{equation}\label{eq-Igusa-HN}
 \Ig_{\Gsf,\xbf,\Xi}^{\bbf,\leqslant\mu}\coloneqq \Ig_{\Gsf,\xbf,\Xi}^{\bbf} \times_{\Xscr^{ Z^{\rm ad}}_{\Gscr^{\rm ad}, \Ibf}}\Xscr^{ Z^{\rm ad},\leqslant\mu}_{\Gscr^{\rm ad},\Ibf},\quad\text{for }\mu\in\Lambda^+.
\end{equation}

\begin{cor}\label{cor-schematic}
    The scheme $\Ig_{\Gsf,\xbf,\Xi}^{\bbf,\leqslant\mu}$ is quasi-compact. In particular, the $\cl\FF_q$-stack $\Ig_{\Gsf,\xbf,\Xi\cdot\Ksf}^{\bbf,\leqslant\mu}\coloneqq\Ig_{\Gsf,\xbf,\Xi}^{\bbf,\leqslant\mu}\sslash \Ksf$ is of finite type for any open compact subgroup $\Ksf\subseteq \prod_{t\notin\xbf}\Gscr(O_t) \times \prod_{x\in\xbf}\Gamma_{b_x}$, and it is representable by a scheme if $\Ksf$ is small enough relative to $\mu$.
\end{cor}
\begin{proof}
    By Prop.~\ref{prop-schematic} it suffices to show the quasi-compactness of $\Ig_{\Gsf,\xbf,\Xi}^{\bbf,\leqslant\mu}$.
    If $\Gscr = \Gscr^{\rm ad}$ then the desired claim follows from Prop.~\ref{prop:HN-qc} and the quasi-compactness of $\Ig^{\bbf^{\rm ad}}_{\Gsf^{\rm ad},\Ibf} \to \Xscr^{Z^{\rm ad}}_{\Gscr^{\rm ad},\Ibf}$. In general, the natural map 
\[ \begin{tikzcd}
    \Ig^{\bbf}_{\Gsf,\xbf,\Xi} \arrow[r] & \Ig^{\bbf^{\rm ad}}_{\Gsf^{\rm ad},\xbf}
\end{tikzcd} \]
is a profinite \'etale $\big(\Zsf_\Gsf(\Fsf)\backslash\Zsf_\Gsf(\AA)/\Xi\big)$-torsor, which can be seen from the moduli description given in Thm.~\ref{th-infinite-level-igusa-var}(\ref{th-infinite-level-igusa-var:moduli}). The corollary now follows since $\Ig_{\Gsf,\xbf,\Xi}^{\bbf,\leqslant\mu} $ is the preimage of the quasi-compact scheme $\Ig^{ \bbf^{\rm ad},\leqslant\mu}_{\Gsf^{\rm ad},\xbf}$ under this profinite \'etale cover.
\end{proof}

\section{The Hecke action on the full level Igusa variety}

 \subsection{Isogeny classes of global $G$-shtukas over $\cl\FF_q$}
 
 Given a pair $(\Vscr'_0,\tau_0')$ as in Thm.~\ref{th-infinite-level-igusa-var}(\ref{th-infinite-level-igusa-var:moduli}) over $S=\Spec \cl\FF_q$, we choose a trivialisation $\Vscr'_0 \cong \Gsf_{\breve\Fsf}$, which identifies $\tau_0'$ with $  b \sigma$ for a $  b  \in \Gsf(\breve\Fsf)$. A different choice of a trivialisation corresponds to replacing $  b $ with a $\sigma$-conjugate. Thus we get a bijection between the isomorphism classes of $(\Vscr'_0,\tau_0')$ and the $\sigma$-conjugacy classes $[b]_{\Gsf(\breve\Fsf)-\sigma} = [b]_\sigma \coloneqq \{\gsf^{-1}  b \sigma(\gsf) \mid \gsf \in \Gsf(\breve\Fsf)\}$ in $\Gsf(\breve\Fsf)$
 
 Denote by $\B(\Fsf,\Gsf)$ the pointed set of $\sigma$-conjugacy classes in $\Gsf(\breve\Fsf)$. In \cite{HamacherKim:Gisoc}, we classify the elements of $\B(\Fsf,\Gsf)$ by two invariants. To elaborate, let 
 \[
  \Div^\circ(\Fsf) = \left\{ \sum n_y \cdot y \in \bigoplus_{x \in |C|} \ZZ\cdot x \mid \sum n_y = 0 \right\}.
 \]
 and let $\Div^\circ(\Fsf^s) = \varinjlim \Div^\circ(\Esf)$, where $\Esf$ runs through all finite separable extensions of $\Fsf$.  We denote by $\DD_{\Esf/\Fsf}$ and $\DD_\Fsf$ the $\Fsf$-protori with  character group $\Div^\circ (\Esf)$ and $\Div^\circ(\Fsf^s)$, respectively. To every $b \in \Gsf(\breve\Fsf)$, we associate invariants
  \begin{align*}
   \nu_\Gsf(b) &\in \Hom_{\breve\Fsf}(\DD_\Fsf,\Gsf),\ \text{and} \\
   \bar\kappa_\Gsf(b) &\in (\pi_1(\Gsf) \otimes \Div^\circ(\Fsf^s))_{\Gal(\Fsf^s/\breve\Fsf)},
  \end{align*} 
 called the Newton point and Kottwitz point of $b$, respectively. Their images
 \begin{align*}
   \bar\nu_\Gsf(b) &\in \Hom_{\breve\Fsf}(\DD_\Fsf,\Gsf)/\Gsf(\breve\Fsf),\ \text{and} \\
   \bar\kappa_\Gsf(b) &\in (\pi_1(\Gsf) \otimes \Div^\circ(\Fsf^s))_{\Gal(\Fsf^s/\Fsf)}
  \end{align*}
 depend only on the $\sigma$-conjugacy class $[b]_\sigma$ and determine it uniquely.
  
\subsection{The local $\Gsf_x$-isoshtuka associated to $(\Vscr'_0,\tau_0')$} \label{sect-local-global-isosht}

 Following the construction in \S\ref{ssect-gl-loc}, we can associate a $\Gsf_x$-isoshtuka $(\Vscr'_0,\tau_0')_{\Fsf_x}$ to $(\Vscr'_0,\tau_0')$ for any $x \in |C|$, where $\Gsf_x\coloneqq \Gsf_{\Fsf_x}$. We can translate the construction into group-theoretic terms as follows. First, consider the embedding $\Gsf(\breve\Fsf) \mono \Gsf(\breve\Fsf \hat\otimes_\Fsf \Fsf_x) \cong (\Res_{\Fsf_x/\f\rpot{\unif_x}} \Gsf)(\breve\Fsf_x)$. This yields the localisation map
 \[
  \B(\Fsf,\Gsf) \to \B(\f\rpot{\unif_x},\Res_{\Fsf_x/\f\rpot{\unif_x}} \Gsf) \cong \B(\Fsf_x,\Gsf),
 \] 
 where the last isomorphism is Shapiro's isomorphism (\emph{cf}.\ \cite[\S2.3]{HamacherKim:Gisoc}). More explicitly, after choosing a place $y$ of $\breve\Fsf$ over $x$ and thus identifying $\breve\Fsf_x$ with the completion of $\breve\Fsf$ at $y$, this map is induced by the norm map 
 \begin{equation}\label{eq-norm}
   N^{(d)}(-)_y\colon \Gsf(\breve\Fsf \hat\otimes_{\Fsf} \Fsf_x) \xrightarrow{N^{(d)}}
   \Gsf(\breve\Fsf \hat\otimes_{\Fsf} \Fsf_x)\cong \prod_{y'|x}\Gsf((\breve\Fsf)_{y'}) \xrightarrow{\pr_y} \Gsf((\breve\Fsf)_{y})\cong\Gsf(\breve\Fsf_x),
 \end{equation}
where $d\coloneqq\deg(x)$ and $N^{(d)}\colon g\mapsto g\cdot\sigma(g)\cdots\sigma^{d-1}(g)$.

 We denote by $\B(\AA,\Gsf)$ the set of $\sigma$-conjugacy classes in $\Gsf(\AA_{\breve\Fsf})$. Since any element of a hyperspecial subgroup $\breve{K} \subset \Gsf(\breve\Fsf \hat\otimes_\Fsf \Fsf_x)$ is $\breve{K}$-$\sigma$-conjugate to $1$, the canonical map $\B(\AA,\Gsf) \to \prod_{x \in |C|} \B(\Fsf_x,\Gsf)$ yields an isomorphism 
 \[
  \B(\AA,\Gsf) \isom \{ [b_x]_\sigma \in \prod_{x \in |C|} \B(\Fsf_x,\Gsf) \mid [b_x]_\sigma=[1]_\sigma \text{ for almost all } x \}.
 \]
 In particular, the localisation of $[b]_\sigma \in \B(\Fsf_x,\Gsf)$ is trivial for almost all $x$.
 
 We denote by $\B(\Fsf,\Gsf)_\bbf \subset \B(\Fsf,\Gsf)$ the set of all $\sigma$-conjugacy classes localising to $[\bbf]_\sigma \in \B(\AA,\Gsf)$. By definition, this set corresponds to the isogeny classes of $\Gsf$-shtukas occurring over $\Ig^\bbf_{\Gsf,\xbf}(\k)$.  

\subsection{The group of self-quasi-isogenies of a shtuka} \label{ssect-self-qisog}
 We now obtain
 \[
  \Aut(\Vscr'_0,\tau_0') \cong \{ g \in \Gsf(\breve\Fsf) \mid g = b \sigma(g)b^{-1} \}.
 \]
 The right hand side is the group of $\Fsf$-rational points of the linear algebraic group $\Jsf_b$ given by
 \[
  \Jsf_b( R )= \{\gsf \in \Gsf(R\otimes_\Fsf \breve\Fsf) \mid \gsf b= b\sigma(\gsf)\}.
 \] 
  By \cite[Prop.~6.2]{HamacherKim:Gisoc} $\Jsf_b$ is an $\Fsf$-form of the centraliser $\Msf_b$ of $\nu_\Gsf(b)$, obtained by twisting the Frobenius action by $b$. In particular, the center $\Zsf_\Gsf$ of $\Gsf$ embeds into $\Jsf_b$. 
  
 By construction, any automorphism of a global $\Gsf$-shtuka induces automorphisms of its local $\Gsf$-shtukas. We can describe this in group-theoretic terms as follows. Given any place $y$ of $\breve\Fsf$ over $x$, the map $b \mapsto  (N^{(d)}b)_y$ \eqref{eq-norm} induces a bijection between $\sigma$-conjugacy classes in $\Gsf(\breve\Fsf \hat\otimes_\Fsf \Fsf_x)$ and $\sigma^d$-conjugacy classes in $\Gsf((\breve\Fsf)_y)$, where $d$ denote the degree of $x$. By construction the automorphism group of the local $\Res_{\Fsf_x/\f\rpot{\unif_x}}\Gsf_x$-isoshtuka given by $b$ can be (canonically) identified with
  \[
   J_b(\Fsf_x) \coloneqq \{ g \in \Gsf(\breve\Fsf \hat\otimes_\Fsf \Fsf_x) \mid g\iv b \sigma(g) = b \}.
  \]
  One easily checks that the projection $\Gsf(\breve\Fsf \hat\otimes_\Fsf \Fsf_x) \to \Gsf((\breve\Fsf)_y)$ defines an isomorphism
  \begin{equation} \label{eq-J_b}
   J_b(\Fsf_x) \isom J_{(N^{(d)}b)_y}(\Fsf_x) \coloneqq \{g \in \Gsf((\breve\Fsf)_y) \mid g\iv (N^{(d)}b)_y \sigma^d(g) = (N^{(d)}b)_y\}.
  \end{equation}
 
 We define the Newton point of $b$ as the $\breve\Fsf \hat\otimes_\Fsf \Fsf_x$-homomorphism $\nu_{\Gsf_x}(b) \colon \DD \to \Gsf$ such that $\restr{\nu_{\Gsf_x}(b)}{(\breve\Fsf)_y} = \nu((N^{(d)}b)_y)$, where the $\nu$ on the right hand side denotes the usual Newton point over $(\breve\Fsf)_y$. We denote by $M_{b,x} \subset \Gsf_{\breve\Fsf \hat\otimes_\Fsf \Fsf_x}$ the centraliser of $\nu_{\Gsf_x}(b)$. By \eqref{eq-J_b} and the usual results over local fields, we have $J_b(\Fsf_x) \subset M_{b,x}(\breve\Fsf \hat\otimes_\Fsf \Fsf_x)$.
 
 \begin{lem}\label{lem-loc-gl-Newton-pt}
  Let $b \in \Gsf(\breve\Fsf)$.
  \begin{subenv}
   \item Fix $x \in |C|$. The canonical embedding $\Gsf(\breve\Fsf) \mono \Gsf(\breve\Fsf \hat\otimes_\Fsf \Fsf_x)$ induces embeddings $\Msf_b(\breve\Fsf) \mono M_{b,x}(\breve\Fsf \hat\otimes_\Fsf \Fsf_x)$ and $\Jsf_b(\Fsf) \mono J_b(\Fsf_x)$.
   \item Let $g \in \Gsf(\breve\Fsf)$ such that $g \in M_{b,x}(\breve\Fsf \hat\otimes_\Fsf \Fsf_x)$ for all $x \in |C|$. Then $g \in \Msf_b(\breve\Fsf)$.
  \end{subenv}
 \end{lem}
 \begin{proof}
  As $\Jsf_b(\Fsf) \subset J_b(\Fsf_x)$ holds by definition, it remains to show the statements about $\Msf_b$. We fix a finite Galois extension $\Esf/\breve\Fsf$ such that $\nu_\Gsf(b)$ factors over $\DD_{\Esf/\Fsf}$. For every place $y'$ of $\Esf$ the map $\Div^\circ(\Esf) \to \ZZ, \sum_z a_z \cdot z \mapsto a_{y'}$ induces a morphism $\iota_{y'}\colon \GG_{m} \to \DD_{\Esf/\Fsf}$ defined over $\Esf$. We denote $\nu_{y'} \coloneqq \nu_{\Gsf}(b) \circ \iota_{y'}$. Note that $\Msf_b$ is the joint centraliser of all $\nu_{y'}$.  
  
  By \cite[Lem.~4.5]{HamacherKim:Gisoc}, we obtain $\nu_{\Gsf_x}(b)$ from $\nu_{\Gsf}(b)$ as follows. For every place $y$ of $\breve\Fsf$ over $x$ we may choose a place $y'$ of $\Esf$ over $y$ which corresponds to an embedding $i_{y'}\colon \Esf \mono \cl{(\breve\Fsf)_y}$. Then
  \[
   \restr{\nu_{\Gsf_x}(b)}{(\breve\Fsf)_y} = \frac{1}{[\Esf_{y'}: (\breve\Fsf)_y]} \cdot\nu_{y'},
  \]
  where $\nu_{y'}$ is considered to be defined over $\cl{(\breve\Fsf)_y}$  via the embedding $i_{y'}$.  In particular, we have $\Msf_b(\breve\Fsf) \subset M_{b,x}(\breve\Fsf \hat\otimes_\Fsf \Fsf_x)$. 
  
  The second assertion follows by the same argument. Let $g \in \Gsf(\breve\Fsf)$ such that $g \in M_{b,x}(\breve\Fsf \hat\otimes_\Fsf \Fsf_x)$ for all $x \in |C|$. As shown above, the claim is equivalent to $g$ centralising $\nu_z$ for every place $z$ of $\Esf$. Denote $y \coloneqq \restr{z}{\breve\Fsf}$. By choosing $y' = z$ in the equation above, the claim follows.
 \end{proof}
 
 Now assume that $[b]_\sigma \in \B(\Fsf,\Gsf)_\bbf$, i.e.\ there exists an isomorphism
 \[
  \eta_0\colon (L_\AA\Gsf,\bbf\sigma) \isom (L_\AA\Gsf,b\sigma).
 \]
 Any such $\eta_0$ is induced from the left multiplication by an element $\gbf \in \Gsf(\AA_{\breve\Fsf})$ sastifying $\bbf = \gbf\iv b \sigma(\gbf)$. So the embedding $\Aut(\Gsf_{\breve\Fsf},b\sigma) \to \Aut(L_\AA\Gsf,\bbf\sigma)$ induced by $\eta_0$ can be explicitly written as follows: 
 \begin{equation} \label{eq-global-to-local-automorphism}
 \iota_{\gbf}\colon \Jsf_b(\Fsf) \mono \JJ_\bbf, j \mapsto \gbf\iv j \gbf.
 \end{equation}
 
 \begin{cor} \label{cor-loc-gl-Newton-pt}
  Let $[b]_\sigma \in \B(\Fsf,\Gsf)_\bbf$ and $\gbf = (g_x) \in G(\AA_{\breve\Fsf})$ such that $\bbf = \gbf\iv b \sigma(\gbf)$. We extend $\iota_{\gbf}$ to a morphism $\Gsf(\breve\Fsf) \mono \Gsf(\breve\Fsf \hat\otimes_\Fsf \Fsf_x), h \mapsto \gbf\iv h \gbf$. Then
  \[
   \iota_\gbf(\Msf_b(\breve\Fsf)) = \iota_\gbf(\Gsf(\breve\Fsf))\cap \prod_{x\in |C|} M_{b_x,x}(\breve\Fsf \hat\otimes \Fsf_x).
  \]
 \end{cor}
 \begin{proof}
  When $\gbf = 1$, this is exactly Lem.~\ref{lem-loc-gl-Newton-pt}. The general case follows from the fact that $\nu_{\Gsf_x}(b_x) = g_x\iv \nu(b) g_x$ by \cite[Lem.~4.6]{HamacherKim:Gisoc}.
 \end{proof}

\subsection{Group theoretic description of $\JJ_{\bbf}$-orbits}\label{ssect-global-J-orbits}
 To describe the $\JJ_\bbf$-orbits, we partition the $\k$-points of the Igusa variety according to isogeny classes.
 Define
 \[
  \pi\colon \Ig^\bbf_{\Gsf,\xbf,\Xi}(\k) \to \B (\Fsf,\Gsf)_\bbf,\quad (\Vscr_0',\tau_0',[\eta_0]) \mapsto (\Vscr_0',\tau_0')/\cong
 \]
 to be the map associating to a point the isogeny class of the $\Gscr$-shtuka above it. (Here, $[\eta_0]$ is the equivalence class of $\eta_0$'s \eqref{eq:adelic-Igusa-level-str} such that $(\Vscr_0',\tau_0',\eta_0)$ are isomorphic.) By construction, the fibres of $\pi$ are precisely the $\JJ_\bbf$-orbits, so it descends to a map $\pi\colon \Ig^\bbf_{\Gsf,\xbf,\Xi\cdot\Ksf}(\k) \to \B (\Fsf,\Gsf)_\bbf$ for every $\Ksf$.

Let us fix $[{b}]_\sigma \in\B(\Fsf,\Gsf)_\bbf$ and a base point $\tilde z= (\Vscr_0',\tau_0',[\eta_0]) \in \pi^{-1}([{b}]_\sigma)$. By \eqref{eq-global-to-local-automorphism}, $[\eta_0]$ induces a closed embedding $\iota_{\ztilde}\colon \Jsf_b(\Fsf) \mono \JJ_\bbf$. Note that $\iota_{\tilde z}(\Jsf_{{b}}(\Fsf))$ is the stabiliser of $\ztilde$ inside $\JJ_\bbf$, hence we have
\begin{equation} \label{eq-PIC-preimage}
 \pi^{-1}([b]_\sigma) \cong \iota_{\tilde z}(\Jsf_{{b}}(\Fsf))\backslash \JJ_\bbf/\Ksf\Xi.
\end{equation}

\section{Parametrising Hecke fixed points on finite level Igusa varieties}\label{sect-preliminary-point-counting}

The continuous $\JJ_\bbf/\Xi$ action on the Igusa variety defines an action of its Hecke algebra on the cohomology. We calculate its trace by reinterpreting the operators via Hecke correspondences. From now on, we choose a prime $\ell$ different from the characteristic $p$ of $F$ and work with $\cl\QQ_\ell$-coefficients.

\subsection{Finite level Hecke correspondences}\label{ssect-Hecke}
Given a double-coset $\Ksf g \Ksf \subset \JJ_\bbf$ of $g\in \JJ_{\bbf} $, we denote $[\Ksf g\Ksf] $ the finite \'etale correspondence on $\Ig^\bbf_{\Gsf,\xbf, \Xi\cdot\Ksf}$ induced by the Hecke action of $g$ on $\Ig^\bbf_{\Gsf,\xbf}$, i.e.
\begin{center}
 \begin{tikzcd}
  & \Ig^\bbf_{\Gsf,\xbf, \Xi\cdot\Ksf_g} \arrow[two heads]{dl}[swap]{\gamma_1\colon (\Vscr_0',\tau_0',[\eta_0]) \mapsto (\Vscr_0',\tau_0', [\eta_0])} \arrow[two heads]{dr}{\gamma_2\colon(\Vscr_0',\tau_0',[\eta_0]) \mapsto (\Vscr_0',\tau_0', [\eta_0g])} & \\
  \Ig^\bbf_{\Gsf,\xbf, \Xi\cdot\Ksf} & & \Ig^\bbf_{\Gsf,\xbf, \Xi\cdot\Ksf}
 \end{tikzcd}
\end{center}
where $\Ksf_g \coloneqq \Ksf \cap g \Ksf g\iv$. The scheme of fixed points is defined $\Fix[\Ksf g \Ksf]$ as the fibre product
\begin{center}
 \begin{tikzcd}
  \Fix[\Ksf g \Ksf] \arrow{r} \arrow{d}
  \arrow[dr, phantom, "\scalebox{1.5}{$\lrcorner$}" , very near start, color=black] &
  \Ig^\bbf_{\Gsf,\xbf,\Xi\cdot\Ksf_g} \arrow{d}{(\gamma_1,\gamma_2)} \\
  \Ig^\bbf_{\Gsf,\xbf,\Xi\cdot\Ksf} \arrow{r}{\Delta} &
  \Ig^\bbf_{\Gsf,\xbf,\Xi\cdot\Ksf} \times \Ig^\bbf_{\Gsf,\xbf,\Xi\cdot\Ksf}.
 \end{tikzcd}
\end{center}

Recall that the compactly supported cohomology of a separated Deligne--Mumford stack $S$ locally of finite type over $\cl\FF_q$ (such as $\Ig^\bbf_{\Gsf,\xbf, \Xi\cdot\Ksf}$) is defined as the filtered direct limit of the compactly supported cohomology of its quasi-compact open substacks. Furthermore, the above correspondence induces an endomorphism 
\[\Rrm\Gamma_c([\Ksf g \Ksf])\colon \Rrm\Gamma_c( \Ig^\bbf_{\Gsf,\xbf,\Xi\cdot\Ksf},\cl\QQ_\ell) \to \Rrm\Gamma_c( \Ig^\bbf_{\Gsf,\xbf,\Xi\cdot\Ksf},\cl\QQ_\ell);\] 
indeed, if we fix a quasi-compact open substack $U\subset \Ig^\bbf_{\Gsf,\xbf,\Xi\cdot\Ksf}$ then the above correspondence induces a morphism
\[
\Rrm\Gamma_c(U,\cl\QQ_\ell)\to \Rrm\Gamma_c(\gamma_2(\gamma_1^{-1}U),\cl\QQ_\ell),
\]
and we obtain the desired endomorphism $\Rrm\Gamma_c([\Ksf g \Ksf])$ via passing to the direct limit. (Recall that  we have $\gamma_1^\ast\cl\QQ_\ell \cong \cl\QQ_\ell\cong\gamma_2^!\cl\QQ_\ell$ as $\gamma_2$ is \'etale.)

Recall that we have fixed a model of $\Ig^\bbf_{\Gsf,\xbf}$ defined over some finite extension $\kappa$ of $\FF_q$; \emph{cf.} \S\ref{ssect-infinite-level-igusa} and \S\ref{notation-Hecke-action}. Then the Hecke correspondence $[\Ksf g\Ksf]$ is also defined for the $\kappa$-models as the moduli description (Thm.~\ref{th-infinite-level-igusa-var}(\ref{th-infinite-level-igusa-var:moduli})) descents over $\kappa$; \emph{cf.} Rmk.~\ref{rmk-igusa-variety-over-finite-fields}. 
If furthermore $[\Ksf g\Ksf]$ stabilises $\Ig_{\Gsf,\xbf,\Xi\cdot\Ksf}^{\bbf,\leqslant\mu}$ defined in \eqref{eq-Igusa-HN} (which is a non-trivial assumption unless $\Ig_{\Gsf,\xbf,\Xi\cdot\Ksf}^{\bbf}$ is already quasi-compact),  then we denote by $[\Ksf g \Ksf]^{\leqslant\mu}$ its restriction to  $\Ig_{\Gsf,\xbf,\Xi\cdot\Ksf}^{\bbf,\leqslant\mu}$, which is also defined over the $\kappa$-model. If $\Ksf$ is small enough compared to $\mu$ so that $\Ig_{\Gsf,\xbf,\Xi\cdot\Ksf}^{\bbf,\leqslant\mu}$ is a scheme (\emph{cf.} Cor.~\ref{cor-schematic}), then for big enough $s \gg 0$ the $s$-th iterated Frobeius twist
\[
 [\Ksf g \Ksf]^{\leqslant\mu,(s)} \coloneqq (\Fr^s \circ \gamma_1, \gamma_2)
\]
has a finite \'etale fixed point scheme and that the trace can be calculated by
\begin{equation*}
 \tr\Rrm\Gamma_c([\Ksf g \Ksf]^{\leqslant\mu,(s)}) = \sum_{x \in \Fix [\Ksf g \Ksf]^{\leqslant\mu,(s)}(\k)} \tr(\id \mid \cl\QQ_\ell) = \# \Fix [\Ksf g \Ksf]^{\leqslant\mu,(s)}(\k).
\end{equation*}
This follows from the Lefschetz trace formula (\emph{cf.} \cite[Cor.~5.4.5]{Fujiwara:TraceFormula}, see also \cite[Thm.~5.4.5]{Varshavsky:Lefschetz-Verdier}), and it will be used later to obtain the ``preliminary point-counting formula'' (\emph{cf.} Prop.~\ref{prop-preliminary-trace-formula}).

\begin{rmk}
    Unless  $\Ig_{\Gsf,\xbf,\Xi\cdot\Ksf}^{\bbf}$ is quasi-compact, we do not expect general Hecke correspondences $[\Ksf g\Ksf]$ to stabilise $\Ig_{\Gsf,\xbf,\Xi\cdot\Ksf}^{\bbf,\leqslant\mu}$'s when $\mu$ gets very convex. 
    To illustrate, consider the analogous setting for $\Xscr^Z_{\Gscr,\Ibf,\overline\xbf}$ when $\Gscr$ is a split reductive group over $C$, $Z$ is a Beilinson--Drinfeld Schubert variety and $\overline\xbf\in C^I(\cl\FF_q)$. Then, C.~Xue \cite[Thm.~2]{Xue:Finiteness} showed that the compact support cohomology of the intersection complex of $\Xscr^Z_{\Gscr,\Ibf,\overline\xbf}$ is finitely generated over the unramified Hecke algebra at any $t\in|C|$, and its proof (esp. \cite[Prop.~2.2.4, Lem.~2.2.8]{Xue:Finiteness}) strongly suggests that Hecke correspondences coming from generators of the unramified Hecke algebra at $t\in|C|$ tend to ``destabilise'' sufficiently convex HN truncations. We expect that Hecke correspondences on Igusa varieties should behave similarly.
    
    Nonetheless, one may still hope to obtain an analogue of $[\Ksf g\Ksf]^{\leqslant\mu}$ via systematically working with ``compactifications'' of $\Ig^{\bbf,\leqslant\mu}_{\Gsf,\xbf,\Xi\cdot\Ksf}$. (A similar strategy has already appeared in the work of L.~Lafforgue \cite{Lafforgue:GlobalLanglands} on point-counting for moduli stacks of Drinfeld shtukas.)
We plan to explore this strategy in  future projects.
\end{rmk}

 \subsection{Group theoretic description of Hecke-fixed points}
 \label{ssect-a}
 We can refine the parametrisation in \S\ref{ssect-global-J-orbits} for the $\Fix [\Ksf g \Ksf]$, provided $\Ksf$ is small enough. We closely follow \cite[\S{V.1}]{HarrisTaylor:TheBook}. We consider the following set 
 \begin{align*}
  \FP_{\bbf} &\coloneqq \{ (\underline\Vscr_0',\tau_0') \mid \underline\Vscr_0' \textnormal{ isoshtuka in } \B(\Fsf,\Gsf)_\bbf, \varphi \in \Aut(\Vscr_0') \}/\cong.
 \shortintertext{Trivialising $\Vscr_0'$, we identify}
  \FP_{\bbf} &= \left \{
\big.(b,a) \in \Gsf(\breve\Fsf) \times \Jsf_b(\Fsf) \text{ such that }
[b]_\sigma \in\B(\Fsf,\Gsf)_\bbf  
\right \}/\sim.
 \end{align*}
where the equivalence relation is given by  $(b,a)\sim (g^{-1}b\sigma(g), g^{-1}ag)$ for any $g\in\Gsf(\breve\Fsf)$. Note that $g^{-1} b\sigma(g) = b$ if and only if $g \in \Jsf_b(F)$, so we may regard the equivalence class of $(b,a)$ as a pair $([b]_\sigma,[a])$ where $[b]_\sigma\in\B(\Fsf,\Gsf)_\bbf$ and $[a]$ is a conjugacy class in $\Jsf_b(\Fsf)$ for some representative $b$ of $[b]_\sigma$.
 
By definition $\zeta = (\Vscr_0',\tau_0',\eta_0\Ksf_g\Xi) \in \Ig^\bbf_{\Gsf,\xbf,\Xi\cdot\Ksf_g}(\cl\FF_p)$ is a fixed point for $g \in \JJ_\bbf$ if and only if 
 \begin{equation} \label{eq-FP-automorphism}
  (\Vscr_0',\tau_0',\eta_0\Ksf\Xi) \cong (\Vscr_0',\tau_0', \eta_0g\Ksf\Xi).
 \end{equation}
 Any such isomorphism defines an element of $\FP_\bbf$, which may not be unique if $\Ksf$ is too big. We can reformulate the construction group-theoretically as follows. If $\pi(\Vscr_0',\tau_0',\eta_0) = [b]_\sigma$ and $\zeta = \tilde{z} \cdot y \Ksf_g\Xi$, then $\zeta \in \Fix([\Ksf g\Ksf])$ if and only if $\ztilde y \Ksf = \ztilde y g \Ksf$, i.e.\ there exist $a \in \Jsf_b(\Fsf)$ and $u \in \Xi\Ksf$ such that
 \begin{equation} \label{eq-a}
  yg = \iota_{\ztilde}(a) y u.
 \end{equation}

 We denote by $\widetilde\pi\colon \Fix [\Ksf g \Ksf] \dashrightarrow \FP_\bbf, \zeta \mapsto ([b]_\sigma,[a])$ the associated one-to-many correspondence.

In order to parametrise $\widetilde\pi\iv([b]_\sigma,[\bar a])$, we denote
 \[
 \Xsf_{\Ksf g\Ksf}(b,a) \coloneqq \{y \in \JJ_\bbf \mid  y\iv \iota_{\ztilde}(a)y \in g \Ksf \Xi\}.
 \]  
 Note that up to bijection, the above set only depends on $[b]_\sigma$ and $[a]$.
 
 \begin{lem}\label{lem-preliminary-point-counting}
 The map
\[
 \Xsf_{\Ksf g\Ksf}(b,a) \to \Fix[\Ksf g \Ksf], y\mapsto \ztilde\cdot y \Ksf_g\Xi
\]
 induces a bijection
 \[
\xymatrix@1{\iota_{\tilde z}(\Zsf_{\Jsf_{{b}}(\Fsf)}(a))\,\backslash\, X_{\Ksf g\Ksf}(b,a) \,/\,\Ksf_g\Xi \ar[r]^-{\cong}  & \widetilde\pi^{-1}([b]_\sigma,[a])}.
\]
\end{lem}
\begin{proof}
 This follows directly from \eqref{eq-a} as
 \begin{align*}
  \widetilde\pi\iv([b]_\sigma,[a]) &= \{ \ztilde \cdot y \in \Ig^\bbf_{\Gsf,\xbf,\Xi\cdot\Ksf}(\k) \mid \exists a' \in [a],u \in \Ksf\Xi: y\iv \iota_ {\ztilde}(a')y  = gu\iv \} \\
  &\cong \iota_{\tilde z}(\Jsf_b(\Fsf)) \,\backslash\, \{y \in \JJ_\bbf \mid \exists a' \in [a]: y\iv \iota_{\ztilde}(a') y \in g\Ksf\Xi \}\,/\,\Ksf_g\Xi \\
  &\cong \iota_{\tilde z}(\Zsf_{\Jsf_b(\Fsf)}(a)) \,\backslash\, \{y \in \JJ_\bbf \mid y\iv \iota_{\ztilde}(a)y \in g\Ksf\Xi\}\,/\,\Ksf_g\Xi.
 \end{align*}
\end{proof} 
 
\subsection{The general point counting formula for $\widetilde\pi\iv([b]_\sigma,[a])$}
 We choose a Haar measure on $\JJ_\bbf$ as the product of the Haar measures on $J_{b_v}(\Fsf_v)$ so that hyperspecial maximal open subgroups at all but finitely many unramified places $v\notin\xbf$ have volume~$1$.  We give a counting measure for all discrete subgroups of $\JJ_\bbf$. Finally, for any $a\in\Jsf_{b}(\Fsf)$ we choose a Haar measure of $\Zsf_{\JJ_\bbf}(\iota_{\tilde z}(a))$, inducing a quotient measure on $\Zsf_{\JJ_\bbf}(\iota_{\tilde z}(a))\backslash \JJ_\bbf$.

 \begin{prop}\label{prop-preliminary-point-counting}
 Let $\mu\in\Lambda^+_\QQ$ be a Harder--Narasimhan parameter, and we consider a double coset $[\Ksf g\Ksf]$ as above. Then the cardinality of 
 $(\widetilde\pi^{-1}([b]_\sigma,[a])^{\leqslant\mu}$ equals
\begin{equation} \label{eq-preliminary-point-counting}
 \vol\left (\iota_{\tilde z}(\Zsf_{\Jsf_{b}(\Fsf)}(a)) \backslash \Zsf_{\JJ_\bbf}(\iota_{\tilde z}(a))/\Xi\right )\cdot \int _{\Zsf_{\JJ_\bbf}(\iota_{\tilde z}(a))\backslash\JJ_\bbf} \tfrac{1}{\vol(\Ksf)} \mathbb{1}_{\Ksf g\Ksf} (y^{-1}\iota_{\tilde z}(a) y)\cdot  \mathbb{1}_{\tilde z}^{\leqslant\mu}(y) d\bar y ,
\end{equation}
where $\mathbb{1}^{\leqslant\mu}_{\tilde z}$ is the Harder--Narasimhan truncator function
\[
\mathbb{1}_{\tilde z}^{\leqslant\mu}(y)\coloneqq\left \{
\begin{array}{ll}
1 & \text{if }\tilde{z}\cdot y \in \Ig^{\bbf,\leqslant\mu} _{\Gsf,\xbf}(\k);\\
0 &\text{otherwise.}
\end{array}
\right .
\]
\end{prop} 
\begin{proof}
By Lem.~\ref{lem-preliminary-point-counting} the cardinality of $\pi\iv([b]_\sigma,[a])^{\leqslant\mu}$ is
\begin{equation}\label{eq-preliminary-point-counting-proof}
\int_{\iota_{\tilde z}(\Zsf_{\Jsf_{b}(\Fsf)}(a))\backslash\JJ_\bbf/\Xi}\tfrac{1}{\vol(\Ksf)} \mathbb{1}_{\Ksf g\Ksf} (y^{-1}\iota_{\tilde z}(a) y)\cdot  \mathbb{1}_{\tilde z}^{\leqslant\mu}(y) dy.
\end{equation}
Now using Fubini's theorem, the above integral can be computed by integrating on each left $\left (\iota_{\tilde z}(\Zsf_{\Jsf_{b}(\Fsf)}(a)) \backslash \Zsf_{\JJ_\bbf}(\iota_{\tilde z}(a))/\Xi\right)$-coset  and then integrate over the quotient,  which is $\Zsf_{\JJ_\bbf}(\iota_{\tilde z}(a))\backslash\JJ_\bbf$. 
Now, we rewrite the integral \eqref{eq-preliminary-point-counting-proof} as follows
\[
\int _{\iota_{\tilde z}(\Zsf_{\JJ_\bbf}(a))\backslash\JJ_\bbf}
\int_{\left (\iota_{\tilde z}(\Zsf_{\Jsf_{b}(\Fsf)}(a)) \backslash \Zsf_{\JJ_\bbf}(\iota_{\tilde z}(a))/\Xi\right )\cdot y_0} 
\tfrac{1}{\vol(\Ksf)} \mathbb{1}_{\Ksf g\Ksf} (y^{-1}\iota_{\tilde z}(a) y)\cdot  \mathbb{1}_{\tilde z}^{\leqslant\mu}(y)\cdot dy d\bar y_0 ,
\]
and note that the integrand is constant on each left $\left (\iota_{\tilde z}(\Zsf_{\Jsf_{b}(\Fsf)}(a)) \backslash \Zsf_{\JJ_\bbf}(\iota_{\tilde z}(a))/\Xi\right )$-coset, so the first integral has the effect of multiplying the volume to the integrand, which gives the desired formula \eqref{eq-preliminary-point-counting}.
\end{proof}

\subsection{Disjointness of the fibres of $\widetilde\pi$} 
 In order to obtain a point counting formula for $\Fix [\Ksf g \Ksf](\k)$ from Prop.~\ref{prop-preliminary-point-counting}, we need the following union
 \[
  \Fix([\Ksf g \Ksf]^{\leqslant\mu}) = \bigcup_{([b]_{\sigma},[a]) \in \FP_{\bbf}} \widetilde\pi\iv([b]_\sigma,[a])^{\leqslant\mu}
 \]
 to be disjoint, i.e.\ that the conjugacy class $[a]$ is uniquely defined by $y$. Assume that
 \begin{equation} \label{eq-FP-small-enough-level}
  \iota_{\ztilde}(\Jsf_b(\Fsf)) \cap y \Ksf\Xi y\iv = \{1\}.
 \end{equation}
 Since \eqref{eq-a} is equivalent to
 $
  ygy\iv = \iota_{\ztilde}(a) y u y\iv,
 $
 we conclude that $a \in \Jsf_b(\Fsf)$ and $u \in \Ksf$ only depend on the choice of $y$. Another choice of $y$ corresponds to a $\Jsf_b(\Fsf)$-conjugate of $a$; \emph{cf.} \cite[Proof of Lem.~V.1.2]{HarrisTaylor:TheBook}. Thus, assuming \eqref{eq-FP-small-enough-level} we may associate a unique element $([b]_\sigma,[a]) \in \FP_\bbf$ to $y$.
 
We will show that the assumption  \eqref{eq-FP-small-enough-level} can be achieved in the relevant cases by shrinking $\Ksf$ if necessary. For this, we need the following lemma.

\begin{lem}\label{lem-Ksf-Xi}
 For any compact subgroup $\Csf\subset \JJ_\bbf$, we have $\iota_{\tilde z}(\Jsf_b(\Fsf))\cap \xi \Csf = \emptyset$ for any $\xi\in\Xi\setminus\{1\}$.
\end{lem}
\begin{proof} 
Let $\Dsf\coloneqq\Gsf/\Gsf^{\der}$ be the cocentre of $\Gsf$. For any $x\in\xbf$, the natural map $\Gsf(\breve\Fsf_{x})\to \Dsf(\breve\Fsf_{x})$ restricts to $J_{b_x}(\Fsf_{x})\to \Dsf(\Fsf_{x})$, so we obtain a natural continuous map $\pr\colon\JJ_\bbf\to\Dsf(\AA)$. Similarly, the natural map $\Gsf(\breve\Fsf)\to\Dsf(\breve\Fsf)$ restricts to $\Jsf_b(\Fsf)\to\Dsf(\Fsf)$, which coincides with $\pr\circ\iota_{\tilde z}$ for any $\tilde z$ viewing $\Dsf(\Fsf)$ as a subgroup of $\Dsf(\AA)$.

Set $\Gamma\coloneqq\Gal(\scl\Fsf/\Fsf)$ and define the following continuous map from $\JJ_\bbf$ to a finite free abelian group:
\[
\vartheta\colon
\xymatrix@1
{\JJ_\bbf \ar[r]^{\pr}& \Dsf(\AA)\ar[r]^-{\deg\circ\mathrm{ev}} & (X^\ast(\Dsf)^\Gamma)^\vee\coloneqq\Hom_{\ZZ}(X^\ast(\Dsf)^\Gamma,\ZZ)}
;
\]
i.e., we set $\vartheta(\gbf)\colon \chi\mapsto \deg(\chi(\pr(\gbf)))$ for any $\gsf\in\JJ_\bbf$ and $\chi\in X^\ast(\Dsf)^\Gamma$.
Note that any compact subgroup of $\JJ_\bbf$ lies in the kernel of $\vartheta$ by continuity, and so does  $\iota_{\tilde z}(\Jsf_b(\Fsf))$ by the product formula.

We next claim that $\ker(\vartheta)\cap \Xi = \{1\}$. Indeed, viewing $(X^\ast(\Zsf_\Gsf)^\Gamma)^\vee$ as a finite-index subgroup of $(X^\ast(\Dsf)^\Gamma)^\vee$ via the natural isogeny, $\vartheta$ restricts to the inflation of the natural map $\Zsf_\Gsf(\Fsf)\backslash\Zsf_\Gsf(\AA) \to (X^\ast(\Zsf_\Gsf)^\Gamma)^\vee$, whose kernel is compact. 
Therefore, the image of $\Xi$ injects into $(X^\ast(\Zsf_\Gsf)^\Gamma)^\vee$.

This shows that $\iota_{\tilde z}(\Jsf_b(\Fsf))$ and $\xi\Csf$ are mapped disjointly under $\vartheta$ for any $\xi\in\Xi\setminus\{1\}$, so the lemma follows.
\end{proof}
 \begin{lem} \label{lem-Ksf-small enough}
  Let $\mu\in\Lambda^+_\QQ$ be a  Harder--Narasimhan parameter, and let $\Ksf' g'\Ksf'$ be a double coset in $\JJ_\bbf$ such that $\Fix [\Ksf' g' \Ksf']^{\leqslant\mu}$ is finite. Assume that $\Xi$ is torsion-free. Then there exists an open normal subgroup $\Ksf \subset \Ksf'$ such that \eqref{eq-FP-small-enough-level} is satisfied for any $\zeta \in \Fix[\Ksf g\Ksf]^{\leqslant\mu}(\k)$ and $\Ksf g \Ksf \subset \Ksf' g' \Ksf'$. In particular \eqref{eq-a} induces a well-defined map
  \[
   \widetilde\pi^{\leqslant\mu}\colon \xymatrix@1{
\Fix[\Ksf g\Ksf]^{\leqslant\mu}(\k) \ar[rr]^-{\zeta\mapsto ([b]_\sigma,[ a])}&& \FP_\bbf}.
  \]
 \end{lem}
 \begin{proof}
 Note that $\Jsf_b(\Fsf)$ is a discrete subgroup of $\Jsf_{b}(\AA)$, so its image under $\iota_{\tilde z}$ is a discrete subgroup of $\JJ_\bbf$. Recall that for any $y' \in \JJ_\bbf$ we have \[\iota_{\ztilde}(\Jsf_b(\Fsf)) \cap y' \Ksf' {y'}\iv = \iota_{\ztilde}(\Jsf_b(\Fsf)) \cap y' \Ksf' \Xi {y'}\iv\] by Lem.~\ref{lem-Ksf-Xi} applied to $\Csf = y'\Ksf {y'}\iv$, and the left hand side   is clearly finite.
 
 Thus there exists $\Ksf \subset \Ksf'$ such that \eqref{eq-FP-small-enough-level} is satisfied for every $y\Ksf\Xi \subset y'\Ksf'\Xi$. Since $\Fix [\Ksf'g'\Ksf']^{\leqslant\mu}$ is finite, we can thus choose $\Ksf$ small enough so that the equality~\eqref{eq-FP-small-enough-level} holds for all $\zeta \in \Fix[\Ksf g \Ksf]^{\leqslant\mu}$.
 \end{proof}
 
 \subsection{The Hecke action on the cohomology of Igusa varieties}
 
 We fix a $\cl\QQ_\ell$-valued Haar measure on $\JJ_{\bbf}$ once and for all, and consider the following Hecke algebra
\begin{align*}
\Hcal \coloneqq \Hcal^{\bbf}_{\Gsf,\xbf,\Xi}(\cl\QQ_\ell) &\coloneqq C^\infty_c\left(\JJ_{\bbf}/\Xi;\cl\QQ_\ell \right) 
\end{align*}
From the  smooth action of $\JJ_{b}/\Xi $ we get a convolution action of the Hecke algebra $\Hcal$ on $\coh{i}_{c}(\Ig^{\bbf} _{\Gsf,\xbf,\Xi},\cl\QQ_\ell) $ via convolution.

The pullback along the canonical projection defines an embedding of cohomology groups $\coh{i}_{c}(\Ig^{\bbf} _{\Gsf,\xbf,\Xi\cdot\Ksf},\cl\QQ_\ell) \mono \coh{i}_{c}(\Ig^{\bbf} _{\Gsf,\xbf,\Xi},\cl\QQ_\ell)$ identifying $\coh{i}_{c}(\Ig^{\bbf} _{\Gsf,\xbf,\Xi\cdot\Ksf},\cl\QQ_\ell)$ with $\coh{i}_{c}(\Ig^{\bbf} _{\Gsf,\xbf,\Xi},\cl\QQ_\ell)^\Ksf$. Moreover, we define
\begin{align*}
 p_K \colon \coh{i}_{c}(\Ig^{\bbf} _{\Gsf,\xbf,\Xi},\cl\QQ_\ell) &\to \coh{i}_{c}(\Ig^{\bbf} _{\Gsf,\xbf,\Xi\cdot\Ksf},\cl\QQ_\ell) \\
 v &\mapsto \int_\Ksf g\cdot v\, dg.
\end{align*}

We can interpret the convolution action of $\Hcal$ on $\coh{i}_{c}(\Ig^{\bbf} _{\Gsf,\xbf,\Xi},\cl\QQ_\ell) $ as a cohomological correspondence. We focus on $\varphi = \mathbf{1}_{\Ksf g \Ksf}$ as any $\varphi \in \Hcal$ can be written as
\[
 \varphi = \sum_{g \in I} \alpha_g \cdot \mathbf{1}_{\Ksf g \Ksf},
\]
for some finite set $I \subset G$, $\alpha_g \in \cl\QQ_\ell$ and compact open subgroup $\Ksf \subset \JJ_\bbf$. By the same argument as in \cite[\S6, Eq.~(11)]{ShinSW:IgusaPointCounting} (see also \cite[\S3.1.1]{ShinSW:thesis}), the action of $\mathbf{1}_{\Ksf g \Ksf}$ on $\coh{i}_{c}(\Ig^{\bbf} _{\Gsf,\xbf,\Ksf\Xi},\cl\QQ_\ell)$ is identical with $[\Ksf g \Ksf] \circ p_\Ksf$. Hence, we obtain the following proposition.
\begin{prop}
    Given a double coset $\Ksf g \Ksf \subset \JJ_\bbf$, we have 
\begin{equation} \label{eq-geometric-interpretation-trace}
 \tr(\mathbf{1}_{\Ksf g \Ksf} \mid  \Rrm\Gamma_c(\Ig^{\bbf} _{\Gsf,\xbf,\Xi},\cl\QQ_\ell)) = \vol(\Ksf)\tr(\Rrm\Gamma_c[\Ksf g \Ksf]).
\end{equation}
\end{prop}
An analogous result holds for $\tr(\mathbf{1}_{\Ksf g \Ksf} \mid  \Rrm\Gamma_c(\Ig^{\bbf,\leqslant\mu} _{\Gsf,\xbf,\Xi},\cl\QQ_\ell))$ if $[\Ksf g\Ksf]$ restricts to a correspondence $[\Ksf g\Ksf]^{\leqslant\mu}$ on $\Ig^{\bbf,\leqslant\mu} _{\Gsf,\xbf,\Xi}$.

 \subsection{Frobenius twists}\label{ssect-Frob-twist} 
 In order to apply Lefschetz' trace formula to \eqref{eq-geometric-interpretation-trace}, we may need to twist the correspondence by a power of the Frobenius. We note that for $(\Vscr_0',\tau_0',\eta_0') \in \Ig^{\bbf}_{\Gsf,\xbf,\Xi\cdot\Ksf}(\cl\FF_p)$ the $r$-th power $\tau_0'^r$ defines an isomorphism
 \[
  \Fr^\ast(\Vscr_0',\tau_0') \isom (\Vscr_0',\tau_0').
 \]
 As $\eta_0\iv(\tau_0') = \bbf\sigma$, we obtain
 \[
  \Fr^\ast(\Vscr_0',\tau_0',\eta_0') \isom (\Vscr_0',\tau_0',\eta_0' \circ N^{(r)}(\bbf)).
 \]
 where $N^{(r)}(\bbf) \coloneqq \bbf \cdot \dotsm \cdot \sigma^{r-1}(\bbf) \in \JJ_\bbf$. Note that $N^{(r)}(\bbf)$ is a central element of $\JJ_\bbf$, made explicit in \eqref{eq-decent}. We obtain
 \[
  [\Ksf g \Ksf]^{(s)} = [\Ksf g N^{(r)}(\bbf) \Ksf].
 \] 
This motivates the following definition.

\begin{defn}
 For any $\varphi\in\Hcal_\Ksf$ and for any positive integer $s$
\begin{equation*} 
\varphi^{(s)}\colon g\mapsto \varphi(g\cdot N^{(r)}(\bbf)), \quad\forall g\in \JJ_\bbf.
\end{equation*}
\end{defn}

 Putting the results of this section together, we obtain the trace formula in its most general form.
 
\begin{prop}\label{prop-preliminary-trace-formula} 
For any $\varphi\in\Hcal$ such that any $\Ksf$-double coset in $\supp \varphi$ stabilises $\Ig^{\bbf,\leqslant\mu}_{\Gsf,\xbf,\Xi\cdot\Ksf}$, there exists an integer $s\gg1$ so that for any small enough open normal subgroup  $\Ksf$ of $\Gamma_\bbf$ we have
\[
\tr(\Rrm\Gamma_c([\varphi^{(s)}]_{\Ksf}^{\leqslant\mu}) =
 \sum_{([b]_\sigma,[a])\in\FP_\bbf}
\int _{\iota_{\tilde z}(\Zsf_{\Jsf_{b}(\Fsf)}(a)) \backslash\JJ_\bbf/\Xi}  \varphi^{(s)} (y^{-1}\iota_{\tilde z}(a) y)\cdot  \mathbb{1}_{\tilde z}^{\leqslant\mu}(y) d\bar y,
\]
where $\mathbb{1}^{\leqslant\mu}_{\ztilde}$ is the indicator function of $\Ig^{\bbf,\leqslant\mu}_{\Gsf,\xbf,\Xi\cdot\Ksf}$. In particular, if $\Ig^\bbf_{\Gsf,\xbf,\Xi\cdot\Ksf}$ is quasi-compact, we obtain
\[
\tr(\Rrm\Gamma_c([\varphi^{(s)}]_{\Ksf}) =
 \sum_{([b]_\sigma,[a])\in\FP_\bbf}\vol\left (\iota_{\tilde z}(\Zsf_{\Jsf_{b}(\Fsf)}(a)) \backslash \Zsf_{\JJ_\bbf}(\iota_{\tilde z}(a))/\Xi\right )\cdot O^{\JJ_{\bbf}}_{\iota_{\tilde z}(a)}(\varphi^{(s)}),\]
 where $O^{\JJ_{\bbf}}_{\iota_{\tilde z}(a)}(\varphi^{(s)})\coloneqq \int _{\Zsf_{\JJ_\bbf}(\iota_{\tilde z}(a))\backslash\JJ_\bbf}  \varphi^{(s)} (y^{-1}\iota_{\tilde z}(a) y) d\bar y$ is an orbital integral.
\end{prop}
\begin{proof}
 Since both sides of the formula are linear in $\varphi$, it suffices to show the claim for $\varphi = \mathbb{1}_{\Ksf' g' \Ksf'}$. We choose $s \gg 0$ such that Lefschetz trace formula applies, in particular $\Fix [\Ksf' g' \Ksf']^{\leqslant\mu,(s)}$ is finite \'etale. After shrinking $\Ksf$ and applying linearity once more, we may assume that $\varphi^{(s)} = \mathbb{1}_{\Ksf g \Ksf}$, with $\Ksf$ and $g$ as in Lem.~\ref{lem-Ksf-small enough}. By  trace formula and Prop.~\ref{prop-preliminary-point-counting} we obtain
 \begin{align*}
  \tr(\Rrm\Gamma_c([\varphi^{(s)}]_{\Ksf}) &= \sum_{z \in Fix[\Ksf g \Ksf]^{\leqslant\mu}} \tr(\varphi^{(s)}\mid \cl\QQ_\ell) \\
  &= \sum_{([b]_\sigma,[a])\in\FP_\bbf}
\int _{\iota_{\tilde z}(\Zsf_{\Jsf_{b}(\Fsf)}(a)) \backslash\JJ_\bbf/\Xi}  \varphi^{(s)} (y^{-1}\iota_{\tilde z}(a) y)\cdot  \mathbb{1}_{\tilde z}^{\leqslant\mu}(y) d\bar y.
 \end{align*}
\end{proof}

\begin{rmk}\label{rmk-nonqc}
If $\Ig^{\bbf}_{\Gsf,\xbf,\Xi}$ is not quasi-compact then it is quite unlikely for $\varphi\in\Hcal$ to stabilise any quasi-compact open cover. But even in that case, it seems reasonable to expect that for ``interesting enough'' class of test functions $\varphi\in\Hcal$ the Harder--Narasimhan truncated local terms $\sum _{\bar y}\tr([\varphi^{(s)}]^{\leqslant\mu}_{\Ksf,\bar y})$ should stabilise as we make $\mu$ more and more convex. (See Prop.~\ref{prop-preliminary-point-counting-elliptic}, Conj.~\ref{conj-EP} and the subsequent paragraphs for further discussions.) In that case, we cautiously speculate that there should be some version of ``Lefschetz trace formula'' applicable to such $\varphi$'s, possibly by constructing the ``HN truncated Hecke correspondences'' as speculated at the end of \S\ref{ssect-Hecke}. This is one of the ideas behind  Conj.~\ref{conj-trace-formula}.
\end{rmk}

In the following sections we will work on making the formula in Prop.~\ref{prop-preliminary-trace-formula} more explicit, restricting  to special cases if necessary.

\section{Elliptic terms of the trace formula} \label{sect-discrete}

 \subsection{The Newton point of a conjugacy class}
 We briefly recall the non-$\sigma$-twisted version of the Newton point as introduced by Kottwitz and Viehmann in \cite[\S2]{MR2931318}. Let $G$ be a linear algebraic group over a local field $F$ and $\gamma \in G(F)$. If $G = \GL_n$ then $\nu^{KV}_{G}(\gamma)\colon \DD \to \GL_n $ is the (rational) quasi-cocharacter given by the decomposition
 \[
  F^n = \bigoplus_{a \in \QQ} V_a,
 \]
 where $V_a$ denotes the direct sum of all generalised eigenspaces of $\gamma$ with valuation $a$. More generally, $\nu^{KV}_{G}(\gamma)$ is the unique functorial extension of this definition to linear algebraic groups. We further extend this definition to the case that $F$ is a product of local non-archimedean local fields the obvious way.
 
 By \cite[L.~2.2]{MR2931318} the $\cl{F}$-conjugacy class $\bar\nu^{KV}_G(\gamma)$ of $\nu^{KV}_G(\gamma)$ only depends on the stable conjugacy class of $\gamma$. For any $\bar\nu \in \Hom_{\cl{F}}(\DD,G)/G(\cl{F})$ we define the corresponding \emph{Newton stratum} in $G(F)$ by
 \[
  G(F)_{\bar\nu} \coloneqq \{\gamma \in G(F) \mid \bar\nu^{KV}_G(\gamma) = \bar\nu \}.
 \]
 \begin{lem}
  For any $\bar\nu$ as above, $G(F)_{\bar\nu}$ is an open and closed subset of $G(F)$.
 \end{lem} 
 \begin{proof}
  Let $\rho\colon G \mono \GL_n$ be a closed embedding such that the induced map \[\rho_\ast\colon \Hom_{\cl{F}}(\GG_m,G)/G(\cl{F}) \to \Hom_{\cl{F}}(\GG_m,\GL_n)/\GL_n(\cl{F})\] is injective. In particular, we have
  \[
   G(F)_{\bar\nu} = \GL_n(F)_{\rho_\ast(\bar\nu)} \cap G(F).
  \]
  Thus we may assume $G = \GL_n$. Let $T \subset B \subset \GL_n$ be the diagonal torus and the Borel of upper triangular matrices. We identify $X_\ast(T) \isom \ZZ^n$ and $\bar\nu$ with its unique representative in $X_\ast(T)_{\QQ,+} \cong \QQ^n_+$.  Let $\chi\colon \GL_n \to \{P \in F[X] \mid P \textnormal{ normalised polynomial of degree } n \} \cong F^{n-1} \times F^\times$ be the map sending $\gamma \in \GL_n(F)$ to its characteristic polynomial. By definition $\bar\nu^{KV}_G$ factors through the map $v\colon F^{n-1} \times F^\times \to \QQ^n_+$ sending $P$ to the (descending-order) tuple of valuations of its roots. These are given by the slopes of its Newton polygon. Thus the set $v\iv (\bar\nu) \subset F^{n-1} \times F^\times$ is given by a finite set of inequalities on the valuations of the coefficients, which defines an open and closed set. Since \revise{$\chi$} is continuous, $G(F)_{\bar\nu} = \revise{\chi}\iv(v\iv(\bar\nu))$ is also open and closed. 
 \end{proof}
 
 In the following let $F$ be a product of non-archimedean local fields, $G$ a reductive linear algebraic group over $F$ and let $\nu\colon \DD \to G$ be a quasi-cocharacter. We denote by $M_\nu \subset P_\nu \subset G$ the associated Levi subgroup and parabolic subgroup, that is $\nu$ acts on $\Lie P_\nu$ by non-negative weights.
 
 \begin{defn} \label{defn-acceptable-elt}
  We say that $\gamma \in M_{\nu}(F)$ is \emph{$\nu$-acceptable} if $P_{\nu^{KV}_G(\gamma)} \subset P_\nu$.
 \end{defn} 
  
 \begin{rmk}
  \begin{subenv}
  \item Since the definition of $\nu$-acceptability only depends on $\bar\nu^{KV}_G(b)$, it only depends on the geometric $M_\nu(\Fbar)$-conjugacy class of $\gamma$. In particular, the set of $\nu$-acceptable elements is a union of Newton strata in $M_\nu(F)$ and hence open and closed.
  \item One easily checks using above decomposition that when $F$ is a local field and $E/F$ a finite extension, then
   \[
     \nu_{G_{E}}^{KV}(\gamma) = e(E/F) \cdot \nu_{G}^{KV}(\gamma),
   \]
   and deduces a similar claim for general $F$. In particular, $P_{\nu^{KV}_G}(\gamma)$ does not change and thus the notion of acceptability is stable under extension of $F$.
   \end{subenv}
 \end{rmk}  
  
 To check acceptability, we will use the following combinatorial lemma.
 
 \begin{lem}
  Fix a maximal torus $T \subset G$ containing the image of $\nu$ and denote by $\Phi(G,T)(\Fbar)$ the set of absolute roots.  For any quasi-cocharacter $\xi\colon \DD \to G$ the following conditions are equivalent
  \begin{assertionlist}
   \item $P_\xi \subset P_\nu$
   \item For every $\alpha \in \Phi(G,T)(\Fbar)$ with $\langle \alpha, \nu \rangle < 0$ we have $\langle \alpha,\xi \rangle < 0$.
  \end{assertionlist}
 \end{lem}
 \begin{proof}
  The first condition can be checked on the Lie algebra of $G$. By considering the root spaces, we get that $\Lie P_\xi \subset \Lie P_\nu$ is equivalent to
  \[
   \{ \alpha \in \Phi(G,T) \mid \langle \alpha, \xi \rangle \geq 0 \} \subset \{ \alpha \in \Phi(G,T) \mid \langle \alpha, \nu \rangle \geq 0 \}
  \]
  or equivalently that $\langle \alpha, \nu \rangle < 0$ implies that $\langle \alpha, \xi \rangle < 0$.
 \end{proof}
 
 \begin{cor} \label{cor-centralising}
  Let $\gamma_1, \gamma_2 \in M_\nu(F)$ be $\nu$-acceptable and assume that there exists $g \in G(\Fbar)$ such that $g \gamma_1 g\iv = \gamma_2$. Then $g \in M_\nu(\Fbar)$.
 \end{cor} 
 \begin{proof}
  By \cite[L.~2.2]{MR2931318} we have $\nu^{KV}_G(\gamma_2) = \Int(g) \circ \nu^{KV}_G(\gamma_1)$. Replacing $\gamma_1,\gamma_2$ by $M(\Fbar)$-conjugates and extending $F$ accordingly if necessary, we can assume that $\nu^{KV}_G(\gamma_1)$ and $\nu^{KV}_G(\gamma_2)$ factorise over $T$, thus $g \in \Norm_G(T)(\Fbar)$. By the previous lemma $g$ stabilises the set $\Phi^-(G,T) \setminus \Phi^-(M,T)$, proving that $g \in M_\nu(\Fbar)$.
 \end{proof}
 
\subsection{Acceptable elements in $\JJ_\bbf$}
 Recall that we fixed a \emph{decent} element $b_x\in\Gsf(\Fsf_{x} \hat\otimes_{\FF_q} \cl\FF_q)$ for each $x\in\xbf$ (\emph{cf.} \S\ref{notation-Hecke-action}), and let $\nu_x\coloneqq \nu_{\Gsf}(b_x)$ denote the Newton point associated to the local $\Res_{\Fsf_{x}/\FF_q\rpot{\varpi_{x}}}\Gsf_{x}$-isoshtuka defined by $b_x$ (not the Kottwitz--Viehmann Newton point of the conjugacy class of $b_x$). By decency of $b_x$, note that $\nu_x\colon \DD\to\Gsf_{\Fsf_{x}}$ is defined over $\Fsf_{x}$, not just over a finite extension of it. 
 
 As  $b_x\in\Gsf(\Fsf_{x} \hat\otimes_{\FF_q} \cl\FF_q)$ is a fixed element, we have an embedding 
 \[
 J_{b_x}(\Fsf_{x})\hookrightarrow M_{\nu_x}(\breve\Fsf \hat\otimes_\Fsf \Fsf_{x})\subset \Gsf (\breve\Fsf \hat\otimes_\Fsf \Fsf_{x}),
 \] 
 which depends on $b_x$. We say that $\delta \in J_{b_x}(\Fsf_{x})$ is \emph{$\nu_x$-acceptable} if it is $\nu_x$-acceptable as \revise{an} element of $M_{\nu_x}(\breve\Fsf \hat\otimes_\Fsf \Fsf_{x})$ as defined in Def.~\ref{defn-acceptable-elt}. 

Setting $\bnu\coloneqq (\nu_x)_{x\in\xbf}$, we say that $g\in\JJ_\bbf$ is \emph{$\bnu$-acceptable} if for any $x\in\xbf$ the image of $g$ under the projection $\JJ_\bbf\twoheadrightarrow J_{b_x}(\Fsf_{x})$ is $\nu_x$-acceptable.

Note that $\nu_x$-acceptability is an empty condition  when $b_x$ is basic (i.e, when $\nu_x$ is central), so $\bnu$-acceptability for an element of $\JJ_\bbf$ is a condition on the components at non-basic places.
 
 \subsection{``Acceptability condition'' for functions on $\JJ_\bbf$}
Let $\Ksf g\Ksf$ be a double coset  of $\JJ_\bbf$ and consider the correspondence $[\Ksf g\Ksf]$ on $\Ig^{\bbf}_{\Gsf,\xbf,\Xi\cdot\Ksf}$.  Let $\mu\in\Lambda^+_\QQ$ be a Harder--Narasimhan parameter, and we set
 \begin{equation}
 \Fix[\Ksf g\Ksf]^{\leqslant\mu}\coloneqq \Fix([\Ksf g\Ksf])\cap \Ig^{\bbf,\leqslant\mu}_{\Gsf,\xbf,\Xi\cdot\Ksf\cap g\Ksf g^{-1}}.
 \end{equation}

For $\mu\in\Lambda^+_\QQ$, let us denote the following conditions as $\Acc(\bnu;\mu)$.
 \begin{assertionlist}
  \item Any $g \in \supp\varphi$ is $\bnu$-acceptable.
  \item There exists a small enough open compact subgroup $\Ksf$ such that $\varphi$ is $\Ksf$-bi\-in\-va\-riant, and for any double coset $\Ksf g\Ksf$ in the support of $\varphi$ we have that $ \Fix[\Ksf g\Ksf]^{\leqslant\mu} $ is finite \'etale over the base finite field.
 \end{assertionlist}

Note that if the second condition holds for $\Ksf$, then it holds for any open compact subgroup $\Ksf'$ of $\Ksf$; indeed, if $\Ksf' g'\Ksf'\subset \Ksf g\Ksf$ then the natural projection map sends $\Fix ([\Ksf' g'\Ksf'])$ to $\Fix ([\Ksf g\Ksf])$.

 \begin{lem}\label{lem-acceptability}
Choose a positive integer $r$ such that $N^{(r)}b_x = r\nu_x(\varpi_{x})\in\Gsf(\Fsf_{x})$.  For any $\varphi\in \Hcal$ and $\mu\in\Lambda^+_\QQ$, there exists an $s_0 \in \NN$ (depending on $\mu$) such that $\varphi^{(rs')} \in \Acc(\bnu,\mu)$ for any $s' \geqslant s_0$. 
 \end{lem}
 \begin{proof}
 The second condition of acceptability can be achieved by replacing $\varphi$ with some iterated Frobenius-twist by \cite[Thm.~2.3.2(a)]{Varshavsky:Lefschetz-Verdier}. To arrange the first condition, we choose a closed embedding $\Gsf \mono \GL(V)$ and fix $1 \leqslant i \leqslant m$ and a place $y|x$ of $\tilde\Fsf$. Let $\bigoplus V_\lambda = V_{(\breve\Fsf)_{y}}$ be the slope decomposition of $(V_{(\breve\Fsf)_{y}},(N^{(r)}b_x)_y \sigma^r)$. We have to show that for $s'$ big enough and any $g \in \supp(\varphi)$ and $\lambda_1 > \lambda_2$, the valuation of any eigenvalue of $(b_x\sigma)^{rs'}(g_{x}) = N^{(rs')}(b_x) \cdot \sigma^{rs'}(g_{x})$ on $V_{\lambda_1}$ is larger than the valuation of any eigenvalue on $V_{\lambda_2}$. Since
 \[  
  \restr{N^{(rs')}(b_x) \cdot \sigma^{rs'}(g_{x})}{V_\lambda} = \unif_{x}^{\lambda r s'} \cdot \restr{\sigma^{rs'}(g_{x})}{V_\lambda},
 \]
 this is clearly true.
 \end{proof}

Next we separate the ``elliptic semisimple terms'' from Prop.~\ref{prop-preliminary-point-counting}. 
In what follows, let us fix a base point $\tilde z\in \pi^{-1}([b])\subset \Ig^{\bbf}_{\Gsf,\xbf,\Xi}(\k)$ for any $[b] \in \B(\Fsf,\Gsf)_\bbf$.

\begin{defn}\label{def-FP-erS} Let $([b]_\sigma,[a])\in\FP_\bbf$.
\begin{subenv}
\item We say that $([b]_\sigma,[a])$ is \emph{acceptable} if for some (equivalently, any) choice of $(b,a)$  and $\tilde z\in\pi^{-1}([b])$, the element $\iota_{\tilde z}(a)\in\JJ_\bbf$ is $\bnu$-acceptable in the sense of Def.~\ref{defn-acceptable-elt}.
\item We say that $([b]_\sigma,[a])$ is \emph{semisimple} if $a$ is semisimple in $\Jsf_b(\Fsf)$ for some (equivalently, any) representative $(b,a)$.
\item We say that  $([b]_\sigma,[a])$ is \emph{elliptic} if  $a\in \Jsf_b(\Fsf)$ is contained in a maximal torus totally anisotropic modulo $\Zsf_\Gsf$ for any representative $(b,a)$.
\item For any place $v$ such that $b_v$ is \emph{basic}, 
we say that $([b]_\sigma,[a])$ is \emph{elliptic at $v$} if for some (equivalently, any) choice of $(b,a)$  and $\tilde z\in\pi^{-1}([b])$, $\iota_{\tilde z}(a)$ projects to an elliptic semisimple element in $J_{b_v}(\Fsf_{v})$, which is an inner form of $G_{\Fsf_{v}}$. 

For a finite non-empty subset $S$ of places of $\Fsf$ such that $b_v$ is basic for any $v\in S$, we similarly define the notion of being \emph{elliptic at $S$}.
\end{subenv}

Let $\FP^{\rm ss}_\bbf$ be the set of $([b]_\sigma,[a])\in\FP_\bbf$ that are  \emph{semisimple}, and  let $\FP^{\ael}_\bbf\subset \FP^{\rm ss}_\bbf$ denote the subset of \emph{acceptable elliptic} elements.
We similarly define the subset $\FP^{\ael(S)}_{\bbf}\subset \FP^{\ael}_\bbf$ of elements that are elliptic at $S$. 
\end{defn}
We will mainly apply these definitions when $\Gsf^{\der}$ is simply connected so the notion of being elliptic matches with the corresponding notion for semisimple conjugacy classes of $\Gsf(\Fsf)$.

\begin{lem}\label{lem-acceptable-centraliser}
If $([b]_\sigma,[a])\in\FP_\bbf$ is  acceptable, then for any representative $(b,a)$ we have
\[
\iota_{\tilde z}(\Zsf_{\Jsf_b(\AA)}(a)) = \Zsf_{\JJ_\bbf}(\iota_{\tilde z}(a)).
\]
\end{lem}
\begin{proof}
We clearly have $\iota_{\tilde z}(\Zsf_{\Jsf_b(\AA)}(a)) \subset \Zsf_{\JJ_\bbf}(\iota_{\tilde z}(a))$. It remains to show that any $g  \in \Zsf_{\JJ_\bbf}(\iota_{\tilde z}(a))$ is contained in the image of $\Jsf_b(\AA) = \Msf_{b}(\AA _{\breve\Fsf})\cap \JJ_\bbf $ via $\iota_{\tilde z}$. By Cor.~\ref{cor-loc-gl-Newton-pt} $g$ centralises $\iota_{\tilde z}(\nu_{\Gsf}(b))$ if and only if it centralises all $\nu_x$. Now, the desired claim follows from Cor.~\ref{cor-centralising}.
\end{proof}

\begin{lem}\label{lem-elliptic-orbital-int}
If $([b]_\sigma,[a])$ is \emph{acceptable} and \emph{elliptic}, \revise{then} $\iota_{\tilde z}(\Zsf_{\Jsf_b(\Fsf)}(a)) \backslash X_{\Ksf g\Ksf}(\iota_{\tilde z}(a))/\Ksf_g\Xi$ is a finite set and the orbital integral $O^{\JJ_\bbf}_{\iota_{\tilde z}(a)}(\mathbb{1}_{\Ksf g \Ksf})$ is well defined. 
\end{lem}

\begin{proof}
Choose a representative $(b,a)$ of $([b]_\sigma,[a])$.
Since $([b]_\sigma,[a])$ is acceptable, we have $\iota_{\tilde z}(\Zsf_{\Jsf_b(\AA)}(a)) = \Zsf_{\JJ_\bbf}(\iota_{\tilde z}(a))$ by Lem.~\ref{lem-acceptable-centraliser}. As $a\in \Jsf_b(\Fsf)$ is elliptic, it follows that $\Zsf_{\Jsf_b(\Fsf)}(a)\backslash\Zsf_{\Jsf_b(\AA)}(a)/\Zsf_\Gsf(\AA)$ has finite volume. Since the image of $\Xi$ in $\Zsf_{\Gsf}(\Fsf)\backslash\Zsf_\Gsf(\AA)$ is discrete and cocompact, it follows that $\Zsf_{\Jsf_b(\Fsf)}(a)\backslash\Zsf_{\Jsf_b(\AA)}(a)/\Xi$ has finite volume.

If $a$ \revise{is semi-simple} then $c_a\colon \iota_{\tilde z}(\Zsf_{\Jsf_b(\Fsf)}(a))\backslash \JJ_\bbf \to \JJ_\bbf, y\mapsto y^{-1}\iota_{\tilde z}(a)y$ is a closed embedding. Hence, $c_a\iv(\Ksf g \Ksf)$ is compact and $O^{\revise{\JJ_\bbf}}_{\iota_{\tilde z}(a)}(\mathbb{1}_{\Ksf g \Ksf}) = \vol (c_a^{-1}(\Ksf g \Ksf))$ is finite.
\end{proof}

Let $\varphi\in\Acc(\mu,\bnu)$ for some Harder--Narasimhan parameter $\mu\in\Lambda^+_\QQ$. Choose a finite subset $S \subset |C|$ such that $b_x$ is \emph{basic} for any  $x\in S \cap \xbf$. We define the $\mu$-truncated $S$-elliptic term of $\tr(\varphi)^{\el(S),\leqslant\mu}$ as follows. Choose a small enough open compact subgroup $\Ksf\subset\JJ_\bbf$ such that $\varphi$ is $\Ksf$-bi-invariant and $\widetilde\pi_{\leqslant\mu}\colon\Fix([\Ksf g\Ksf]^{\leqslant\mu})\to\FP_{\bbf}$ is well defined for any $\Ksf g\Ksf\subset\supp(\varphi)$. Then  we set
\begin{equation}
\tr(\varphi)^{\el,\leqslant\mu} 
 \coloneqq \sum_{\bar y} \tr([\varphi]_{\Ksf,\bar y}) ,
\end{equation}
where   the sum is over  $\bar y\in\Fix([\Ksf g\Ksf]^{\leqslant\mu})(\k)$ for some  $\Ksf g\Ksf\subset\supp(\varphi)$ such that $\widetilde\pi_{\leqslant\mu}(\bar y) \in\FP^{\ael}_{\bbf}$, and the summand is the local term of $[\varphi]_\Ksf$  at $\bar y$. Note that due to the compatibility of $\widetilde\pi$ and base change this does not depend on the choice of $\Ksf$. We similarly define $\tr(\varphi)^{\el(S),\leqslant\mu} $ by making the sum go over $\bar y\in \widetilde\pi_{\leqslant\mu}^{-1}(\FP^{\ael(S)}_{\bbf})$.

\begin{prop}\label{prop-preliminary-point-counting-elliptic}
Assume that for any $\varphi\in\Hcal$ there are only finitely many $([b]_\sigma,[a])\in\FP^{\ael}_{\bbf}$ such that $O^{\JJ_\bbf}_a(\varphi)$ is non-zero.  Then for any sufficiently convex Harder--Narasimhan parameter $\mu\in\Lambda^+_\QQ$ and divisible enough positive integer $s$, we have
\[\tr(\varphi^{(s)})^{\el,\leqslant\mu}  =
 \sum_{([b]_\sigma,[a])\in\FP^{\ael}_{\bbf}}\vol\left (\Zsf_{\Jsf_{b}(\Fsf)}(a) \backslash \Zsf_{\Jsf_{b}(\AA)}(a)/\Xi\right )\cdot 
O^{\JJ_{\bbf}}_{\iota_{\tilde z}(a)}(\varphi^{(s)}).\]

Similarly, if $S$ is a subset of $\xbf$ such that $b_x$ is basic for any $x\in S$, then for any sufficiently convex $\mu\in\Lambda^+_\QQ$ and divisible enough positive integer $s$, we have
\[\tr(\varphi^{(s)})^{\el(S),\leqslant\mu}  =
 \sum_{([b]_\sigma,[a])\in\FP^{\ael(S)}_{\bbf}}\vol\left (\Zsf_{\Jsf_{b}(\Fsf)}(a) \backslash \Zsf_{\Jsf_{b}(\AA)}(a)/\Xi\right )\cdot 
O^{\JJ_{\bbf}}_{\iota_{\tilde z}(a)}(\varphi^{(s)}).\]
\end{prop}

\begin{proof}
Let us prove the statement on $\tr(\varphi)^{\el,\leqslant\mu}$, as the rest follows from this. We note that by Lem.~\ref{lem-elliptic-orbital-int} and Lem.~\ref{lem-preliminary-point-counting} the set $\widetilde\pi\iv([b]_\sigma,[a])$ is finite for any $\Ksf g \Ksf \subset \supp \varphi$ and hence contained in $[\Ksf g \Ksf]^{\leqslant\mu}$ for sufficiently convex $\mu$. Choosing $\mu$ convex enough for all relevant choices of $([b]_\sigma,[a])$ and $\Ksf g \Ksf$, we apply Lem.~\ref{lem-Ksf-small enough} finitely many times so that \eqref{eq-a} is satisfied for any $\Ksf g \Ksf \subset \supp \varphi$ and $y \in \widetilde\pi\iv([b]_\sigma,[a])$. Now the claim follows from Prop.~\ref{prop-preliminary-point-counting} for $\mu \to \infty$.
\end{proof}

If $\varphi$ and $\mu$ satisfy the conclusion of above Proposition, then we write
\begin{align}\label{eq-ell-terms}
 \tr(\varphi^{(s)})^{\el} &\coloneqq\tr(\varphi^{(s)})^{\el,\leqslant\mu},\ \text{and} \\
 \tr(\varphi^{(s)})^{\el(S)} &\coloneqq\tr(\varphi^{(s)})^{\el(S),\leqslant\mu}\notag
\end{align}
and call it the \emph{elliptic term} and \emph{$S$-elliptic term} of the trace of $\varphi^{(s)}$, respectively. The elliptic term $\tr(\varphi^{(s)})^{\el} $ may not admit any natural interpretation in general, but there is a special case when any local term of $\varphi^{(s)}$ is elliptic; namely, when $\Jsf_b(\Fsf)$ is totally anisotropic modulo $\Zsf_\Gsf$ for every $[b] \in \B(\Fsf,\Gsf)_\bbf$. This happens if one of the following holds.
\begin{subequations}\label{eq-trace-equals-elliptic-terms}
\begin{align}
&\text{$\exists v$ s.t. $J_{b_v}$ is totally anisotropic modulo $\Zsf_{\Gsf,\Fsf_v}$}. \\
&\text{$\Gsf=\Dsf^{\times}$ for some central division algebra $\Dsf/\Fsf$ split at all $x\in\xbf$.}\label{eq-trace-equals-elliptic-terms-div-alg}
\end{align}
\end{subequations}
Recall that when $v\notin\xbf$ we set $b_v=1$, so we have $J_{b_v} = \Gsf_{\Fsf_v}$.

Let us show that the $\Fsf$-group $\Jsf_b$ is totally anisotropic modulo $\Zsf_\Gsf$ for any $[b]\in\B(\Fsf,\Gsf)_\bbf$ in each of the above settings. In the former case $\Jsf_b(\Fsf_v)$ is a closed subgroup of $J_{b_v}(\Fsf_v)$, which is compact modulo $\Zsf_\Gsf(\Fsf_v)$. In the latter case, we have $\Jsf_b = \Dsf_b^{\times}$ for some finite-dimensional simple $\Fsf$-algebra $\Dsf_b$, which can be deduced from \cite[(9.6), (A.6)]{Laumon-Rapoport-Stuhler}. Now, if $\Dsf_b$ is not a division algebra, then one can find an $\Fsf$-subalgebra $M_r(\Fsf) \subset\Dsf_b$  for some $r>1$, which implies that $M_r(\Fsf_v)\subset \Dsf\otimes_\Fsf\Fsf_v$ for any $v\notin \xbf$. Since $\Dsf$ splits at all places in $\xbf$, it follows that the local invariant  $\inv_v(\Dsf)$ is killed by $n/r$ for all places $v$ of $\Fsf$, where $n^2=\dim_\Fsf\Dsf$. This contradicts to the assumption that $\Dsf$ is an index-$n$ central division algebra over $\Fsf$.

 The following is now a simple consequence of the Lefschetz trace formula (\emph{cf.} \cite[Cor.~5.4.5]{Fujiwara:TraceFormula}, see also \cite[Thm.~5.4.5]{Varshavsky:Lefschetz-Verdier}).
 \begin{cor}\label{cor-preliminary-point-counting-elliptic}
 Assume that $\Ig^{\bbf}_{\Gsf,\xbf,\Xi}$ is quasi-compact. Then given $\varphi\in\Hcal$ we have 
 \[\tr (\Rrm\Gamma_c([\varphi^{(s)}]_{\Ksf}) =
  \sum_{([b]_\sigma,[a]))\in\FP_{\bbf}}\vol\left (\Zsf_{\Jsf_{b}(\Fsf)}(a) \backslash \Zsf_{\Jsf_{b}(\AA)}(a)/\Xi\right )\cdot 
O^{\JJ_{\bbf}}_{\iota_{\tilde z}(a)}(\varphi^{(s)})\]
for any sufficiently divisible $s$. 
Furthermore, if one of the two conditions in \eqref{eq-trace-equals-elliptic-terms} holds then the right hand side can be taken over $\FP^{\ael}_{\bbf}$ equals the elliptic term $\tr(\varphi^{(s)})^{\el}$.
 \end{cor}

\section{Kottwitz--Igusa triples} \label{sect-KS}

In this section, we rewrite the point-counting formula in Cor.~\ref{cor-preliminary-point-counting-elliptic} in terms of a certain variant of Kottwitz triples, which we call \emph{Kottwitz--Igusa triples}, under additional mild assumptions introduced in \S\ref{ssect-FP-to-KI}. 

 To associate Kottwitz--Igusa triple to them, we need to study the equivalence classes $([b]_\sigma,[a]) \in \FP_\bbf$ in more details.

\begin{lem}\label{lem-transfer}
 Let $G^\ast$ be a quasi-split reductive group over an infinite field $F$ with simply connected derived subgroup, and $H^\ast$ an $F$-subgroup of $G^\ast$ containing a maximal torus. Let $H$ be an inner form of $H^\ast$, and fix an isomorphism $H_{\scl{F}} \cong H^\ast_{\scl{F}}$. In particular, we view $H(F)$ as a subgroup of $G^\ast(\scl F)$.
 Then for any semisimple element $a \in H(F)$, the $G^\ast(\scl F)$-conjugacy class of $a$ contains an element $\gamma \in G^\ast(F)$.
\end{lem}
\begin{proof}
 Given $\tau \in \Gal(\scl{F}/F)$ we denote by $\tau$ the action on $H(\scl{F})$ and by $\tau^\ast$ the action on $H^\ast(\scl{F})$. For every $\tau  \in \Gal(\scl{F}/F)$ we choose $b_\tau \in H^\ast(\scl{F})$ such that $\tau = \Int(b_\tau) \circ \tau^\ast$. 
 
 We first assume that $a$ is regular in $G^\ast(\scl F)$. Then for any $\tau \in \Gal(\scl{F}/F)$ we have that $\tau^\ast(a) = \Int(b_\tau)(a)$; i.e., the $G^\ast(\scl F)$-conjugacy class of $a$ is $\Gal(\scl F/F)$-stable. Now the conclusion of the lemma follows from \cite[Thm.~A.1.1]{HamacherKim:Gisoc}.
 
 If $a$ is not regular, then we can choose an element $a' \in Z_H(a)(F)$ that is regular semi-simple in $G^\ast(\scl F)$. Indeed, any maximal $F$-torus $T$ of $Z_H(a)$ defines a maximal torus of $G^\ast$ over $\scl F$, and regularity in $G^\ast(\scl F)$ is a Zariski open dense condition on $T$. Thus the existence of $a'\in T(F)$ with desired properties follows from  Zariski density of $T(F)$ in $T$. Furthermore, by construction we have \[T_{\scl F}=Z_{G^\ast}(a') \subset Z_{G^\ast}(a).\]
 
 As $a'$ is regular in $G^\ast(\scl F)$, there exists $g \in G^\ast(\scl{F})$ such that for every $\tau$ we have
 \[
  ga'g\iv = \tau^\ast(ga'g\iv) = \tau^\ast(g)b_\tau a' b_\tau\iv \tau^\ast(g)\iv.
 \]
 This equation is equivalent to $g\iv\tau^\ast(h)b_\tau \in Z_{G^\ast}(a')(\scl{F})$. Since $Z_{G^\ast}(a') \subset Z_{G^\ast}(a)$, 
 we get $gag\iv = \tau^\ast(gag\iv)$ for every $\tau \in \Gal(\scl{F}/F)$; i.e., $gag\iv \in G^\ast(F)$.
\end{proof}

%
%

\begin{cor} \label{cor-transfer}
 Let $([b]_\sigma,[a]) \in \FP^{\rm ss}_\bbf$. Then there exists a representative $(b',\gamma_0)$ such that $\gamma_0 \in \Gsf^\ast(\Fsf)$. The element $\gamma_0$ is unique up to stable conjugacy.
\end{cor}
\begin{proof}
 By \cite[Cor.~6.8]{HamacherKim:Gisoc} and \cite[Lem.~5.3]{HamacherKim:Gisoc} we may choose a representative $(b,a)$ such that $\nu_\Gsf(b)$ is rational as \revise{an} element of $\Hom(\DD_\Fsf,\Gsf^\ast)$. In particular $\Jsf_b$ is an inner form of $\Msf_b^\ast$ by \cite[Cor.~6.2]{HamacherKim:Gisoc}. By the previous lemma there exists an element $\gamma_0 \in \Gsf^\ast(F)$, which is stably conjugate to $a$. Since $\coh{1}(\breve\Fsf,\Gsf^{\ast}_{\gamma_0}) = 1$, there exists $g \in \Gsf(\breve\Fsf)$ such that $gag\iv = \gamma_0$. Hence we can take $(b',\gamma_0) = g(b,a)g\iv$. By definition $a$ is unique up to $\Gsf(\breve\Fsf)$-conjugacy; i.e.\ stable conjugacy.
\end{proof}

\subsection{A parametrisation of $\FP^{\rm ss}_\bbf$}\label{ssect-FP-to-KI}
From now on, we assume that $\Gsf^\der$ is simply connected and that $\Gsf$ admits a quasi-split \emph{pure} inner form $\Gsf^\ast$.
We fix an isomorphism $\Gsf_{\breve\Fsf} \cong \Gsf_{\breve\Fsf}^\ast$ which identifies the Frobenius action $\sigma^\ast$ on $\Gsf^\ast_{\breve\Fsf}$ with $\Int(c\iv)\circ \sigma$ for some $c \in \Gsf(\breve\Fsf)$, where $\sigma$ is the Frobenius action on $\Gsf(\breve\Fsf)$.

 We fix a semisimple element $\gamma_0 \in \Gsf^\ast(\Fsf)$ and denote by $\Gsf^\ast_{\gamma_0} \subset \Gsf^\ast$ its centraliser. Our next task is to parametrise the following set:
 \[\FP_{\bbf,\gamma_0} \coloneqq \{([b]_\sigma,[a]) \in \FP_\bbf \mid a \textnormal{ stably conjugate to } \gamma_0\}. \]
Firstly, note that for $b \in \Gsf(\breve\Fsf)$ we have $\gamma_0 \in \Jsf_b(\Fsf)$ if and only if $bc \in \Gsf^\ast_{\gamma_0}(\breve\Fsf)$. (Indeed, $\gamma_0\in\Jsf_b(\Fsf)\cap \Gsf^\ast(\Fsf)$ implies $\gamma_0b=b\sigma(\gamma_0) = b c\gamma_0c\iv$.) Moreover for two such $b,b'$ we have $(b,\gamma_0) \sim (b',\gamma_0)$ if there exists a $g \in \Gsf_{\gamma_0}^\ast(\breve\Fsf)$ such that $b = gb'\sigma(g)\iv$, or equivalently $bc = gb'\sigma^\ast(g)\iv$. Hence $b' \mapsto b c\iv$ defines a bijection 
 \begin{equation} \label{eq-FP-param-1}
  \{ [b'] \in \B(\Fsf,\Gsf_{\gamma_0}^\ast) \mid [b'c\iv]_\Gsf \in \B(\Fsf,\Gsf)_\bbf \} \bij \FP_{\bbf,\gamma_0}.
 \end{equation}
 
 Next, we would like to rewrite the parametrisation \eqref{eq-FP-param-1} in terms of local invariants, which leads to the following definition.
 
\begin{defn}\label{df-Kottwitz-Igusa}
By a \emph{Kottwitz--Igusa triple} (of type $\bbf$), we mean a triple $(\gamma_0;{\boldsymbol\gamma},{\boldsymbol\delta})$ consisting of the following data:
\begin{itemize}
\item  $\gamma_0 \in \Gsf^*(\Fsf) $ is a semisimple element defined up to $\Gsf(\scl\Fsf)$-conjugate.
\item $(\ggamma,\ddelta) \in\Gsf(\AA^{\xbf}) \times \prod_{x \in \xbf} J_{b_x}(\Fsf_{x}) = \revise{\JJ_\bbf}$ is an element  ``stably conjugate'' to $\gamma_0$; i.e., $(\ggamma,\ddelta)$ and $\gamma_0$ are conjugate in $\Gsf(\scl\Fsf \otimes_{\Fsf} \AA)$.
\end{itemize}
The equivalence relation of Kottwitz--Igusa triples is generated by the following:
\begin{itemize}
	\item $(\gamma_0;\ggamma,\ddelta)\sim(\gamma_0';\ggamma,\ddelta)$ if  $\gamma_0$ and $\gamma_0'$ are stably conjugate; i.e., they are conjugate in $\Gsf(\scl\Fsf)$;
	\item $(\gamma_0;\ggamma,\ddelta)\sim(\gamma_0;\ggamma',\ddelta')$ if $(\ggamma,\ddelta)$ and $(\ggamma',\ddelta')$ are conjugate in $\JJ_\bbf$.
\end{itemize}
Let $\KS_{\bbf}$ denote the set of equivalence classes of Kottwitz--Igusa triples of type $\bbf$.  

We say that a Kottwitz--Igusa triple $(\gamma_0,\ggamma,\ddelta)$ is \emph{acceptable} if $\ddelta$  is $\bnu$-acceptable. We say that a Kottwitz--Igusa triple $(\gamma_0,\ggamma,\ddelta)$ is \emph{elliptic} if $\gamma_0\in \Gsf^\ast(\Fsf)$ is so. We let $\KS^{\ael}_{\bbf}$ denote the set of Kottwitz--Igusa triples that are acceptable and elliptic.

Let $S$ be a non-empty subset of $\xbf$  such that $b_x$ is basic for any $x\in \xbf$. We say that a Kottwitz--Igusa triple $(\gamma_0,\ggamma,\ddelta)$ is \emph{elliptic} at $S$ if the component $\delta_{x}$ of $\ddelta$ is so at any $x\in S$. We similarly define $\KS^{\ael(S)}_{\bbf}\subset \KS_{\bbf}$.
\end{defn}

\begin{rmk}
In our definition of Kottwitz--Igusa triples, we do \emph{not} require  $\gamma_0$ to have a representative in $\Gsf(\Fsf)$ but only in its quasi-split inner form $\Gsf^{\ast}(\Fsf)$. This is because we cannot strengthen Cor.~\ref{cor-transfer} to produce $\gamma_0$ in $\Gsf(\Fsf)$ without additional hypothesis. See \S\ref{ssect-EL} (especially, Cor.~\ref{cor-transfer-EL}) for more discussions.

It may be useful to consider the notion of elliptic regular elements and $S$-elliptic regular elements for $\FP_\bbf$ and $\KS_\bbf$. Elliptic regular (or $S$-elliptic regular) Kottwitz--Igusa triples do not play any role in this paper, but there could be a setting where the extra regularity may be useful; \emph{cf.} Prop.~\ref{prop-elliptic-transfer}.
\end{rmk}

As in \S\ref{notation-Hecke-action}, we view $\bbf = (b_x)_{x\in\xbf}$ as an element of $\Gsf(\AA_{\breve\Fsf})$ by setting $b_x = 1$ for $x\notin \xbf$. Note that we have $J_{b_v}(\Fsf_v) = \Gsf(\Fsf_v)$ for any $v\notin \xbf$.
\begin{lem} \label{lemma-KS-param}
 Let $(\gamma_0,\ggamma,\ddelta) \in \KS_{\bbf}$. Then there exists $\mathbb{g} \in \Gsf(\AA_{\breve\Fsf})$ such that $\mathbb{g}\gamma_0\mathbb{g}\iv = (\ggamma,\ddelta)$. The one-to-many correspondence $(\gamma_0,\ggamma,\ddelta) \mapsto \mathbb{g}\mathbb{b}\sigma(\mathbb{g})\iv c$ defines a bijection
 \[
  \{[(\ggamma,\ddelta)] \in \revise{\JJ_\bbf}/\text{conj.} \mid [(\gamma_0,\ggamma,\ddelta)] \in \KS_\bbf\} \bij \{[\mathbb{b}'] \in \B(\AA,\Gsf^\ast_{\gamma_0}) \mid [\mathbb{b}'c\iv]_{\Gsf} = [\mathbb{b}] \} .
 \]
\end{lem}
\begin{proof}
 If $\ggamma\in\Gsf(\AA^{\xbf})$ is stably conjugate to $\gamma_0$, then it follows by \cite[Prop.~7.1]{Kottwitz:StTrFormEllSing} that $\gamma_0$ and $\gamma_v$ are conjugate under a hyperspecial subgroup for all but finitely many $v$. Thus the $\Gsf(\AA^\xbf)$-conjugacy classes of such $\ggamma$ are in canonical bijection with the restricted product of $\Gsf(\Fsf_v)$-conjugacy classes within the stable conjugacy class of $\gamma_0$. Hence it suffices to prove the lemma locally, i.e.\ that we get a bijection for every $x \in |C|$
 \[
  \{\delta \in J_{b_x}(\Fsf_x)/\text{conj.} \mid \delta \text{ st.\ conj.\ to } \gamma_0\} \bij \{[b'] \in \B(\Fsf_x,\Gsf^\ast_{\gamma_0}) \mid [b'c\iv]_\Gsf = [b_x] \} .
 \]
 Now let $\delta \in J_{b_x}(\Fsf_{x})$ be stably conjugate to $\gamma_0$, i.e.\ there exists $g\in \Gsf(\scl \Fsf_{x})$ such that $g\gamma_0 g^{-1} = \delta$. Since $\coh 1(\Fsf_{x}^{\ur},\Gsf)$ is trivial, we can assume that $g\in\ \Gsf(\Fsf_{x}^{\ur})$. Now
\[
b_x  = \delta^{-1} b_x \sigma(\delta)= (g\gamma_0^{-1}g^{-1})\cdot b_x \cdot (\sigma(g)c \gamma_0 c\iv \sigma(g)^{-1}).
\]
Therefore, we have
\begin{equation*}
b_x'\coloneqq g^{-1}b_x\sigma(g)c= g\iv (b_xc) \sigma^\ast(g) \in \Gsf^\ast_{\gamma_0}(\Fsf_{x}^{\ur}). 
\end{equation*}
Note that another choice of $g$ replaces $b_x'$ with a $\Gsf^\ast_{\gamma_0}(\Fsf^{\ur}_{x})$-$\sigma^\ast$-conjugate. \reviselong{On the other hand}, given $[b_x']\in\B(\Fsf_{x},\Gsf^\ast_{\gamma_0})$ such that $[b_x'c\iv]_\Gsf = [b_x]_\Gsf$, we can reverse the construction to obtain $\delta \in J_{b_x}(\Fsf_{x})$ where a different choice of $g$ replaces $\delta$ with a $J_{b_x}(\Fsf_{x})$-conjugate.
\end{proof}

Now we define a map 
\begin{equation}\label{eq-FP-to-KS}
	\Kfr\colon\xymatrix@1{\FP^{\rm ss}_{\bbf} \ar[r]&  \KS_{\bbf}},
\end{equation}
as follows. Choose a representative $(b,a)$ of $([b]_\sigma,[a])\in\FP^{\rm ss}_{\bbf}$. By Cor.~\ref{cor-transfer} there exists $\gamma_0\in\Gsf^\ast(\Fsf)$ conjugate to $a$ in $\Gsf(\scl\Fsf)$. Now, choose a base point $\tilde z\in\pi^{-1}([b])$ as in \S\ref{ssect-global-J-orbits}, and set 
$
({\boldsymbol\gamma},{\boldsymbol\delta}) \coloneqq \iota_{\tilde z}(a).
$
 Since $\iota_{\tilde z}$ is independent of $\tilde z$ up to $\JJ_{\bbf}$-conjugate, the above formula gives a well-defined equivalence class $[(\gamma_0,{\boldsymbol\gamma},{\boldsymbol\delta})] \in \KS_{\bbf}$. It is not hard to check that the equivalence class of $(\gamma_0; {\boldsymbol\gamma},{\boldsymbol\delta})$ only depends on the equivalence class of $(b,a)$.  We now set $\Kfr(([b]_\sigma,[a]))\coloneqq  (\gamma_0; {\boldsymbol\gamma},{\boldsymbol\delta})$, which defines \eqref{eq-FP-to-KS}.

\begin{prop}\label{prop-FP-to-KS}
 Assume that $([b]_\sigma,[a])$ corresponds to $[b'] \in \B(\Fsf,\Gsf^\ast_{\gamma_0})$ under the bijection \eqref{eq-FP-param-1}. Then $\Kfr([b]_\sigma,[a])$ corresponds to the image of $[b']$ in $\B(\AA,\Gsf^\ast_{\gamma_0})$ under the bijection  in Lem.~\ref{lemma-KS-param}. In particular $\Kfr$ is surjective and each \revise{fibre} has the same number of elements as the corresponding \revise{fibre} of $\B(\Fsf,\Gsf^\ast_{\gamma_0}) \to \B(\AA,\Gsf^\ast_{\gamma_0})$.
\end{prop}
\begin{proof}
 We choose a representative $(b,a)$ such that $\gamma_0 \coloneqq a \in \Gsf^\ast(\Fsf)$. Then we get $[b'] = [bc]$.
 
 We recall the construction of $\iota_{\ztilde}$. The moduli description in Thm.~\ref{th-infinite-level-igusa-var}(\ref{th-infinite-level-igusa-var:moduli}) defines a tuple $(\underline{\Vscr'_0},\eta_0')$ over $\ztilde$, where $\underline\Vscr'_0$ is a global $\Gsf$-isoshtuka and $\eta'_0\colon (L_\AA\Gsf,\mathbb{b}\sigma) \isom \Lcal_\AA\underline\Vscr_0'$.  We fix an isomorphism $(\underline\Vscr_0') \cong (\Gsf_{\breve\Fsf},b\sigma)$ thus $\eta'_0\colon ((L_\AA\Gsf)_\k,\mathbb{b}\sigma) \isom ((\Lcal_\AA\Gsf)_\k,b\sigma)$ is given by an element $\mathbb{g}\iv \in \Gsf(\AA_{\breve\Fsf})$ such that $\mathbb{g}\iv b \sigma(\mathbb{g}) = \mathbb{b}$. Then $\iota_{\ztilde}$ is the induced morphism $\Jsf_b(\Fsf) \cong \Aut(\underline\Vscr'_0) \mono \Aut((L_\AA\Gsf)_{\cl\FF_q},\mathbb{b}\sigma) \cong \JJ_\bbf$. Thus 
 \[ 
 (\ggamma,\ddelta) = \iota_{\ztilde}(\gamma_0) = \mathbb{g} \gamma_0 \mathbb{g} \iv.
 \]
 Hence Lem.~\ref{lemma-KS-param} associates to $\Kfr(([b]_\sigma,[a]))$ the $\sigma$-conjugacy class $[\mathbb{g}\mathbb{b}\mathbb{g}\iv c] = [b']$.  
\end{proof}

Now let $\varphi\colon \JJ_\bbf/\Xi\to\CC$ be a smooth compactly supported function. Then for any Kottwitz--Igusa triple $(\gamma_0;\ggamma,\ddelta)$, let $O^{\JJ_\bbf}_{(\ggamma,\ddelta)}$ denote the orbital integral computed with respect to suitable Haar measures on $\JJ_\bbf$ and $\Zsf_{\JJ_\bbf}(\ggamma,\ddelta)$, which we will specify whenever necessary. Note that the orbital integral is essentially a finite sum since the conjugacy class of $(\ggamma,\ddelta)$ is closed in $\JJ_\bbf$ by semisimplicity.
\begin{prop}\label{prop-gl-finiteness}
Let $\varphi\colon \JJ_\bbf/\Xi\to\CC$ be a smooth compactly supported function. Then the orbital integral $O^{\JJ_\bbf}_{(\ggamma,\ddelta)}$ is zero for all but  finitely many equivalence classes of  Kottwitz--Igusa triples $(\gamma_0;\ggamma,\ddelta)$.
\end{prop}

\begin{proof}
The proof is quite similar to the proof of \cite[Prop.~8.2]{Kottwitz:StTrFormEllSing}. 

Given any compact subset $C\subset \JJ_{\bbf}/\Xi$, let $\KS_\bbf(C)$ denote the set of equivalence classes of Kottwitz--Igusa triples $(\gamma_0;\ggamma,\ddelta)$ such that the $\JJ_\bbf$-conjugacy class of $(\ggamma,\ddelta)$ has non-empty intersection with $C$. Then the orbital integral $O^{\JJ_\bbf}_{(\ggamma,\ddelta)}$ is clearly zero whenever $(\gamma_0;\ggamma,\ddelta)$ is not in $\KS_\bbf(\supp(\varphi))$. Therefore, it suffices to show that $\KS_\bbf(C)$ is finite for any compact subset $C\subset \JJ_{\bbf}/\revise{\Xi}$.

Let $\Xsf\coloneqq\Gsf\sslash\Gsf$, which is an affine variety defined over $\Fsf$ equipped with a natural conjugation-invariant map $\underline\chi\colon\Gsf\to\Xsf$. For two semisimple elements $\gamma_0,\gamma'_0\in\Gsf(\Fsf)$, we have $\underline\chi(\gamma_0)=\underline\chi(\gamma_0')$ if and only if they are conjugate in $\Gsf(\scl{\Fsf})$. Therefore, we get a well-defined map
\begin{equation}\label{eq-KS-to-conj-cl}
\KS_\bbf(C) \hookrightarrow \KS_\bbf \to \Xsf(\Fsf);\quad (\gamma_0;\ggamma,\ddelta)\mapsto \underline\chi(\gamma_0).
\end{equation}
To show that $\KS_\bbf(C)$ is finite, it suffices to show that the image and all the fibres of the above map are finite.

Let us first show that the image is finite. Note that $\JJ_\bbf$ is a closed subgroup of $\Gsf(\AA_{\breve\Fsf})$, so we view $C$ as a compact subset of $\Gsf(\AA_{\breve\Fsf})/\Xi$. Then the image $\underline\chi(C)\subset \Xsf(\AA_{\breve\Fsf})/\Xi$ is also compact. On the other hand, $\Xsf(\Fsf)$ is a discrete closed subset of $\Xsf(\AA_{\breve\Fsf})$, hence the same holds for its image in $\Xsf(\AA_{\breve\Fsf})/\Xi$. Since the image of \eqref{eq-KS-to-conj-cl} is contained in the intersection of a compact subset 
 $\underline\chi(C)$ and a discrete closed subset (namely, the image of $\Xsf(\Fsf)$),  the image of \eqref{eq-KS-to-conj-cl} is finite.

Now let us show that each fibre of \eqref{eq-KS-to-conj-cl} is finite; i.e., given a stable conjugacy class of $\gamma_0$ in $\Gsf(\Fsf)$, there are only finitely many conjugacy classes  in $\JJ_\bbf$ that are stably conjugate to  $\gamma_0$ and intersect nontrivially with $C$. 

Let us choose a finite set of places $T$ of $\Fsf$ containing $\xbf$ and a compact open subgroup $\Ksf^{T}\subset \Gsf(\AA^T)$, such that  the following conditions are satisfied.
\begin{enumerate}
\item For any $v\notin T$, the group $\Gsf$ is unramified at $v$. Furthermore, we have $\Ksf^T = \Ksf_v\Ksf^{T\cup\{v\}}$, where $\Ksf_v$ is a hyperspecial maximal compact subgroup of $\Gsf(\Fsf_v)$ and $\Ksf^{T\cup\{v\}}\subset \Gsf(\AA^{T\cup\{v\}})$.
\item For any $v\notin T$, the $v$-component $\gamma_v$ of $\ggamma$ belongs to $\Ksf_v$, and $1-\alpha(\gamma_v)$ is either $0$ or a unit in $O_{\cl{\Fsf}_v}$ for any root $\alpha$ of $\Gsf$.
\item The image of $C$ under the projection $\JJ_\bbf\twoheadrightarrow \Gsf(\AA^T)$ is contained in $\Ksf^{T}$. 
\end{enumerate}
Note that the second condition can be arranged because $\gamma_v$ and $\gamma_0$ are conjugate in $\Gsf(\cl\Fsf_v)$, and $1-\alpha(\gamma_0)\in\cl\Fsf$ is either $0$ or a unit locally at all but finitely many places.

Now, we fix a stable semisimple conjugacy class of $\gamma_0\in\Gsf(\Fsf)$. We want to show that only finitely many conjugacy class in $\JJ_\bbf$ are stably conjugate to $\gamma_0$ and intersect non-trivially with $C$. Indeed,  any such conjugacy class should contain a representative $(\ggamma,\ddelta)\in\JJ_\bbf$ such that for any $v\notin T$ its $v$-component $\gamma_v$ is equal to $\gamma_0$ by (the positive characteristic analogue of) \cite[Prop.~7.1]{Kottwitz:StTrFormEllSing}. Now, the desired finiteness follows from the finiteness of conjugacy classes in a fixed stable semisimple conjugacy class in $\Gsf(\Fsf_v)$ for any $v\in T\setminus \xbf$ (respectively, in $J_{b_x}(\Fsf_{x})$ for any $x\in\xbf$).
\end{proof}

\subsection{Background on Tamagawa numbers}
In order to formulate the trace formula for elliptic terms, we need to recall some basic facts on Tamagawa numbers over global function fields. For a connected reductive algebraic group $\Hsf$ over $\Fsf$, set $X^\ast(\Hsf) \coloneqq \Hom_{\scl\Fsf}(\Hsf,\GG_m)$, which is a finitely generated free $\ZZ$-module with discrete action of $\Gamma\coloneqq\Gal(\scl\Fsf/\Fsf)$. 

Consider the following homomorphism
\begin{equation}
\vartheta_\Hsf\colon \Hsf(\AA) \to (X^\ast(\Hsf)^{\Gamma})^\vee\coloneqq\Hom(X^\ast(\Hsf)^{\Gamma},\ZZ); \quad h\mapsto [\chi \mapsto \deg(\chi(h))].
\end{equation}
By \cite[Chap~I, Prop.~5.6b)]{Oesterle:Tamagawa}, $\Image(\vartheta_\Hsf)$ is a finite-index subgroup of $(X^\ast(\Hsf)^{\Gamma})^\vee$.

Set
\begin{equation}
\Hsf(\AA)_1 \coloneqq \ker(\vartheta_\Hsf),
\end{equation}
which is clearly an open subgroup of $\Hsf(\AA)$ containing  $\Hsf(\Fsf)$ by the product formula.

One can define a natural Haar measure $\bar\mu^{1}_{\Hsf}$ on $\Hsf(\Fsf)\backslash\Hsf(\AA)_1$ as in \cite[Chap~I,~\S5]{Oesterle:Tamagawa}. To recall, by choosing a top-degree non-zero invariant differential form on some smooth model of $\Hsf$ over $C$, we obtain a Haar measure $\mu_v$ of $\Hsf(\Fsf_v)$ for any place $v$ by the recipe of \cite[Chap~I, \S2]{Oesterle:Tamagawa}. 
Then we set
\begin{align}\label{eq-Tamagawa-measure}
\mu^1_{\Hsf} &\coloneqq a(\Hsf/\Fsf)\cdot\prod_v L_v(X^\ast(\Hsf),1)\mu_v,\ \text{where}\\
a(\Hsf/\Fsf) & \coloneqq q^{-(g_C-1)d_\Hsf}\cdot (\log q)^{r(\Hsf/\Fsf)}L^*(X^\ast(\Hsf),1)\cdot [(X^\ast(\Hsf)^{\Gamma})^\vee:\Image(\vartheta_\Hsf)]. \notag
\end{align}
Here, $g_C$ is the genus of $C$, $d_\Hsf\coloneqq\dim \Hsf$, $r(\Hsf/\Fsf)\coloneqq \rank_\ZZ (X^\ast(\Hsf)^{\Gamma}) = \ord_{s=1}L(X^\ast(\Hsf),s)$, and $L^{\ast}(X^\ast(\Hsf),1)\coloneqq\lim_{s\to1}(s-1)^{r(\Hsf/\Fsf)}L(X^\ast(\Hsf),s)$, where $L(X^\ast(\Hsf),s)$ is the Artin $L$-function associated to the $\Gamma$-module $X^\ast(\Hsf)$. Note that the resulting measure $\mu^1_{\Hsf}$ is independent of auxiliary choices involved in the construction of $\mu_v$'s.

Let $\bar\mu^{1}_\Hsf$ be the measure of $\Hsf(\Fsf)\backslash \Hsf(\AA)_1$ obtained as the quotient of $\mu^{1}_{\Hsf}$ by the counting measure on $\Hsf(\Fsf)$. 

\begin{defn}
The \emph{Tamagawa number} of $\Hsf$ is defined to be
\[\tau(\Hsf)\coloneqq \bar\mu^1_\Hsf\big (\Hsf(\Fsf)\backslash\Hsf(\AA)_1\big ),\]
which is finite when $\Hsf$ is a connected reductive group, by the classical theorem of \revise{Harder.}
\end{defn}

Note that $\tau(\Hsf)$ has a well-known cohomological formula when $\Hsf$ is tori. When $\Hsf$ is semisimple and simply connected,  Gaitsgory and Lurie \cite{GaitsgoryLurie:Tamagawa} proved that $\tau(\Hsf) = 1$. Since the behaviour of Tamagawa numbers under short exact sequences is well known  (\emph{cf.}\cite[Thm.~III.5.3]{Oesterle:Tamagawa}), we can deduce the following formula for any connected reductive group $\Hsf$ by considering a suitable $z$-extension.
\begin{equation}\label{eq-tamagawa-formula}
\tau(\Hsf') = \frac{\left |(\pi_1(\Hsf)_{\Gamma})_{\tor}\right |}{\left |\ker^1(\Fsf,\Hsf)\right |} \quad\text{for any inner form $\Hsf'$ of $\Hsf$}.
\end{equation}
See \cite[Thm.s~1.1,~1.4]{Rosengarten:Tamagawa} where this formula was obtained more generally for pseudo-reductive groups. Note that the formula in \emph{loc.~cit} is equivalent to ours by \cite[(2.4.1)]{Kottwitz:StTrFormCuspTemp}.

Now, choose a discrete \emph{torsion-free} subgroup $\Xi\subset \Zsf_\Hsf(\AA)$ such that $\Zsf_\Hsf(\Fsf)\cap\Xi=\{1\}$ and $\Zsf_\Hsf(\Fsf)\backslash \Zsf_\Hsf(\AA)/\Xi$ is compact; \emph{cf.} \S\ref{ssect-Xi}. 
By adapting the argument in the proof of Lem.~\ref{lem-Ksf-Xi} for $\vartheta_\Hsf$, one can deduce that $\Xi\cap\Hsf(\AA)_1 = \{1\}$ so $\Hsf(\Fsf)\backslash\Hsf(\AA)_1$ naturally injects into $\Hsf(\Fsf)\backslash\Hsf(\AA)/\Xi$.

For simplicity, let us introduce the following notation
\begin{equation}
i(\Hsf;\Xi)\coloneqq  [\Image(\vartheta_\Hsf):\vartheta_\Hsf(\Xi)],
\end{equation}
which is a positive integer. Indeed, both $\vartheta_\Hsf(\Xi)\subseteq \vartheta_\Hsf(\Zsf_\Hsf(\AA))$ and $\vartheta_\Hsf(\Zsf_\Hsf(\AA))\subseteq \Image(\vartheta_\Hsf)$ are of finite index by compactness of  $\Zsf_\Hsf(\Fsf)\backslash\Zsf_\Hsf(\AA)/\Xi$ the natural isogeny $\Zsf_\Hsf\to\Hsf/\Hsf^{\der}$, respectively.

By definition, we clearly have
\begin{equation}
\bar\mu^1_\Hsf(\Hsf(\Fsf)\backslash \Hsf(\AA)/\Xi) =i(\Hsf;\Xi)\cdot \tau(\Hsf).
\end{equation}
We next show that $i(\Hsf;\Xi)$, hence $\bar\mu^1_\Hsf(\Hsf(\Fsf)\backslash \Hsf(\AA)/\Xi)$, is invariant under inner twisting of $\Hsf$.

\begin{lem}\label{lem-index-inner-twisting}
For any pure inner form $\Hsf'$ of $\Hsf$, we have $i(\Hsf;\Xi) = i(\Hsf';\Xi)$ where we naturally identify $\Zsf_\Hsf$ and $\Zsf_{\Hsf'}$.
\end{lem}
\begin{proof}
Let $\Dsf\coloneqq\Hsf/\Hsf^{\der}$ denote the cocentre of $\Hsf$, which also coincides with the cocentre of $\Hsf'$. Let $\pr\colon \Hsf\to\Dsf$ and $\pr'\colon\Hsf'\to\Dsf$ respectively denote the natural map. Since $\vartheta_{\Hsf}$ and $\vartheta_{\Hsf'}$ factor through $\pr(\Hsf(\AA))$ and $\pr'(\Hsf'(\AA))$, respectively, to prove the lemma it suffices to prove that $\pr(\Hsf(\AA))=\pr'(\Hsf'(\AA))$ in $\Dsf(\AA)$; i.e., $\pr(\Hsf(\Fsf_v)) = \pr'(\Hsf'(\Fsf_v))$ in $\Dsf(\Fsf_v))$ for any place $v$ of $\Fsf$.

Set $F\coloneqq \Fsf_v$. Then by \cite[Prop.~3.12, Prop.~5.1]{Borovoi:AbCohRedGp}, $\pr\colon\Hsf(F)\to\Dsf(F)$ factors as follows:
\[\xymatrix@1{
\Hsf(F) \ar@{->>}[r]^-{\ab^0}&
\coh0_{\ab}(F,\Hsf) \ar[r]^-{\pr_{\ab}}&
\coh0_{\ab}(F,\Dsf) = \Dsf(F)
}.\]
And since the abelianised cohomology is invariant under inner twist, the map $\revise{\pr'_{\ab}}\colon\coh0_{\ab}(F,\Hsf')\to\Dsf(\AA)$ can naturally be identified with $\pr_{\ab}$; \emph{cf.} \cite[Prop.~2.8, Lem.~1.8]{Borovoi:AbCohRedGp}. Hence, we have $\pr(\Hsf(F))=\pr'(\Hsf'(F))$, from which the lemma now follows.

Note that \cite{Borovoi:AbCohRedGp} has a blanket assumption that the base field is of characteristic~$0$, but the proof of the results we are using works in any characteristic.
\end{proof}

We can now rephrase the final form of the elliptic part of the trace formula. 
\begin{thm}\label{th-point-counting-elliptic}
Assume that $\Gsf^{\der}$ is simply connected and $\Gsf$ admits a quasi-split pure inner form. Given $\varphi\in\Hcal$, the following holds for any sufficiently divisible $s$ (depending on $\varphi$)
\begin{align*}
\tr(\varphi^{(s)})^{\el}& =
 \sum_{(\gamma_0;\ggamma,\ddelta)\in\KS^{\ael}_\bbf}\left |\revise{\ker^1(\Fsf,\Gsf^\ast_{\gamma_0})}\right |\cdot i(\Gsf^\ast_{\gamma_0};\Xi)\cdot \tau(\Gsf^\ast_{\gamma_0})\cdot
O^{\JJ_{\bbf}}_{(\ggamma,\ddelta)}(\varphi^{(s)})\\
& =
 \sum_{(\gamma_0;\ggamma,\ddelta)\in\KS^{\ael}_\bbf}
i(\Gsf^\ast_{\gamma_0};\Xi)\cdot \left |(\pi_1(\Gsf^\ast_{\gamma_0})_{\Gamma})_{\tor}\right | \cdot 
O^{\JJ_{\bbf}}_{(\ggamma,\ddelta)}(\varphi^{(s)}),
\end{align*}
where we give $\bar\mu^1_{\Gsf^\ast_{\gamma_0}}$ on $\Gsf^\ast_{\gamma_0}(\Fsf)\backslash\Gsf^\ast_{\gamma_0}(\AA)$ and a suitable Haar measure on $\JJ_\bbf$ to compute the orbital integral and volume. Moreover, the summations are finite.
\end{thm}

 \begin{proof}
 We choose $s$ as in Prop.~\ref{prop-preliminary-point-counting-elliptic}. Recall that for $([b]_\sigma,[a])) \in \FP_\bbf^{\rm ss}$ and $\Kfr([b]_\sigma,[a]) \eqqcolon (\gamma_0;\ggamma,\ddelta)$ we may assume that $\gamma_0 = a$ and $\iota_{\tilde z}(a) = (\ggamma,\ddelta)\in\JJ_\bbf$.
 Since $\gamma_0$ is $\bnu$-acceptable, we have $\Gsf_{\gamma_0,\scl{\Fsf}}^\ast \subset \Msf_b$ by Cor.~\ref{cor-centralising} and \ref{cor-loc-gl-Newton-pt}. In other words, we get that $Z_{\Jsf_b,\breve\Fsf} = \Gsf^\ast_{\gamma_0,\breve\Fsf}$ and that  the  $[b'] \coloneqq [bc] \in \Gsf^\ast_{\gamma_0,\breve\Fsf}$ is basic. 
By  \eqref{eq-tamagawa-formula} and Lem.~\ref{lem-index-inner-twisting}, the summands for $([b]_\sigma,[a])$ in the formula of Prop.~\ref{prop-preliminary-point-counting-elliptic} only depend on $\Kfr([b]_\sigma,[a])$.

 By Prop.~\ref{prop-FP-to-KS} and \cite[Thm.~5.5]{HamacherKim:Gisoc} we obtain
\[
 \left| \Kfr\iv (\gamma_0,\ggamma,\ddelta)\right| = \left |\ker( (\pi_1(\Gsf^\ast_{\gamma_0}) \otimes \Div^\circ(\Fsf^s))_{\Gal(\Fsf^s/\Fsf)} \to \prod_{v \in \Fsf} (\pi_1(\Gsf^\ast_{\gamma_0}))_{\Gal(\Fsf_v^s/\Fsf_v)})\right |,
\]
which can be easily seen to coincide with $|\ker^1(\Fsf,\Gsf^\ast_{\gamma_0})|$; cf.~\cite[\S4]{Kottwitz:StTrFormCuspTemp}.

 Altogether we obtain
\[
\tr(\varphi^{(s)})^{\el}  =
 \sum_{(\gamma_0;\ggamma,\ddelta)\in\KS^{\ael}_\bbf} |\ker^1(\Fsf,\Gsf^\ast_{\gamma_0})|\cdot i(\Gsf^\ast_{\gamma_0};\Xi)\cdot\tau(\Gsf^\ast_{\gamma_0})\cdot 
O^{\JJ_{\bbf}}_{(\ggamma,\ddelta)}(\varphi^{(s)}).
\]
By Prop.~\ref{prop-gl-finiteness} the index set of the sum is finite. 
\end{proof}

\subsection{Kottwitz--Igusa triples for EL-type Igusa varieties}\label{ssect-EL}

In our definition of Kottwitz--Igusa triples, the global component $\gamma_0\in\Gsf^\ast(\Fsf)$ is not required to have a representative in $\Gsf(\Fsf)$, unlike its number field counterpart. When $\Gsf$ is an inner form of $\GL_n$ however, we now show that the global component of any Kottwitz--Igusa triple has a representative   in $\Gsf(\Fsf)$ under some mild additional assumption. We start with the following lemma.

\begin{lem}\label{lem-transfer-EL}

Let $\Bsf$ be a central simple algebra over $\Fsf$, and fix an isomorphism $\Bsf\otimes_{\Fsf}\breve\Fsf \cong M_n(\breve\Fsf)$. Let $\gamma_0^\ast\in\GL_n(\Fsf)$ be a semisimple element. 
Then $\gamma_0^\ast$ is stably conjugate to some element $\gamma_0\in\Bsf^\times$ if and only if  $\gamma_0^\ast$ is stably conjugate to some element $\ggamma\in (\Bsf\otimes_\Fsf\AA)^\times$. Furthermore, the conjugacy class of $\gamma_0\in\Bsf^\times$ is uniquely determined by $\gamma_0^\ast$; or equivalently, by $\ggamma$.
\end{lem}
\begin{proof}

The uniqueness assertion is clear as a semisimple conjugacy class is determined by its characteristic polynomial.

As $\gamma^\ast_0$ is semisimple, one can find a maximal commutative \'etale $\Fsf$-subalgebra $\Esf\subset M_n(\Fsf)$ that contains $\gamma_0^\ast$.  Furthermore, for any place $v$ of $\Fsf$ one can find an $\Fsf_v$-algebra embedding $\iota_v\colon\Esf\otimes_{\Fsf}\Fsf_v\to\Bsf_v$ with image containing the $v$-component $\gamma_v$ of $\ggamma$, obtained from an embedding of the commutant of $\gamma^\ast_0$ in $M_n(\Fsf)$ to the commutant of $\gamma_v$ in $\Bsf_v$. So to prove the lemma, it suffices to find an $\Fsf$-algebra embedding $\iota\colon\Esf \to \Bsf$ which is conjugate to $\iota_v$ over $\Fsf_v$ for any place $v$; indeed, if such $\iota$ exists then we may take $\gamma_0\coloneqq\iota(\gamma_0^\ast)$.

The proof uses the following standard fact, which we recall;  \emph{cf.} \cite[p.~420]{Kottwitz:PtShimuraVarFinFields}.
\begin{quote}
Given a central simple algebra $B$ over any field $F$ and an \'etale $F$-algebra $E$ with $(\dim_FE)^2=\dim_FB$, there is an $F$-algebra embedding $\iota\colon E\to B$  if and only if $B\otimes_FE\cong M_n(E)$, in which case $\iota$ is unique up to inner automorphism.
\end{quote}

Applying this fact to $B=\Bsf_v$ and $E = \Esf\otimes_\Fsf\Fsf_v$, we have $(\Bsf\otimes_{\Fsf}\Esf)\otimes_\Fsf\Fsf_v\cong M_n(\Esf\otimes_\Fsf\Fsf_v)$. Since $\Esf$ is a finite product of global fields, we have an isomorphism $\Bsf\otimes_\Fsf\Esf\cong M_n(\Esf)$ by class field theory. Therefore we obtain a desired embedding $\iota\colon\Esf\to \Bsf$, and the uniqueness aspect of the above fact implies that $\iota$ and $\iota_v$ are conjugate over $\Fsf_v$ for any place $v$.
\end{proof}

\begin{cor}\label{cor-transfer-EL}
Let $\Gsf/\Fsf$ be an inner form of $\GL_n$ that splits at each place $x\in\xbf$, and we fix an isomorphism $\Gsf_{\breve\Fsf}\cong \Gsf^\ast_{\breve\Fsf} = \GL_{n,\breve\Fsf}$. Then for any Kottwitz--Igusa triple $(\gamma^\ast_0;\ggamma,\ddelta)\in \KS_\bbf$ there exists a semisimple element $\gamma_0\in\Gsf(\Fsf)$ stably conjugate to $\gamma_0^\ast$.
\end{cor}
\begin{proof}
By definition of Kottwitz--Igusa triple, $\gamma^\ast_0$ is stably conjugate to $\ggamma\in\Gsf(\AA^\xbf)$. So it remains to produce $\gamma_{x}\in\Gsf(\Fsf_{x})$ stably conjugate to $\gamma^\ast_0$. Indeed, since $\Gsf$ is split at $x\in\xbf$ and $J_{b_x}$ is an inner form of a Levi subgroup of $\Gsf_{\Fsf_{x}}$, there exists $\gamma_{x}\in \Gsf(\Fsf_{x})$ stably conjugate to the image $\delta_{x}$ of $\ddelta$ in $J_{b_x}(\Fsf_{x})$. By definition of Kottwitz--Igusa triple, $\gamma_{x}$ is stably conjugate to $\gamma^\ast_0$. Now, we apply Lem.~\ref{lem-transfer-EL} to conclude the proof.
\end{proof}

\begin{rmk}
We do not expect Lem.~\ref{lem-transfer-EL} and Cor.~\ref{cor-transfer-EL} to generalise for any connected reductive groups other than inner forms of $\GL_n$ without imposing any additional condition. Indeed, Langlands and Kottwitz  constructed an obstruction for Lem.~\ref{lem-transfer-EL}, which is non-trivial in general; \emph{cf.} \cite[Thm.~6.6]{Kottwitz:StTrFormEllSing}.
\end{rmk}

We now record a version of Lem.~\ref{lem-transfer-EL} and Cor.~\ref{cor-transfer-EL} for general $\Gsf$ in the presence of an ``elliptic regular place''. The result will not be used in this paper, but it will be useful in the future study.
\begin{prop}\label{prop-elliptic-transfer}
    Suppose that  $\Gsf^{\der}$ is simply connected. 
    \begin{subenv}
        \item Let $\gamma_0^\ast\in\Gsf^\ast(\Fsf)$ be a semisimple element that is \emph{elliptic regular} at some place $v_0$. Then $\gamma_0^\ast$ is stably conjugate to an element $\gamma_0\in\Gsf(\Fsf)$ if and only if it is stably conjugate to an element $\ggamma\in\Gsf(\AA)$.
        \item Suppose furthermore that $\Gsf$ is quasi-split at each $x\in\xbf$. Then for any Kottwitz--Igusa triple $(\gamma_0^\ast;\ggamma,\ddelta)\in \KS_\bbf$ such that $\gamma_0^\ast\in\Gsf^\ast(\Fsf)$ is elliptic regular at some place $v_0$, there exists a semisimple element $\gamma_0\in \Gsf(\Fsf)$ stably conjugate to $\gamma_0^\ast$.
    \end{subenv}
\end{prop}

This proposition indicates why Def.~\ref{df-Kottwitz-Igusa} is slightly different from its number field counterpart where the global stable conjugacy class is required to be elliptic at the archimedean place (ignoring the additional regularity assumption).

\begin{proof}[Proof of {Prop.~\ref{prop-elliptic-transfer}}]
    The first claim implies the second by the same proof as Cor.~\ref{cor-transfer-EL}. To prove the non-trivial implication of the first claim, suppose that $\gamma_0^\ast\in\Gsf^\ast(\Fsf)$ is a semisimple element that is elliptic regular at $v_0$ and stably conjugate to $\ggamma\in\Gsf(\AA)$. Set $\Tsf\coloneqq \Gsf^\ast_{\gamma_0^\ast}$, which is an elliptic maximal $\Fsf$-torus by assumption. We now claim that there is an $\Fsf$-embedding $\iota\colon \Tsf\to\Gsf$ such that $\iota(\Tsf(\AA))$ is stably conjugate to the centraliser of $\ggamma$ in $\Gsf(\AA)$. Granting this claim, we can conclude the proof by setting $\gamma\coloneqq\iota(\gamma_0^\ast)$.
    
    Indeed, the existence of $\iota$ as above can be extracted from the proof of Lem.~14.1 in \cite{Kottwitz:PtShimuraVarFinFields}. To explain, Langlands \cite[Ch.~VII]{Langlands:DebutStable} showed that the obstruction for the existence of $\iota$ lies in the Pontryagin dual of
    \[
    \ker\left(\pi_0((\widehat\Tsf^{\der})^\Gamma)\to \prod_v\pi_0((\widehat\Tsf^{\der})^{\Gamma(v)})\right ),
    \]
    where $\widehat\Tsf^{\der}$ is the dual torus of $\Tsf^{\der}\coloneqq \Tsf\cap\Gsf^{\der}$ and $\Gamma(v)$ is the decomposition group at $v$. (See \cite[\S9]{Kottwitz:StTrFormCuspTemp} for an alternative proof.) And as $\Tsf^{\der}$ is totally anisotropic at $v_0$ it follows that $\pi_0((\widehat\Tsf^{\der})^\Gamma)\to \pi_0((\widehat\Tsf^{\der})^{\Gamma(v_0)})$ is injective, which shows the vanishing of the obstruction.
\end{proof}

\section{Case of division algebras}\label{sect-div-alg}

Let $\Dsf$ be a central division algebra over $\Fsf$ with index $n$, and let $\Gsf\coloneqq \Dsf^{\times}$. If $\Dsf$ is split at each $x\in\xbf$, then $\Ig^{\bbf}_{\Gsf,\xbf,\Xi}$ is quasi-compact by \cite[Prop.~1.11]{Lau:Degeneration}.

The main goal of this section is to compute the ``alternating sum'' of the compactly supported cohomology of the Igusa variety $\Ig^{\bbf}_{\Gsf,\xbf,\Xi}$ in terms of automorphic representations of $\Gsf(\AA)$ under the additional assumption that $n<p$. See Thm.~\ref{th-second-basic-id} and its corollaries for further details. 

The statement and its proof are strongly analogous to the \emph{Second Basic Identity} for certain unitary Shimura varieties and associated Igusa varieties (\emph{cf.} \cite[Thm.~V.5.4]{HarrisTaylor:TheBook}, \cite[Thm.~6.7]{ShinSW:CohRZ}). As in the classical case, the main result essentially follows from the \emph{trace formula} for the Igusa variety (\emph{cf.} Thm.~\ref{th-point-counting-elliptic}) via standard \emph{local harmonic analysis}. 
Note that various complications of local harmonic analysis in characteristic~$p$ can be resolved at least for $\GL_n$ with $n<p$, thanks to  \cite{DeligneKazhdanVigneras:CenSimAlg} and later developments such as \cite{Lemaire:LocIntegrabilityGLN, Badulescu:Transfer} to list a few. The assumption $n<p$ is needed only for the characterisation of orbital integrals (\emph{cf.} Thm.~\ref{th-characterisation-orb-int}), but it seems within reach to formulate and prove the characterisation for general $n$. This will be considered in a future project.

\subsection{Background on character distributions}\label{ssect-LJ1}

We fix a locally profinite group $H$ and a Haar measure $\mu_H$ on it. By \emph{distribution} on $H$, we mean a linear functional
$C^{\infty}_c(H;\CC)\to\CC.$

If $(\pi,V)$ is an \emph{admissible} representation of $H$, then any $f\in C^{\infty}_c(H;\CC)$ gives rise to a \emph{finite-rank} endomorphism $\pi(f)$ on $V$ via convolution. In particular, we have a well-defined \emph{character distribution} of $\pi$:
\begin{equation}
\tr(-\mid \pi)\colon C^{\infty}_c(H;\CC)\to\CC;\quad \tr(f\mid \pi)\coloneqq \tr\pi(f), \  \forall f\in C^\infty_c(H).
\end{equation}

Now assume that $H$ is the group of $F$-rational points for a reductive group over a non-archimedean local field $F$. Let $H^{\rss}$ denote the subset of regular semisimple elements, which is open dense in $H$ with measure-$0$ complement. By \cite[Prop.~13.1]{AdlerKorman:CharExpansion} there is a locally constant function $\Theta_\pi$ on $H^{\rss}$ such that we have 
\begin{equation}\label{eq-char-dist}
\tr(f\mid \pi) = \int_{H^{\rss}} (f\cdot\Theta_\pi)\ d\mu_H
\end{equation}
for any smooth compactly supported function $f$ on $H^{\rss}$. (The characteristic~$0$ case is classical, while \emph{loc.~cit.} gave a characteristic-free proof. The key step is to prove the Harish-Chandra submersion theorem for any characteristic, which is done in Appendix~B in \cite{AdlerDeBacker:MurnaghanKirillov}, written by G.~Prasad.) Since $\tr(-\mid \pi)$ is invariant under conjugation by $H$, the same holds for the function $\Theta_\pi$. 

Furthermore, the formula~\eqref{eq-char-dist} is known to hold for any $f\in C^{\infty}_c(H;\CC)$, not necessarily supported in $H^{\rss}$, if one of the following holds:

\begin{enumerate}
\item if $F$ is of characteristic~$0$; \emph{cf.}~\cite{HarishChandra:AMS}.
\item if $G$ is an inner form of $\GL_n$ and $F$ is of characteristic $p$; \emph{cf.}~\cite{Badulescu:Transfer} and \cite{Lemaire:LocIntegrabilityTwistedCharacters}, built upon the case of $\GL_n(F)$ in \cite{Lemaire:LocIntegrabilityGLN}.
\end{enumerate}
Therefore, in the above cases the trace distribution $\tr(-\mid \pi)$ can be \emph{represented} by an (arbitrary) extension of $\Theta_\pi$ to $H$, as $H\setminus H^{\rss}$ is of measure~$0$.

Given a locally profinite group $H$ let $\Groth(H)$ denote the Grothendieck group of finite-length admissible representations of $H$. We linearly extend the definition of character distribution $\tr(-\mid \pi)$ and the function $\Theta_\pi$ on $H^{\rss}$ for $\pi\in\Groth(H)$.

\subsection{Background on the Jacquet-Langlands correspondence}
Let $F$ be a local field of characteristic~$p$, and set $H^\ast  = \prod_{j=1}^m\GL_{n_j}$. We choose an inner form $H=\prod_{j=1}^m\GL_{r_j}(D_j)$ where $D_j$ is a central division algebra  over $F$ for any $j$. We choose $H^{\ast}(\scl F)\cong H(\scl F)$ realising $H$ as an inner twist of $H^\ast$, and view $H(F)$ as a subgroup of $H^\ast(\scl F)$. If there is no risk of confusion, we let $H$ and $H^\ast$ also denote $H(F)$ and $H^\ast(F)$, respectively.

We say that semisimple elements $g \in H$ and $g^\ast  \in H^\ast$ have \emph{matching conjugacy classes} and write $g\leftrightarrow g^\ast $, if they are conjugate in $H^\ast(\scl F)$. There is a more direct description in terms of matching of characteristic polynomial of each factor.

Recall that Badulescu \cite[Prop.~3.3]{Badulescu:JLunitarisabilite} defined the following surjective group homomorphism
\begin{equation}\label{eq-LJ}
\LJ_H\colon \Groth(H^\ast ) \to \Groth(H),
\end{equation}
where given $\pi^\ast \in\Groth(H^\ast)$, $\LJ_H(\pi^{\ast})$ is characterised by the following identity
\begin{equation}\label{eq-LJ-char-id}
\Theta_{\pi^\ast }(g^\ast ) = e(H)\cdot\Theta_{\LJ_H(\pi^\ast )}(g),
\end{equation}
for any $g^\ast \in H^{\ast ,\rss}$ and $g\in H^{\rss}$ with matching conjugacy classes. Here, $e(H)= \prod_{j=1}^{m}(-1)^{n_j-r_j}$ is the \emph{Kottwitz sign}. 

The map $\LJ_H$ could be understood as the \emph{inverse} of the usual Jacquet--Langlands correspondence as follows. There is a natural bijection $\pi\rightsquigarrow \JL_H(\pi)$ from the set of isomorphism classes of essentially square-integrable representations of $H$ to those of $H^\ast $ constructed by Badulescu, and we have $\LJ_H(\JL_H(\pi)) = \pi$ for any irreducible essentially square-integrable representation $\pi$ of $H$.

Recall that $M_{b_x}$ denotes the centraliser of the Newton cocharacter of $b_x$. 
Since $\Gsf_{\Fsf_{x}}$ is isomorphic to $\GL_n$ for $x\in\xbf$, the Levi subgroup $M_{b_x}$ is a product of suitable $\GL_{n_j}$'s. Let
\begin{equation}\label{eq-LJbi}
\LJ^{b_x}\coloneqq \LJ_{J_{b_x}}\colon \Groth(M_{b_x}(\Fsf_{x})) \to \Groth(J_{b_x}(\Fsf_{x})),
\end{equation}
denote the homomorphism \eqref{eq-LJ} for $H^\ast =M_{\nu_x}$ and $H = J_{b_x}$.

\subsection{The main result for division algebras}
Let $P_{b_x}$ denote the parabolic subgroup of $\Gsf_{\Fsf_{x}}$ where  the weights of the Newton cocharacter $\nu_x=\nu(b_x)$ on the Lie algebra of $P_{b_x}$ is non-positive. Note that our choice of $P_{b_x}$ is opposite \revise{to} Shin's; \emph{cf.} \cite[\S3]{ShinSW:StableTraceIgusa}, \cite[\S6]{ShinSW:CohRZ}. Then we finally define
\begin{equation}
\Red^{b_x} \coloneqq e(J_{b_x})\cdot \LJ^{b_x}\circ \Jac^{\Gsf(\Fsf_{x})}_{P_{b_x}(\Fsf_{x})}\colon \Groth(\Gsf(\Fsf_{x}))\to\Groth(J_{b_x}(\Fsf_{x})),
\end{equation}
where $e(J_{b_x})$ is the \emph{Kottwitz sign} of $J_{b_x}$ and $\Jac^{\Gsf(\Fsf_{x})}_{P_{b_x}(\Fsf_{x})}$ is the \emph{unnormalised} Jacquet functor.

Finally, we set 
\begin{equation}\label{eq-Red}
\Red^{\bbf}\colon \Groth(\Gsf(\AA))\to\Groth(\JJ_\bbf)
\end{equation}
by sending an irreducible representation $\pi =  \pi^{\xbf}\otimes \big ( \bigotimes_{x\in\xbf}\pi_x\big)$ to $\pi^{\xbf}\otimes \big ( \bigotimes_{x\in\xbf}\Red^{b_x}(\pi_x)\big)$.

We choose a non-zero degree id\`ele $\xi\in\AA^\times = \Zsf_\Gsf(\AA)$ so that the subgroup $\Xi\coloneqq\langle\xi\rangle\subset\Zsf_\Gsf(\AA)$ satisfies all the requirements in \S\ref{ssect-Xi}. Let 
\[
\Ascr(\Gsf(\Fsf)\backslash\Gsf(\AA)/\Xi)\coloneqq \Ccal^{\infty}(\Gsf(\Fsf)\backslash\Gsf(\AA)/\Xi;\CC);
\]
denote the space of locally constant functions on $\Gsf(\Fsf)\backslash\Gsf(\AA)/\Xi$ equipped with the right regular $\Gsf(\AA)$-action; i.e., the space of automorphic forms with central character trivial on $\Xi$. By compactness of $\Gsf(\Fsf)\backslash\Gsf(\AA)/\Xi$, any irreducible constituents are subrepresentations and square-integrable (i.e., discrete series). Furthermore, by strong multiplicity one theorem \cite[Thm.~3.3]{BadulescuRoche:JLC}, given a smooth irreducible  representation $\pi^{\xbf}$ of $\Gsf(\AA^\xbf)$ there exists at most one irreducible subrepresentation $\pi\subset \Ascr(\Gsf(\Fsf)\backslash\Gsf(\AA)/\Xi)$ such that $\pi\cong \pi^{\xbf}\otimes\left (\bigotimes_{x\in\xbf}\pi_{x}\right )$ for some smooth irreducible  representation $\pi_{x}$ of $\Gsf(\Fsf_{x})$.

Fixing an isomorphism $\CC\cong \cl{\QQ}_l$,  we may view $\coh{i}_c(\Ig^{\bbf}_{\Gsf,\xbf,\Xi},\cl\QQ_\ell)$ as an admissible $\JJ_\bbf$-representation by quasi-compactness of Igusa varieties. So $\Hcal\coloneqq C^{\infty}_c(\JJ_\bbf/\Xi;\CC)$ acts as finite-rank endomorphisms via convolution. 
The following is the main result of this section.
\begin{thm}\label{th-second-basic-id}
Assume that $\Gsf=\Dsf^{\times}$ where $\Dsf$ is a central division algebra over $\Fsf$ of dimension $n^2$ with $n<p$ that splits at each $x\in\xbf$. Then for any $\varphi\in\Hcal$, we have 
\[
\sum_i(-1)^i\tr\big(\varphi\mid \coh i_c(\Ig^{\bbf}_{\Gsf,\xbf,\Xi},\cl\QQ_\ell)\big) = \big(\prod_{x\in\xbf}e(J_{b_x})\big)\cdot\sum_{\pi\subset \Ascr(\Gsf(\Fsf)\backslash\Gsf(\AA)/\Xi)} \tr \big(\varphi\ | \Red^\bbf(\pi)\big),
\]
where the right hand side is the sum over all automorphic representations $\pi$ with central character trivial on $\Xi=\langle\xi\rangle$ chosen as above, with only finitely many nonzero summand. Here, $e(J_{b_x})\in\{\pm1\}$ is the Kottwitz sign of $J_{b_x}$.
\end{thm}

\begin{rmk}
We assumed that $n<p$ as Thm.~\ref{th-characterisation-orb-int} was only proved under this assumption. In particular, this assumption can be removed once we obtain a suitable version of Thm.~\ref{th-characterisation-orb-int} that works when $n\geqslant p$. See Rmk.~\ref{rmk-Dixmier} for more discussions.
\end{rmk}

Given a smooth irreducible representation $\pi^{\xbf}$ of $\Gsf(\AA^\xbf)$, we define
\begin{multline}
[\Rrm\Gamma_c(\Ig^{\bbf}_{\Gsf,\xbf,\Xi})](\pi^{\xbf})\coloneqq \sum_i(-1)^{i}\Hom_{\Gsf(\AA^{\xbf})}\big(\pi^{\xbf,\vee},\coh{i}_c(\Ig^{\bbf}_{\Gsf,\xbf,\Xi},\cl\QQ_\ell)\big)\\
\in\Groth\big(\prod_{x\in\xbf}J_{b_x}(\Fsf_{x})\big).
\end{multline}
\begin{cor}\label{cor-second-basic-id}
In the same setting as Thm.~\ref{th-second-basic-id}, let $\pi^{\xbf}$ be a smooth irreducible representation of $\Gsf(\AA^\xbf)$. Then $\pi^\xbf$ appears as the local factor of an automorphic representation $\pi=\pi^{\xbf}\otimes\left (\bigotimes_{x\in\xbf}\pi_{x}\right )$ with central character trivial on $\Xi$ if and only if $[\Rrm\Gamma_c(\Ig^{\bbf}_{\Gsf,\xbf,\Xi})](\pi^{\xbf})\ne 0$, in which case we have 
\[ [\Rrm\Gamma_c(\Ig^{\bbf}_{\Gsf,\xbf,\Xi})](\pi^{\xbf}) = \big(\prod_{x\in\xbf}e(J_{b_x})\big)\cdot\bigotimes_{x\in\xbf} \Red^{b_x}(\pi_{x}).\]
\end{cor}

\begin{rmk}
Note that  the $\pi^{\xbf}$-isotypic part of the cohomology $[\Rrm\Gamma_c(\Ig^{\bbf}_{\Gsf,\xbf,\Xi})](\pi^{\xbf})$ only depends on $J_{b_x}$'s for $x\in\xbf$, but not on the $\sigma$-conjugacy classes of $b_x$. Indeed, the same remark applies to the left hand side of Thm.~\ref{th-second-basic-id}.
\end{rmk}

\subsection{Cohomology of moduli stacks of global $\Gscr$-shtukas} 

In the spirit of the \emph{Second Basic Identity} for certain unitary Shimura varieties (\emph{cf.} \cite[Thm.~V.5.4]{HarrisTaylor:TheBook}, \cite[Thm.~6.7]{ShinSW:CohRZ}) one would expect to relate the cohomology of Igusa varieties with the intersection cohomology of moduli stacks of $\Gscr$-shtukas for a suitable integral model $\Gscr$ of $\Gsf$. Indeed, Thm.~\ref{th-second-basic-id} and Cor.~\ref{cor-second-basic-id} can be interpreted this way once we have the decomposition of the intersection cohomology of moduli stacks of $\Gscr$-shtukas for some smooth integral model $\Gscr$ of $\Gsf$.

To explain, let $\Gscr$ be a smooth integral model of $\Gsf$ given by the unit of a maximal order of $\Dsf$. Let $U\subset C$ be an open dense subscheme over which $\Gscr$ is reductive, and we will consider $\Gscr$-shtukas with legs in $U$, indexed by a finite set $I$. Finally, choose a tuple $\boldsymbol\lambda\coloneqq (\lambda_i)_{i\in I}$  of dominant coweights of $\Gsf_{\scl\Fsf} \cong \GL_{n,\scl\Fsf}$ such that $\sum_{i\in I}\deg(\lambda_i)=0$, which naturally defines a $\restr{L^+_{C^I}\Gscr}{U^I}$-stable closed subscheme $Z_{\blambda} \subset \restr{\Gr_{\Gscr,\Ibf}}{U^I}$; \emph{cf.} \cite[Prop.~1.1]{Lau:Degeneration}. Let $\Xscr^{Z_{\blambda}}\coloneqq\Xscr^{Z_{\blambda}}_{\Gscr,\Ibf}|_{U^I}$ be the moduli stack of $\Gscr$-shtukas bounded by $Z_{\blambda}$, where $\Ibf$ is a partition of $I$ into singletons. Given any finite closed subscheme $N\subset C$, we can also define the level-$N$ moduli stack $\Xscr^{Z_{\blambda}}_N$ over $(U\setminus N)^I$.

Choose $\Xi\coloneqq\langle\xi\rangle$ as in Thm.~\ref{th-second-basic-id}, and \emph{assume} that $\Xscr^{Z_{\blambda}}/\Xi$ is proper over $U^I$. This can be guaranteed if $\Dsf$ is ``sufficiently ramified''; see \cite[Thm.~A]{Lau:Degeneration} for the precise condition.  For any finite closed subscheme $N\subset C$, let $H^i_{N,\blambda}$ denote the ``$i$th intersection cohomology'' of the fibre of $\Xscr^{Z_{\blambda}}_N/\Xi$ over a geometric generic point of $(U\setminus N)^I$, where the intersection complex is normalised so that it restricts to $\cl\QQ_l[d](\frac{d}{2})$ on the smooth locus where $d$ is the relative dimension of $\Xscr^{Z_{\blambda}}$ over $U^I$. Our $H^i_{N,\blambda}$ is the geometric generic fibre of the sheaf on $U^I$ defined in \cite[Def.~8.2.2]{Lau:Thesis}, and it is a finite-dimensional $\cl\QQ_\ell$-vector space equipped with commuting actions of the ``level-$N$ Hecke algebra'' and $\Gal(\scl\Fsf/\Fsf)^I$ (\emph{cf.} \cite[Cor.~8.2.5]{Lau:Thesis}). From now on, we will ignore the Galois action for simplicity. By taking the colimit with respect to $N$ we obtain an admissible $\Gsf(\AA)$-representation
\[H^i_{\blambda}\coloneqq \varinjlim_{N\subset C}H^i_{N,\boldsymbol\lambda}.\]
Given a smooth irreducible representation $\pi^{\xbf}$ of $\Gsf(\AA^\xbf)$, set \[H^i_{\boldsymbol\lambda}(\pi^{\xbf})\coloneqq\Hom_{\Gsf(\AA^\xbf)}(\pi^{\xbf,\vee}, H^i_{\boldsymbol\lambda}),\] which is an admissible representation of $\prod_{x\in\xbf}\Gsf(\Fsf_{x})$ with finite length by \cite[Cor.~8.2.6]{Lau:Thesis} and the strong multiplicity one theorem. Therefore, we can define the alternating sum in the Grothendieck group; namely,
\begin{equation}
H_{\boldsymbol\lambda}(\pi^{\xbf})\coloneqq \sum_{i} (-1)^iH^i_{\boldsymbol\lambda}(\pi^{\xbf}).
\end{equation}

The following corollary is immediate from Cor.~\ref{cor-second-basic-id} and the decomposition of the intersection cohomology \cite[Thm.~9.3.3]{Lau:Thesis}.
\begin{cor}\label{cor-second-basic-id-Kottwitz}
Assume that $\Xscr^{Z(\boldsymbol\lambda)}/\Xi$ is proper over $U^I$ with positive relative dimension, and the ``base change fundamental lemma'' holds for $\GL_n$.
Then we have 
\[
\Red^{\bbf}(H_{\boldsymbol\lambda}(\pi^{\xbf})) = \big(\prod_{x\in\xbf}e(J_{b_x})\big)\cdot d_{\boldsymbol\lambda}\cdot [\Rrm\Gamma_c(\Ig^{\bbf}_{\Gsf,\xbf,\Xi})](\pi^{\xbf})
\]
in $\Groth\big(\prod_{x\in\xbf}J_{b_x}(\Fsf_{x})\big)$. Here, $d_{\boldsymbol\lambda} \coloneqq\prod_{i\in I}d_{\lambda_i}$ where $d_{\lambda_i}$ is the dimension of the irreducible $\widehat\Gsf$-representation with highest weight $\lambda_i$.
\end{cor}
Note that \cite[Thm.~9.3.3]{Lau:Thesis} also requires the properness of $\Xscr^{Z_{(\mu^+,\mu^-)}}$ over $U^2$ where $\mu^+\coloneqq (1,0,\cdots,0)$ and $\mu^-\coloneqq(0,\cdots,0,-1)$, but this is implied by the assumptions of Cor.~\ref{cor-second-basic-id-Kottwitz}; \emph{cf.} \cite[Thm.~A]{Lau:Degeneration}.

By \emph{base change fundamental lemma}, we mean the characteristic~$p$ analogue of the statement \cite[4.5, 3.13]{ArthurClozel:BC} for finite unramified extensions $E/F$ of local fields of characteristic $p$, noting that $\Gsf_{\Fsf_{x}}\cong \GL_n$ for $x\in|U|$. (See also \cite[\S5.3]{Lau:Thesis}.) Some special cases have been obtained by Drinfeld (\emph{cf.} \cite[Thm.~4.5.5]{laumonbook1}) and Ng\^o (\emph{cf.} \cite[\S5.7]{NgoBC:DSht}), and the proof for the characteristic~$0$ case in \cite{ArthurClozel:BC} is believed to work in characteristic~$p$; \emph{cf.} footnote on page~242 in \cite{NgoBC:DSht}.

\begin{rmk}
    The proof of Cor.~\ref{cor-second-basic-id-Kottwitz} is via directly comparing Cor.~\ref{cor-second-basic-id} with \cite[Thm.~9.3.3]{Lau:Thesis}, and it does \emph{not} make use of any geometric relation between $\Ig^\bbf_{\Gsf,\xbf}$ and $\Xscr^{Z_{\blambda}}_N$, such as the foliation theory. When $\Xscr^{Z_{\blambda}}/\Xi$ is proper over $U^I$, there should be an alternative proof of Cor.~\ref{cor-second-basic-id-Kottwitz} (under different assumptions) using the foliation structure of Newton strata for the fibre of $\Xscr^{Z_{\blambda}}/\Xi$ at some fixed legs $\overline\xbf\in C(\cl\FF_q)^I$.
\end{rmk}

\begin{rmk}
Without any properness assumption on $\Xscr^{Z_{\blambda}}/\Xi\to U^I$, an analogue of \cite[Thm.~9.3.3]{Lau:Thesis} is obtained by V.~Lafforgue and X.~Zhu \cite[Cor.~3.1]{LafforgueZhu:elliptic}, which is \emph{unconditional} as $\Xscr^{Z(\boldsymbol\lambda)}/\Xi$ is of finite type over $U^I$. Note that Lafforgue--Zhu decomposed the \emph{direct sum} of $H^i_{N,\boldsymbol\lambda}$, while Lau decomposed the \emph{alternating sum}. The two results are compatible when $\Xscr^{Z(\boldsymbol\lambda)}/\Xi$ is proper over $U^I$ in which case the odd degree intersection cohomology vanishes; \emph{cf.} \cite[Cor.~9.1.4]{Lau:Thesis}. 
\end{rmk}

\subsection{Some local harmonic analysis in positive characteristic}
We fix a place $x$ of $\Fsf$ and write $F\coloneqq \Fsf_x$. We will not strictly distinguish an algebraic group over $F$ and its group of $F$-rational points. Let $G\coloneqq \Gsf(\Fsf_x) = \Dsf_x^{\times}$.

In this subsection, we recall some basic local harmonic analysis results for $G$ under the additional assumption that  $n<p$ where $n$ is the index of $D$. It is observed in \cite[Appendix~1]{DeligneKazhdanVigneras:CenSimAlg} that for inner forms of $\GL_n$ many local harmonic analysis \revise{results hold} even in characteristic $p$ (not necessarily assuming $n<p$) by replacing the exponential map on the Lie algebra with the map $X\mapsto 1+X$. Now, under the extra assumption $n<p$ we have the following simplifying features:
\begin{enumerate}
\item Any element $g\in G$ admits a canonical Jordan decomposition; indeed, an element $g$ of an inner form of $\GL_n(F)$ admits a canonical Jordan decomposition if all the irreducible factors of the characteristic polynomial of $g$ are separable, which is automatic if $n<p$.
\item There are only finitely many $F$-tori in $G$ up to \revise{$G$-conjugacy}; indeed, this amounts to the finiteness of the number of separable extensions of $F$ with degree at most $p-1$. Note that there are infinitely many separable extension of $F$ with degree $p$ coming from the Artin--Schreier theory.
\end{enumerate}

It turns out that the proof of the characterisation of orbital integrals due to Vign\'eras  \cite{Vigneras:OrbInt} remains valid for inner forms of $\GL_n/F$ with $n<p$, as the main ingredients of the proof hold true in this setting. 

To state the result, let us recall the following definitions from \cite[\S1]{Vigneras:OrbInt}. 
\begin{defn}
Given a semisimple element $s\in G$, the \emph{standard torus} associated to $s$ is defined to be the (connected) centre of the centraliser $Z_G(s)$. A \emph{standard torus} $T$ of $G$ means the standard torus associated to some semisimple element $s\in G$. A pair $(T,u)$ of a standard torus $T$ and a unipotent element $u\in Z_G(T)$  is called a \emph{standard pair}. We obviously define the notion of conjugacy class of standard pairs.
\end{defn}
 
We note that  if $n<p$ then there are only finitely many standard pairs up to conjugation and any element of $G$ is contained in $T\cdot u$ for some standard pair $(T,u)$ by Jordan decomposition.

\begin{defn}
Given a subset $X\subset G$, let $X^{\reg}$ denote the subset of $X$ consisting of elements whose $G$-conjugacy class has dimension larger than or equal to the dimension of $G$-conjugacy class of any other element of $X$.
\end{defn}
Note that elements of $X^{\reg}$ may be neither regular nor semisimple in $G$. For a standard pair $(T,u)$, we have $(T\cdot u)^{\reg} = T^{\reg}\cdot u$.

\begin{thm}[\emph{cf.} {\cite[1.n]{Vigneras:OrbInt}}]\label{th-characterisation-orb-int} 
Let $G$ be an inner form of $\GL_n(F)$ such that $n < p\coloneqq \cha F$. Then a conjugate-invariant function $W\colon G\to\CC$ is an orbital integral (i.e., there exists $f\in C_c^{\infty}(G;\CC)$ such that $W(g)=O_g(f)$) if and only if for any standard pair $(T,u)$ the following properties hold
\begin{enumerate}
\item $W|_{T^{\reg}\cdot u}$ is locally constant;
\item $W|_{T\cdot u}$ is compactly supported;
\item For any $s\in T$ there exists a neighbourhood $U\subset T$ of $s$ (depending on $W$) such that for any $t\in U\cap T^{\reg}$ we have the following ``germ expansion''
\[W(tu) = \sum W(su_i) a^{Tu}_{su_i}(tu),\]
where the sum is over the conjugacy classes whose semisimple part is conjugate to $s$, and $ a^{Tu}_{su_i}$ is a function on $U\cdot u$ whose germ at $su$ is independent of $W$.
\end{enumerate}
\end{thm}
\begin{proof}
 The proof is identical to the characteristic 0 case proven in \cite{Vigneras:OrbInt} if one replaces the exponential map with $X\mapsto 1+X$.
\end{proof}

\begin{rmk}\label{rmk-Dixmier}
If $n\geqslant p$, the above criterion fails to give a sufficient condition for $W$ to be an orbital integral as it says nothing about the behaviour of $W$ at elements that do not admit Jordan decomposition. 

In \cite[Appendix~1, 2.e]{DeligneKazhdanVigneras:CenSimAlg}, it is asserted that one can obtain a characterisation of orbital integrals for inner forms $G$ of $\GL_n(F)$ in any characteristic  by working with \emph{Dixmier sheets} in place of $T^{\reg}\cdot u$. Dixmier sheets are strata of $\Lie G$ defined in terms of elementary divisors, which is stable under the adjoint action of $G$; see \cite[Appendix~1, 2.c]{DeligneKazhdanVigneras:CenSimAlg} for the definition of Dixmier sheets. 

The characterisation of orbital integrals in \cite[Appendix~1, 2.e]{DeligneKazhdanVigneras:CenSimAlg} asserts that any orbital integral $W$ is locally constant on each Dixmier sheet. Now note that there is a unique non-empty \emph{open} Dixmier sheet containing all the regular elements in $\Lie G$ in the sense that the dimension of the centraliser is $n$. The open Dixmier sheet for $G=\GL_n(F)$ contains all regular semisimple elements as well as regular unipotent elements, and one can find a sequence of elliptic regular elements converging to a regular unipotent element, which can be seen by looking at the characteristic polynomials. On the other hand, one can find an orbital integral that vanishes at the regular unipotent conjugacy class while being non-zero at any elliptic regular semisimple elements in the small enough neighbourhood of a regular unipotent element. Such an example can be constructed from the restriction of the very cuspidal Euler--Poincar\'e function to the locus where the determinant is a unit, using \cite[Thm.~5.1.3]{laumonbook1}.
\end{rmk}

As a consequence of the characterisation of orbital integrals, we may obtain the following Proposition. Let $\Hcal(G)$ denote the Hecke algebra of $G$; i.e., $\Hcal(G) = C^{\infty}_c(G;\CC)$ equipped with the convolution product. Let $J_G\subset \Hcal(G)$ denote the ideal generated by functions of the form $f - f^{g}$ where $f\in\Hcal(G)$ and $f^g(h) = f(ghg^{-1})$ for any $g,h\in G$.

\begin{prop}\label{prop-DKV-a4h}
Assume that $G$ is the group of $F$-points of an inner form of $\GL_n$ with $n<p$.
For $f\in\Hcal(G)$, the following are equivalent
\begin{enumerate}
\item\label{prop-DKV-a4h-J} We have $f\in J_G$.
\item\label{prop-DKV-a4h-trace} For any irreducible representation $\pi$ of $G$, we have $\tr(f\mid \pi) = 0$.
\item\label{prop-DKV-a4h-trace-tempered} For any irreducible \emph{tempered} representation $\pi$ of $G$, we have $\tr(f\mid\pi) = 0$.
\item\label{prop-DKV-a4h-orb-int} For any $g\in G$, we have $O^G_g(f) = 0$.
\item\label{prop-DKV-a4h-orb-int-rss} For any  $g\in G^{\rss}$, we have $O^G_g(f) = 0$.
\end{enumerate}
\end{prop}
\begin{proof}
The proof in \cite[A.4.h]{DeligneKazhdanVigneras:CenSimAlg} works verbatim thanks to Thm.~\ref{th-characterisation-orb-int}. To explain,  we clearly have \eqref{prop-DKV-a4h-J}$\Rightarrow$\eqref{prop-DKV-a4h-trace}$\Rightarrow$\eqref{prop-DKV-a4h-trace-tempered} and \eqref{prop-DKV-a4h-J}$\Rightarrow$\eqref{prop-DKV-a4h-orb-int}$\Rightarrow$\eqref{prop-DKV-a4h-orb-int-rss}. The implication \eqref{prop-DKV-a4h-trace-tempered}$\Rightarrow$\eqref{prop-DKV-a4h-orb-int-rss} is proved in \cite[A.2]{DeligneKazhdanVigneras:CenSimAlg}. Finally, if $n<p$ then one can prove that \eqref{prop-DKV-a4h-orb-int-rss} implies \eqref{prop-DKV-a4h-J} by the same proof as \cite[Appendix~1, 2.f]{DeligneKazhdanVigneras:CenSimAlg} using Thm.~\ref{th-characterisation-orb-int} in place of \cite[Appendix~1, 2.e]{DeligneKazhdanVigneras:CenSimAlg}.
\end{proof}

\subsection{Transfer of functions}
Let $G\coloneqq \GL_n(F)$ with $F\coloneqq \Fsf_x$. We fix a decent element $b\in \Gsf(\breve F)$ and $M_b$ denote the centraliser of the Newton cocharacter of $b$, which is a Levi subgroup of $G$ defined over $F$. 
The goal of this subsection is to obtain a \emph{transfer result} for functions on $J_b$ and $G$ when $n<p$; \emph{cf.} Cor.~\ref{cor-Red-transfer}. We proceed in two steps: transfer from $J_b$ to $M_{b}$, and from $M_{b}$ to $G$.

Given $f\in C^{\infty}_c(J_b;\CC)$, Badulescu \cite[Thm.~3.2]{Badulescu:Transfer} showed that there exists a \emph{transfer} $f^\ast \in C^{\infty}_c(M_b;\CC)$ of $f$ in the following sense: For any $g\in J_b^{\rss}$ and $g^\ast \in M_{b}^{\rss}$ with matching conjugacy classes, we have
\begin{equation}\label{eq-LStransfer}
O_{g^\ast }^{M_b}(f^\ast ) = e(J_b)\cdot O_{g}^{J_b}(f),
\end{equation}
where the orbital integrals are computed with respect to compatibly chosen Haar measures on $Z_{J_b}(g)$ and $Z_{M_{b}}(g^\ast )$, which are isomorphic.

Under some additional hypothesis enabling us to apply Thm.~\ref{th-characterisation-orb-int} and Prop.~\ref{prop-DKV-a4h}, we shall extend \eqref{eq-LStransfer} for conjugacy classes that are not necessarily regular semisimple. Let us fix an isomorphism $(J_b)_{F^{\ur}} \riso (M_b)_{F^{\ur}}$ of reductive groups over $F^{\ur}$; it clearly exists if $b$ is decent (which we assumed from \S\ref{notation-Hecke-action} on), and the case for general $b$ can easily be reduced to the decent case. 

\begin{thm}\label{th-transfer}
Assume that $M_b$ is the product of $H^\ast_j\coloneqq\GL_{r_j}(F)$ with each $r_j<p$. 
\begin{enumerate}
    \item\label{th-transfer-conj} For any $g\in J_b$ there exists an element $g^\ast \in M_b$ such that $g$ and $g^\ast$ are conjugate in $M_b(F^{ur})$. Furthermore, $[g]\mapsto [g^\ast]$ defines an injective map from the set of conjugacy classes in $J_b$ to the set of conjugacy classes in $M_b$.
    \item\label{th-transfer-ftn} If $n_j<p$ for any $j$ then the transfer identity \eqref{eq-LStransfer} holds for any $g\in J_b$ and $g^\ast \in M_b$ that are conjugate in $M_b(F^{\ur})$.
\end{enumerate}
\end{thm}
\begin{proof}
To prove the theorem, we may assume that $M_b = \GL_r(F)$.

For (\ref{th-transfer-conj}), recall that given any field $F'$ and any inner form $H/F'$ of $\GL_r/F'$  a conjugacy class $[g]$ in $H$ is determined by its \emph{elementary divisors} -- a sequence of monic polynomials  $(q_1,\dotsc,q_r)$ over $F'$ with $q_{i+1}|q_i$ such that $\prod_{i=1}^rq_i$ coincides with the characteristic polynomial of $g$. See \cite[2.c]{DeligneKazhdanVigneras:CenSimAlg} for the proof, which is essentially by the classification of finitely generated modules over PID and the Galois descent. From the construction it easily follows that the formation of elementary divisors commutes with the scalar extension by any separable extensions of $F$, and that the elementary divisors coming from $H$ forms a subset of those coming from $\GL_r(F)$. Now claim (\ref{th-transfer-conj}) follows.

To prove (\ref{th-transfer-ftn}) we fix $f\in C^\infty_c(J_b;\CC)$ and set $W(g) = O^{J_b}_g(f)$ for any $g\in J_b$. We now define a function $W^\ast$ on $M_b$ as follows: for any $g'\in M_b$, we set
\[
W^\ast(g')\coloneqq \left\{
\begin{array}{ll}
    e(J_b)W(g) & \text{if }\exists\, g\in J_b\text{ such that }g' \text{ is conjugate to }g \text{ in }M_b(F^{\ur}); \\
    0 & \text{otherwise.}
\end{array}
\right.\]

Recall that for any conjugacy class in $J_b$ there is a unique standard pair $(T,u)$ such that $T^{\reg}\cdot u$ is contained in the conjugacy class, and the same holds for $M_b$. Furthermore, the standard pairs for $g\in J_b$ and $g^\ast\in M_b$ are isomorphic. Thus Thm.~\ref{th-characterisation-orb-int} can be applied to $W^\ast$, and we obtain $f'\in C^\infty(M_b;\CC)$ such that we have $W^\ast(g') = O^{M_b}_{g'}(f')$ for any $g'\in M_b$. Now the transfer identity \eqref{eq-LStransfer} implies $W^{\ast}(g') = O^{M_b}_{g'}(f^\ast)$ for any $g'\in M_b^{\rss}$, so by Prop.~\ref{prop-DKV-a4h} the orbital integrals of $f'$ and $f^\ast$ coincide, which concludes the proof.
\end{proof}

We write $f\mapsto f^\ast $ if $f^\ast $ is a transfer of $f$. Although $f^\ast $ is not uniquely determined by \eqref{eq-LStransfer}, any $f^\ast $'s have the same trace on any irreducible representation of $M_{b}$ by Prop.~\ref{prop-DKV-a4h}, which is all that matters.

Now the following lemma is immediate from \eqref{eq-char-dist} and \eqref{eq-LJ-char-id}.
\begin{lem}\label{lem-transfer-LJ}
The surjective group homomorphism $\LJ^b\colon \Groth (M_{b})\to\Groth(J_b)$ is uniquely characterised by the following character identity
\[\tr (f^\ast\mid \pi^\ast) = \tr (f\mid  \LJ^b(\pi^\ast))\]
for any $f$ and $f^\ast $ with $f\mapsto f^\ast $.
\end{lem}

We now move on to a transfer result for $G$ and $M_b$. Let $\nu\coloneqq \nu_G(b)$ denote the Newton cocharacter of $b$, and let $P_b$ be the parabolic subgroup where the weights of $\nu$ on its Lie algebra are non-positive. We fix Haar measures on $G$ and $M_b$.
\begin{prop}[\emph{Cf.} {\cite[Lem.~V.5.2]{HarrisTaylor:TheBook}, \cite[Lem.~3.9]{ShinSW:StableTraceIgusa}, \cite[Lem.~3.1.2]{KretShin:H0Igusa}}]
\label{prop-transfer-levi}
Suppose that $G=\GL_n(F)$ with $n<p$, and $f\in C_c^{\infty}(M_b;\CC)$ is supported on $\nu$-acceptable elements. Then there exists a function $\tilde f \in C_c^\infty(G;\CC)$ that satisfies the following properties:
\begin{enumerate}
\item\label{prop-transfer-levi-vanishing}  If $g\in G$ is a semisimple element that is not $G$-conjugate to any $\nu$-acceptable element in $M_b$, then $O_g^G(\tilde f) =0$.
\item\label{prop-transfer-levi-orbital-int} If $g\in G$ is a  semisimple element that is $G$-conjugate to a $\nu$-acceptable element $m\in M_b$, then we have
\[O_g^G(\tilde f) =  O_m^{M_b}(f),\]
where the orbital integrals are computed with respect to compatibly chosen Haar measures on $Z_{M_b}(m)$ and $Z_{G}(g)$, which are isomorphic.
\item\label{prop-transfer-levi-trace}  For any smooth irreducible representation $\pi$ of $G$ we have
\[\tr (\tilde f\mid \pi) = \tr (f\mid \Jac^{G}_{P_b}(\pi)).\]
\end{enumerate}
\end{prop}
\begin{proof}
The proof is essentially identical to the proof of \cite[Lem.~V.5.2]{HarrisTaylor:TheBook}, as explained in \cite[Lem.~3.9]{ShinSW:StableTraceIgusa} (which contains a minor sign error corrected in \cite[Lem.~3.1.2]{KretShin:H0Igusa}). Note that our sign convention for $P_{b}$ is opposite to the aforementioned references, and we have normalised our $\tilde f$ as in \cite[Lem.~V.5.2]{HarrisTaylor:TheBook} to hide the square root of the modulus character.

To explain, define the following function on $G$:
\[
W(g) \coloneqq \left \{
\begin{array}{ll}
O_m^{M_{b}}(f)& \text{if $g$ is $G$-conjugate to a $\nu$-acceptable element }m;\\
0&\text{otherwise}.
\end{array}
\right .
\]
As $f$ is supported on $\nu$-acceptable elements, it follows from Cor.~\ref{cor-centralising} that $W$ is invariant under $G$-conjugation. Since $G=\GL_n(F)$ with $n<p$, we may verify the characteristation of orbital integrals (i.e., Thm.~\ref{th-characterisation-orb-int}) following the same proof as in \cite[Lem.~V.5.2]{HarrisTaylor:TheBook}. Note that the statement of Thm.~\ref{th-characterisation-orb-int} is identical to the case of characteristic~$0$ obtained by Vign\'eras \cite[1.n]{Vigneras:OrbInt}, and all the tools employed in \cite[Lem.~V.5.2]{HarrisTaylor:TheBook} work in characteristic~$p$ as well, including \cite[Lem.~2.5]{Vigneras:OrbInt} built upon \cite[p~52, Lem.~19]{HarishChandra:LNM162}. This proves claims \eqref{prop-transfer-levi-vanishing} and \eqref{prop-transfer-levi-orbital-int}.

To prove \eqref{prop-transfer-levi-trace}, we may repeat the argument in \cite[pp~189--190]{HarrisTaylor:TheBook}, noting that the trace distribution is locally integrable even in positive characteristics so the argument via the Weyl integration formula remains valid. 
\end{proof}

Now we obtain the following transfer results:
\begin{cor}\label{cor-Red-transfer}
Assume that $f\in C^{\infty}_c(J_b;\CC)$ is supported on $\nu$-acceptable elements. Then there exists a function $\tilde f^\ast \in C^\infty_c(G;\CC)$ that satisfies the following properties:
\begin{enumerate}
\item\label{cor-transfer-vanishing} If $g\in G$ is a semisimple element that is not $G$-conjugate to any $\nu$-acceptable element in $M_b$, then $O_g^G(\tilde f^\ast ) =0$.
\item\label{cor-transfer-orb-int} Assume that $g\in G$ is a regular semisimple element that is $G$-conjugate to a $\nu$-acceptable element $m^\ast \in M_b$.  If there exists $m\in J_b$ such that $m$ and $m^\ast $ have matching conjugacy classes, then we have
\[O_g^G(\tilde f^\ast ) = e(J_b) \cdot O_m^{J_b}(f)\]
where the orbital integrals are computed with respect to compatibly chosen Haar measures on $Z_{J_b}(\gamma)$ and $Z_{G}(g)$, which are isomorphic. 

If there does not exist  $m\in J_b$ as above, then we have $O_g^G(\tilde f^\ast ) =0$.
\item\label{cor-transfer-trace} For any smooth irreducible representation $\pi$ of $G$ we have
\[\tr(\tilde f^\ast \mid  \pi) = \tr (f\mid \Red^b(\pi)).\]
\end{enumerate}
If the multiplicity of each slope of $b$ is less than $p$, then \eqref{cor-transfer-orb-int} holds even when $g,m,m^\ast $ are semisimple and not necessarily regular.
\end{cor}

\begin{proof}
Choose a transfer $f^\ast \in C^\infty_c(M_b,\CC)$ of $f$, and find $\tilde f^\ast \in C^\infty_c(G;\CC)$ so that $f^\ast $ and $\tilde f^\ast $ have ``matching orbital integral'' in the sense of Prop.~\ref{prop-transfer-levi}\eqref{prop-transfer-levi-orbital-int}. Then the corollary follows from Lem.~\ref{lem-transfer-LJ}, Prop.~\ref{prop-transfer-levi} and Thm.~\ref{th-transfer}.
\end{proof}

\subsection{Proof of the main result} \label{ssect-divisionalg-proof}

Note that both the Igusa variety trace formula (Thm.~\ref{th-point-counting-elliptic}) and the transfer result (Cor.~\ref{cor-Red-transfer}) can be applied only after replacing $\varphi\in\Hcal$ with $\varphi^{(s)}$ for sufficiently divisible positive integer $s$. The following lemma asserts that such a restriction is harmless in determining representations of $\JJ_\bbf$.

 \begin{lem}\label{lem-acceptable-trace}
Let $\Pi_1,\Pi_2$ be admissible $\JJ_\bbf$-representations. Assume that given any $\varphi\in C^{\infty}_c(\JJ_\bbf)$ there exists a positive integer $s$ depending on $\varphi$ such that we have $\tr (\varphi^{(s)}\mid \Pi_1) = \tr(\varphi^{(s)}\mid \Pi_2)$. Then we have an equality $\tr (\varphi\mid \Pi_1) = \tr(\varphi\mid \Pi_2)$ for any $\varphi\in C^\infty_c(\JJ_\bbf;\CC)$.
 \end{lem}
 \begin{proof}
 Since $\Pi_i(\varphi)$ is non-zero only on finitely many irreducible constituents of $\Pi_i$, we may assume that both $\Pi_1$ and $\Pi_2$ are of finite length. Then we may repeat the proof of \cite[Lem.~6.4]{ShinSW:IgusaPointCounting}.
 \end{proof} 

\begin{proof}[Proof of Thm.~\ref{th-second-basic-id}]
Assume that $\Gsf$ is a totally anisotropic inner form of $\Gsf^\ast\coloneqq\GL_n$ with $n<p$. We want to show that the following equality holds for any $\varphi\in\Hcal$:
\[
\sum_i(-1)^i\tr\big(\varphi\mid \coh i_c(\Ig^{\bbf}_{\Gsf,\xbf,\Xi},\cl\QQ_\ell)\big) =\big(\prod_{x\in\xbf}e(J_{b_x})\big)\cdot\sum_{\pi\subset \Ascr(\Gsf(\Fsf)\backslash\Gsf(\AA)/\Xi)} \tr \big(\varphi\  | \Red^\bbf(\pi)\big);
\]
We may assume that $\varphi = \varphi^{\xbf}\cdot \prod_{x\in\xbf}f_x$ for $f_x\in \Hcal(J_{b_x})$. Furthermore, by Lem.~\ref{lem-acceptable-trace}, it suffices to show the equality after replacing $\varphi$ with $\varphi^{(s)}$ for some sufficiently divisible $s$, and thus we may assume that we have the following point-counting formula 
\[
\tr(\varphi\mid \Rrm\Gamma_c(\Ig^{\bbf}_{\Gsf,\xbf,\Xi},\cl\QQ_\ell)) =
 \sum_{(\gamma^\ast_0;\ggamma,\ddelta)\in\KS^{\ael}_\bbf} i(\Gsf^\ast_{\gamma^\ast_0};\Xi)\cdot O^{\JJ_{\bbf}}_{(\ggamma,\ddelta)}(\varphi);
\]
\emph{cf.} Thm.~\ref{th-point-counting-elliptic}.

Now choose $\tilde f_x^\ast \in\Hcal(\Gsf(\Fsf_{x}))$ for $f_x$ as in Cor.~\ref{cor-Red-transfer}, and set $\tilde\varphi^\ast \coloneqq \varphi^{\xbf}\cdot\prod_{x\in\xbf}\tilde f^\ast _x$. We will proceed by comparing the above point-counting formula for the Igusa variety with the Arthur--Selberg trace formula for $\Ascr (\Gsf(\Fsf)\backslash\Gsf(\AA)/\Xi)$ \cite[(13.5)]{Laumon-Rapoport-Stuhler}, which we recall below:
\[
\sum_{\pi\subset \Ascr(\Gsf(\Fsf)\backslash\Gsf(\AA)/\Xi)} \tr \big(\tilde\varphi^\ast \  | \pi\big) 
 = \sum_{\gamma_0 \in \Gsf_{\natural}} i(\Gsf_{\gamma_0};\Xi)\cdot O^{\Gsf(\AA)}_{\gamma_0}(\tilde\varphi^\ast ), 
\]
where $\Gsf_{\natural}$ is the set of representatives of conjugacy classes in $\Gsf(\Fsf)$. 
Recall that we now have the  multiplicity one theorem for $\Gsf=\Dsf^{\times}$; \emph{cf.} \cite[Thm.~3.3]{BadulescuRoche:JLC}.

Let us start with the geometric side. 
By Cor.~\ref{cor-Red-transfer}\eqref{cor-transfer-vanishing}, for the orbital integral $O^{\Gsf(\AA)}_{\gamma_0}(\tilde\varphi^\ast )$ to be non-zero we should have that $\gamma_0$ is stably conjugate to some element in $J_{b_x}(\Fsf_{x})$ for each $x\in\xbf$; i.e., $\gamma_0$ is  stably conjugate to $\gamma_0^\ast\in\Gsf^\ast(\Fsf)$ that appears in a Kottwitz--Igusa triple $(\gamma^\ast_0;\ggamma,\ddelta)\in\KS_{\bbf}$. Note also that $\gamma_0\in \Gsf(\Fsf)$ uniquely determines such a Kottwitz--Igusa triple up to equivalence; \emph{cf.} Cor.~\ref{cor-transfer-EL}. Therefore, the geometric side of the Arthur--Selberg trace formula can be written as the sum over $\KS_\bbf$ instead of $\Gsf_{\natural}$. Now applying Cor.~\ref{cor-Red-transfer}\eqref{cor-transfer-orb-int} and Lem.~\ref{lem-index-inner-twisting}, we obtain
\begin{align}\label{eq-geom-side}
\sum_{\gamma_0 \in \Gsf_{\natural}} i(\Gsf_{\gamma_0};\Xi)\cdot O^{\Gsf(\AA)}_{\gamma_0}(\tilde\varphi^\ast ) 
&= \big(\prod_{x\in\xbf}e(J_{b_x})\big)\cdot\sum_{(\gamma_0;\ggamma,\ddelta) \in \KS_\bbf} i(\Gsf_{\gamma_0};\Xi)\cdot O^{\JJ_\bbf}_{(\ggamma,\ddelta)}(\varphi)\\
&=\big(\prod_{x\in\xbf}e(J_{b_x})\big)\cdot\tr(\varphi\mid \Rrm\Gamma_c(\Ig^{\bbf}_{\Gsf,\xbf,\Xi},\cl\QQ_\ell)). \notag
\end{align}
To deduce the second equality from Thm.~\ref{th-point-counting-elliptic}, observe that  $\left |(\pi_1(\Gsf^\ast_{\gamma^\ast_0})_{\Gamma})_{\tor}\right |= 1$ for any elliptic semisimple element $\gamma^\ast_0\in\Gsf^\ast(\Fsf) = \GL_n(\Fsf)$; in fact, $\Gsf^\ast_{\gamma^\ast_0}$ is an inner form of $\Res_{\Fsf'/\Fsf}\GL_r$ for some separable extension $\Fsf'/\Fsf$, so we have $\pi_1(\Gsf_{\gamma_0})_{\Gamma} \cong\ZZ$, which is torsion-free.

Now let us move on to the spectral side. Given an irreducible constituent $\pi = \pi^{\xbf}\otimes\big(\bigotimes_{x\in\xbf}\pi_x\big)$ of $\Ascr(\Gsf(\Fsf)\backslash\Gsf(\AA)/\Xi)$, Cor.~\ref{cor-Red-transfer}\eqref{cor-transfer-trace} implies the following:
\begin{align}\label{eq-spec-side}
\tr (\tilde \varphi^\ast \mid \pi) &= \tr(\varphi^{\xbf}\mid \pi^{\xbf})\cdot\prod_{x\in\xbf}\tr(\tilde f_x^\ast  \mid \pi_x)\\
&=\tr(\varphi^{\xbf}\mid \pi^{\xbf})\cdot\prod_{x\in\xbf}\tr(f_x\mid \Red^{b_x}(\pi_x)) = \tr (\varphi\mid \Red^{\bbf}(\pi)).\notag
\end{align}
Comparing \eqref{eq-geom-side} and \eqref{eq-spec-side}, we obtain the desired equality.
\end{proof}

\begin{rmk}\label{rmk-unitary-groups}
Let $\Gsf$ be a unitary group over $\Fsf$ such that there is a place $w$ of $\Fsf$ where $\Gsf$ is totally anisotropic modulo centre. Such a unitary group exists by Landherr's theorem; \emph{cf.} \cite[Chap~10,~Thm.~2.4]{Scharlau:Forms}. Then for any finite set $\xbf$ of unramified places for $\Gsf$, the Igusa variety $\Ig^\bbf_{\Gsf,\xbf}$ is quasi-compact and any relevant $\Jsf_b$ is totally anisotropic, both of which can be seen from embedding $\Gsf$ into $\Res_{\Esf/\Fsf}\Dsf^\times$ for some central division algebra $\Dsf$ over a quadratic separable extension $\Esf/\Fsf$. Since any unitary group admits a quasi-split pure inner form, Thm.~\ref{th-point-counting-elliptic} can be applied.

If we also choose $\xbf$ so that $\Gsf$ is split at each $x\in\xbf$, then the local harmonic analysis reviewed in this section can be applied to $\Gsf(\Fsf_{x})\cong \GL_n(\Fsf_{x})$, and we have a weaker version of the transfer result (\emph{cf.} Prop.~\ref{prop-elliptic-transfer}).

Therefore, it should be possible to adapt the proof of Thm.~\ref{th-second-basic-id} in the setting when $\Gsf$ is a unitary group of rank less than $p$ satisfying the above conditions. To state the result precisely, one should take into account the fact that the multiplicity one result may not hold and $\pi_1(\Gsf^\ast_{\gamma^\ast_0})_\Gamma$ may not be torsion-free. 
\end{rmk}

\section{Other cases}\label{sect-conjectures}

In sections~\ref{sect-discrete} and \ref{sect-KS} we calculated the elliptic term of the trace formula. The most serious obstacle for relating it to the cohomology of Igusa varieties is that they are in general not quasi-compact unless $\Gsf$ is totally anisotropic modulo centre. So the compact support cohomology of the Igusa variety would not be an admissible $\JJ_\bbf$-representation  and $\tr(\varphi^{(s)}\mid \Rrm\Gamma_c(\Ig^{\bbf}_{\Gsf,\xbf,\Xi},\cl\QQ_\ell))$ may not be well defined in general. Nonetheless, we expect that under some suitable conditions on $\varphi\in\Hcal$ the alternating sum of traces $\tr(\varphi^{(s)}|\Rrm\Gamma_c(\Ig^{\bbf}_{\Gsf,\xbf,\Xi},\cl\QQ_\ell))$ should be well defined and coincide with the elliptic term $\tr (\varphi^{(s)})^{\el}$. 

After reviewing some backgrounds, we formulate the precise conjecture under the simplifying assumption that $\Gsf$ is a connected simply connected semisimple algebraic group over $\Fsf$; \emph{cf.} Conj.~\ref{conj-trace-formula}. We then discuss a few simple cases where one can say more about the conjecture in \S\ref{known-cases-conj-trace-formula}.

\subsection{Background}
To formulate the conjecture, let us give a brief review of Euler--Poincar\'e functions and matrix coefficients of supercuspidal representations.
Let $F=\Fsf_{v}$ and $G \coloneqq \Gsf_{\Fsf_{w'}}$. We fix a Haar measure on $G(F)$ so that $C^{\infty}_c(G(F))$ acts on any smooth representation of $G(F)$ via convolution. Let $\pi$ be an irreducible supercuspidal representation of $G(F)$. As $G$ is semisimple, $\pi$ is unitarisable; i.e., there exists a $G(F)$-invariant hermitian inner product $\langle,\rangle$ on $\pi$ that is sesquilinear on the \emph{first} variable. Choosing such $\langle,\rangle$ on $\pi$, we can make the following definition.

\begin{defn}\label{df-coeff}
A \emph{matrix coefficient} $\varphi_\pi$ of a unitary representation $\pi$ is a function defined as below for some $u,u'\in\pi$:
\[
\varphi_\pi\colon G(F)\to\CC; \quad g\mapsto \langle \pi(g) u, u'\rangle.
\]
\end{defn}
We follow the definition and convention of \cite[A.3.a]{DeligneKazhdanVigneras:CenSimAlg}, which is the complex conjugate of the usual definition of matrix coefficients for not necessarily unitary representations.

\begin{lem}\label{lem-selberg-princ}
Any matrix coefficient $\varphi_\pi$ of a discrete series representation $\pi$ is supercuspidal in the sense of \cite[A.I.b]{DeligneKazhdanVigneras:CenSimAlg}; namely, for any unipotent radical $N$ of some proper parabolic subgroup of $G$, we have 
\begin{equation}\label{eq-supercusp-form}
\int_{N(F)}\varphi_{\pi}(gnh)dn = 0\quad \forall g ,h\in G(F).
\end{equation}
\end{lem}
\begin{proof}
This is essentially Harish-Chandra's \emph{Selberg Principle}, which asserts that any matrix coefficient of a discrete series representation is cuspidal (i.e., any constant term vanishes); \emph{cf.} \cite[Cor.~4.4.7]{Silberger:Book}.
Indeed, the integral in \eqref{eq-supercusp-form} is a constant term of $L(g^{-1})(\varphi_\pi)$ where $L(g^{-1})$ denotes the left regular action, and $L(g^{-1})(\varphi_\pi)$ is itself a matrix coefficient of $\pi$.
\end{proof}

For $f\in C^\infty_c(G(F))$ and $g\in G(F)$, let $R(g)(f)$ denote the right regular action of $g$ on $f$. For any $\varphi,f\in C^{\infty}_c(G(F))$ one can also define the convolution action 
\[
R(\varphi)(f)\colon g\mapsto \int_{G(F)} \varphi(h)f(gh)dh
\] 
of $\varphi$ on $f$.
\begin{cor}\label{cor-selberg-princ}
Let $\varphi_\pi$ be a matrix coefficient of $\pi$. Then  $R(\varphi_\pi)(f)$ is \emph{cuspidal} for any $f\in C^\infty_c(G(F))$; in other words, for any unipotent radical $N$ of some proper parabolic subgroup of $G$, we have 
\begin{equation}\label{eq-cusp-form}
\int_{N(F)}(R(\varphi_\pi)(f))(ng)dn = 0\quad \forall g\in G(F).
\end{equation}
\end{cor}

\begin{rmk}
As another application of Lem.~\ref{lem-selberg-princ}, we have the vanishing of regular non-elliptic semisimple orbital integral for $\varphi_\pi$; \emph{cf.} \cite[A.3.e]{DeligneKazhdanVigneras:CenSimAlg}. 
\end{rmk}

Lastly, let us recall the following spectral property of matrix coefficients of supercuspidal representations, proved in \cite[A.3.g]{DeligneKazhdanVigneras:CenSimAlg}.
\begin{prop}\label{prop-Schur-ortho}
Let $\pi$ be an irreducible supercuspidal representation of $G(F)$, and choose a matrix coefficient $\varphi_\pi$ of $\pi$ such that $\varphi_\pi(1)$ is  non-zero. Then for any irreducible representation $\pi'$ of $G(F)$, we have $\pi'(\varphi_\pi)$ is non-zero if and only if $\pi\cong \pi'$. Furthermore, $\tr \pi(\varphi_\pi)$ is non-zero.
\end{prop}
The vanishing of $\pi'(\varphi_\pi)$ for $\pi'\not\cong\pi$ can be deduced from a suitable analogue of Schur orthogonality of matrix coefficients. Also, the result in \emph{loc.~cit.} shows that $\tr \pi(\varphi_\pi)=1$  if we suitably arrange the value of $\varphi_\pi$  at $1$. Hence, a matrix coefficient $\varphi_\pi$ normalised this way is an example of \emph{pseudo-coefficient} of discrete series as introduced in \cite[A.4]{DeligneKazhdanVigneras:CenSimAlg}.

 We deduce that the action of $\varphi$ on $\coh{i}_c(\Ig^{\bbf}_{\Gsf,\xbf},\cl\QQ_\ell)$ annihilates  $\JJ_\bbf$-subrepresentations where any irreducible constituent of $\pi_{w'}$ does not appear as a local factor at $w'$. In particular, $\varphi$ annihilates any irreducible constituents that are non-supercuspidal at $w'$. We expect that there is an Igusa variety analogue of C.~Xue's finiteness result on cuspidal cohomology  \cite{Xue:Cuspidal}; namely, the maximal $\JJ_\bbf$-subrepresentation of $\coh{i}_c(\Ig^{\bbf}_{\Gsf,\xbf},\cl\QQ_\ell)$ that is supercuspidal at $w'$ is admissible. 
 If this is the case then $\tr(\varphi^{(s)}\mid \Rrm\Gamma_c(\Ig^{\bbf}_{\Gsf,\xbf},\cl\QQ_\ell))$ would be well defined. Assuming the existence of a trace formula that relates $\tr(\varphi^{(s)}\mid \Rrm\Gamma_c(\Ig^{\bbf}_{\Gsf,\xbf},\cl\QQ_\ell))$ to local terms at fixed points, Conj.~\ref{conj-EP} below suggests that most of the local terms (i.e., orbital integrals) for test functions of the form $\varphi^w\varphi_{\EP,w}$ should vanish except the ``$w$-elliptic term''.
 
\begin{conj}[{\cite[(8.4)]{Gross:AJM2011}}]\label{conj-EP}
Unless $\gamma\in G(F)$ is elliptic semisimple, we have  $O^{G(F)}_{\gamma}(\varphi_{\EP})=0$.
\end{conj}
Implicit in the conjecture is that the orbital integral $O^{G(F)}_{\gamma}(\varphi_{\EP})$ should be well defined and equal to zero when the conjugacy class of $\gamma$ is not closed in $G(F)$ for the analytic topology. We can also allow $G$ to be any connected reductive group by letting $\varphi_{\EP}$ to be the inflation of the Euler--Poincar\'e function on $G^{\ad}(F)$.

We can view Conj.~\ref{conj-EP} as a positive characteristic analogue of a result of Kottwitz' \cite[\S2, Thm.~2]{Kottwitz:TamagawaNo}. Although the proof of \emph{loc.~cit.} only works over finite extensions of $\QQ_p$, the part of the argument handling \emph{semisimple} orbital integrals remains valid even in characteristic~$p$, so we have the vanishing of semisimple non-elliptic  orbital integral of $\varphi_{\EP}$. Note that a variant of Conj.~\ref{conj-EP} for $G=\GL_n$ is verified in any characteristic by Laumon \cite[Thm.~5.1.3]{laumonbook1}, where $\varphi_{\EP}$ is replaced with a so-called \emph{very cuspidal} Euler--Poincar\'e function \cite[(5.1.2)]{laumonbook1}.

We are now ready to state the conjecture on the cohomology of non quasi-compact Igusa varieties. For simplicity, we assume $\Gsf$ to be a simply connected semisimple algebraic group over $\Fsf$, and set $\Xi\coloneqq\{1\}$.

\begin{conj}\label{conj-trace-formula}

Choose a finite-length supercuspidal representation $\pi_{w'}$ of $\Gsf(\Fsf_{w'})$, and suppose that $\varphi\in\Hcal$ is of the form
\begin{equation}\label{eq-test-fcns}
\varphi = \varphi^{w,w'}\cdot \varphi_{\EP,w}\cdot \varphi_{\pi_{w'}},
\end{equation}
where $\varphi_{\EP,w}$ is the Euler--Poincar\'e function on $\Gsf(\Fsf_{w})$ defined by the same formula as the $p$-adic case (\emph{cf.} \cite[\S2]{Kottwitz:TamagawaNo}) and $\varphi_{\pi_{w'}}$ is  a matrix coefficient of $\pi_{w'}$ (\emph{cf.} Def.~\ref{df-coeff}) that is non-zero at the identity element. 

Then for any positive integer $s$, the action of $\varphi^{(s)}$ on $\coh{j}_c(\Ig^{\bbf}_{\Gsf,\xbf},\cl\QQ_\ell)$ is of \emph{finite rank} so $\tr(\varphi^{(s)}\mid \Rrm\Gamma_c(\Ig^{\bbf}_{\Gsf,\xbf},\cl\QQ_\ell))$ is well defined. Furthermore, if $s$ is sufficiently divisible, then we have 
\begin{equation}\label{eq-trace-formula}
\tr(\varphi^{(s)}\mid \Rrm\Gamma_c(\Ig^{\bbf}_{\Gsf,\xbf},\cl\QQ_\ell)) = \tr(\varphi^{(s)})^{\el(w)} 
\end{equation}
where $\tr(\varphi^{(s)})^{\el(w)}$ is the $w$-elliptic term \eqref{eq-ell-terms}.
\end{conj}
Note that the $s$-fold Frobenius-twist $\varphi^{(s)}$ of $\varphi$ does not modify the local component at $w$ and $w'$, so $\varphi^{(s)}$ is also of the form \eqref{eq-test-fcns}.

Combining the above conjecture with Thm.~\ref{th-point-counting-elliptic}, we obtain the formula
\[
\tr(\varphi^{(s)}\mid \Rrm\Gamma_c(\Ig^{\bbf}_{\Gsf,\xbf},\cl\QQ_\ell)) =
 \sum_{(\gamma_0;\ggamma,\ddelta)\in \KS^{\ael(w)}_\bbf}
i(\Gsf_{\gamma_0};\Xi)\cdot \left |(\pi_1(\Gsf_{\gamma_0})_{\Gamma})_{\tor}\right | \cdot 
O^{\JJ_{\bbf}}_{(\ggamma,\ddelta)}(\varphi^{(s)}).
\]

\begin{rmk}
 We assumed that $\Gsf$ is semisimple so that both $\varphi_{\EP,w}$ and $\varphi_{\pi_{w'}}$ are compactly supported, not just compactly supported modulo centre. To allow more general $\Gsf$, one has only \revise{to} consider representations and functions with fixed unitary central character, and test functions with the same central character. 
\end{rmk}

\subsection{Special cases of Conj.~\ref{conj-trace-formula}}\label{known-cases-conj-trace-formula}

 Let us first consider the case when $b_v = 1$ for all $v \in |C|$; i.e.,  the case with no leg. Then the cohomology of Igusa variety is nothing but $C^{\infty}_c(\Gsf(\Fsf)\backslash \Gsf(\AA),\cl\QQ_\ell)$, where $\Hcal = C^{\infty}_c(\Gsf(\AA))$ acts via right regular action. We deduce that the left hand side of \eqref{eq-trace-formula} is well-defined by the following lemma. 
\begin{lem}\label{lem-cusp-forms}
Given $\varphi\in\Hcal$ as in \eqref{eq-test-fcns}, the image of $R(\varphi)$ acting on $C^{\infty}_c(\Gsf(\Fsf)\backslash \Gsf(\AA))$ is contained in the subspace of \emph{cuspidal automorphic forms} on $\Gsf(\Fsf)\backslash \Gsf(\AA)$. In particular, $R(\varphi)$ is of finite rank.
\end{lem}
\begin{proof}
It follows from Cor.~\ref{cor-selberg-princ} that $R(\varphi)(f)$ is cuspidal for any $f\in C^{\infty}_c(\Gsf(\Fsf)\backslash \Gsf(\AA))$.
Now by \cite[Prop.~5.9]{BorelJacquet:Corvalis} it follows that any smooth cuspidal function on $\Gsf(\Fsf)\backslash\Gsf(\AA)$ is an automorphic form, as $\Zsf_\Gsf(\AA)$-finiteness is automatic by compactness of $\Zsf_\Gsf(\AA)$. This shows that $R(\varphi)(f)$ is a cuspidal automorphic form on $\Gsf(\Fsf)\backslash\Gsf(\AA)$ for any $f$.

Finally, by admissibility of the space of cuspidal automorphic forms as a $\Gsf(\AA)$-representation (\emph{cf.} \cite[Prop.~5.2]{BorelJacquet:Corvalis}), $R(\varphi)$ is of finite rank.
\end{proof}
Now we have $\varphi^{(s)}=\varphi$ for any $\varphi\in\Hcal$, and Conj.~\ref{conj-trace-formula} reduces to the \emph{simple trace formula}; \emph{cf.} Working Hypothesis II in \cite[\S8.2]{GanLomeli:GlobalizationSC}. A special case of this conjecture for more restrictive choice of $\varphi$ and $\pi_{w'}$ was formulated by Gross \cite[Conj.~5.2]{Gross:AJM2011}, and it is also asserted in the same paper that Conj.~\ref{conj-EP} implies the conjecture in \emph{loc.~cit}, by Ng\^o Dac's result (which generalises the work of Kottwitz in characteristic~$0$; \emph{cf.} \cite[\S2]{Kottwitz:TamagawaNo}).

 We now drop the assumption that $b_v = 1$ for all $v$ and assume that  $\Gsf$ is a totally anisotropic form of $\SL_n$, implying that $\Ig^{\bbf}_{\Gsf,\xbf}$  is quasi-compact. One can see this by embedding $\Gsf$ into $\Dsf^\times$ where $\Dsf$ is a central division algebra over $\Fsf$ or its quadratic separable extension. In that case, the following variant of Conj.~\ref{conj-trace-formula} holds unconditionally; namely, if $w$ is a place where $\Gsf$ splits and $\varphi_w\in C^\infty_c(\Gsf(\Fsf_w))$ is the very cuspidal Euler--Poincar\'e function \cite[(5.1.2)]{laumonbook1}, then for any $\varphi = \varphi^w\cdot\varphi_{w}\in C^\infty_c(\JJ_\bbf)$, we have 
\begin{equation}
\tr(\varphi^{(s)}\mid \Rrm\Gamma_c(\Ig^{\bbf}_{\Gsf,\xbf},\cl\QQ_\ell)) = \tr(\varphi^{(s)})^{\el(w)},
\end{equation}
for any sufficiently divisible $s$. Indeed, this formula is the refinement of the Lefschetz trace formula for  $\Ig^{\bbf}_{\Gsf,\xbf}$  (\emph{cf.} Cor.~\ref{cor-preliminary-point-counting-elliptic}) using  \cite[Thm.~5.1.3]{laumonbook1} which proves Conj.~\ref{conj-EP} for $\varphi_w$ in place of $\varphi_{\EP,w}$. We may allow $w$ to be a non-split place for $\Gsf$, provided that \emph{loc.~cit.} can be extended to a suitable variant of Euler--Poincar\'e function $\varphi_w$.

In the non quasi-compact case, the conjecture seems more tractable when $\Gsf$ is a inner (or even, outer) form of $\SL_n$. (Perhaps, its  variant for forms of $\GL_n$ might be more convenient.) We plan to explore this case in a future project.

\bibliographystyle{amsalpha}
\bibliography{bib}

\end{document}